\documentclass[10pt, reqno]{amsart}
\usepackage{a4wide}
\usepackage{wrapfig}
\usepackage[normalem]{ulem}
\newcommand{\stkout}[1]{\ifmmode\text{\sout{\ensuremath{#1}}}\else\sout{#1}\fi}
\usepackage{amsmath}
\usepackage{amsthm}
\usepackage{mathrsfs}
\usepackage{amssymb}
\usepackage{enumerate}
\usepackage{multicol}
\usepackage{eufrak}
\usepackage{color}
\usepackage[dvipsnames]{xcolor}
\usepackage{graphicx}
\usepackage{url}
\usepackage{fancyhdr}
\usepackage{tikz}
\usetikzlibrary{matrix}
\usetikzlibrary{arrows}
\usepackage{subfig}
\usepackage{hyperref}
\usepackage{varwidth}
\def\id{{\fontsize{.85em}{1.1em}\selectfont1}\normalfont\kern-.8ex1}
\numberwithin{equation}{section}

\newcommand\tenq[2][1]{%
 \def\useanchorwidth{T}%
  \ifnum#1>1%
    \stackunder[0pt]{\tenq[\numexpr#1-1\relax]{#2}}{\scriptscriptstyle\widetilde}%
  \else%
    \stackunder[1pt]{#2}{\scriptscriptstyle\widetilde}%
  \fi%
}

\usetikzlibrary{shapes.misc}

\newcommand{\strikeout}[1]{%
\ifmmode%
      \tikz[inner sep=0.5pt,baseline] \node [strike out,draw=OrangeRed,anchor=text]{$#1$};%
\else%
      \tikz[inner sep=0.5pt,baseline] \node [strike out,draw=OrangeRed,anchor=text]{#1};%
\fi%
}

\newtheorem{theorem}{Theorem}[section]
\newtheorem{proposition}[theorem]{Proposition}
\newtheorem{corollary}[theorem]{Corollary}
\newtheorem{lemma}[theorem]{Lemma}
\newtheorem{remark}[theorem]{Remark}
\newtheorem{definition}[theorem]{Definition}

\newcommand{\be}{\begin{equation}}
\newcommand{\ee}{\end{equation}}
\newcommand{\bp}{\begin{proof}}
\newcommand{\ep}{\end{proof}}
\newcommand{\bel}{\begin{equation}\label}
\newcommand{\eeq}{\end{equation}}
\newcommand{\bea}{\begin{eqnarray}}
\newcommand{\eea}{\end{eqnarray}}
\newcommand{\bee}{\begin{eqnarray*}}
\newcommand{\eee}{\end{eqnarray*}}
\newcommand{\ben}{\begin{enumerate}}
\newcommand{\een}{\end{enumerate}}

\newcommand{\p}{\partial}
\newcommand{\ds}{\displaystyle}

\newcommand{\zf}{\mathfrak{z}}
\newcommand{\cf}{\mathfrak{c}}

\newcommand{\cH}{{\mathcal H}}
\newcommand{\cN}{{\mathcal N}}

\newcommand{\cZ}{{\mathcal Z}}

\newcommand{\R}{{\mathbb R}}
\newcommand{\Z}{{\mathbb Z}}
\newcommand{\C}{{\mathbb C}}

\newcommand{\N}{{\mathbb N}}

\newcommand{\ka}{\kappa}

\newcommand{\al}{\alpha}

\newcommand{\si}{\sigma}

\newcommand{\arctanh}{\operatorname{arctanh}}

\newcommand{\sech}{\operatorname{sech}}
\newcommand{\arccotanh}{\operatorname{arccotanh}}

\newcommand{\re}{\operatorname{Re}}

\DeclareMathOperator{\sgn}{sgn}

\title[]{Orbital stability of the black soliton for the  quintic Gross-Pitaevskii equation}

\author[Miguel \'A. Alejo]{Miguel \'A. Alejo}
\address{Departamento de Matem\'aticas. Universidad de C\'ordoba\\
C\'ordoba, Spain.}
\email{malejo@uco.es}

\author[A.J. Corcho]{ Ad\'an J. Corcho}
\address{Departamento de Matem\'aticas. Universidad de C\'ordoba\\
C\'ordoba, Spain.}
\email{a.corcho@uco.es}

\thanks{M.A. Alejo and A.J. Corcho  thank the  department of Mathematics at UFSC, Brazil, and the 
Instituto de Matem\'atica, UFRJ, Brazil, for their support and hospitality during several stays there while this work was being done.}


\date{\today}

\begin{document}

\maketitle \markboth{Orbital stability of the black soliton - Quintic GP} {M.A. Alejo  and A.J. Corcho}

\setcounter{page}{1}

\begin{quote}
\textbf{Abstract.}
{\small In this work, a proof of the orbital stability of the black soliton solution of the quintic Gross-Pitaevskii 
    equation in one spatial dimension is obtained.  We first build and show explicitly  black and dark soliton solutions and we prove that the corresponding Ginzburg-Landau 
	energy is coercive around them  by  using some orthogonality conditions related  to perturbations of the black and dark solitons.  
	The existence of suitable perturbations around black and dark solitons satisfying the required  orthogonality conditions is deduced 
	from an Implicit Function Theorem. In fact, these perturbations involve dark solitons with sufficiently small speeds and some proportionality 
	factors ari\-sing from the explicit expression of their spatial derivative.  
}
\end{quote}
\tableofcontents

\section{Introduction}\label{Sec1}
In this work we consider the one-dimensional quintic Gross-Pitaevskii equation (quintic GP)
\be\label{5gp}
\begin{cases}
iu_t +u_{xx}=(|u|^4-1)u,\quad (t,x)\in \R^2, \medskip \\
u(0,x)=u_0(x),
\end{cases}
\ee
where $u$ is a complex-value function and the initial data $u_0$ satisfies the boundary condition

\be \label{5gp-boundary} 
\lim\limits_{|x|\to +\infty} |u_0(x)|^2 = 1.
\ee
From the physical point of view {\color{black}it} is interesting to look for solutions $u(t, x)$ of \eqref{5gp} satisfying  the boundary condition \eqref{5gp-boundary} for all $t\ge0$.

This is a \emph{defocusing} nonlinear Schr\"odinger equation modeling for example ultra-cold dilute Bose gases in highly
elongated traps. More specifically, it describes dynamics of weak density
modulations of one dimensional bosonic clouds (Tonks-Girardeau gases) when the tight transverse
confinement potential is turned off. In fact \eqref{5gp}, in the case of one dimensional atomic strings,
allows to explain many fermionic properties arising in one dimensional chains of bosons, phenomena usually named as  \emph{bosonic fermionization}. See 
\cite{Kevrek,Kolom,Lieb,Minguz} and references therein for a complete background on the physical phenomena accounted for by this quintic defocusing model.

\medskip 

The quintic GP equation is phase (also called $U(1)$ invariance) and translation invariant, meaning that if $u$ is a solution of \eqref{5gp}, then 

\[e^{i\theta}u(t, x + a),~ a\in\R,~\theta\in\R,
\]
\noindent
is also a solution of \eqref{5gp}. The quintic GP \eqref{5gp} also bears  Galilean invariance, namely
$$e^{i(
		\frac{c}{2}x - \frac{c^2}{4}t)}u(t, x - ct),~c\in\R,$$
\noindent
but this will not be used in our approach. Note moreover that in \eqref{5gp-boundary}
the  asymptotic value 1 can be changed to any number $\zeta>0$ without loss of generality by rescaling the values of $u$ through  
 $v=\zeta u(\zeta^4t,\,\zeta^2x)$. Under this change \eqref{5gp} recasts as 
\begin{equation*}
iv_t +v_{xx}=(|v|^4-\zeta ^4)v,\quad (t,x)\in \R^2.
\end{equation*}

Furthermore, and as far as we know, the quintic GP \eqref{5gp}-\eqref{5gp-boundary} 
is a non-integrable hamiltonian model (see \cite{Jia,Percy}), with well known low order conservation laws for \textit{regular solutions}, such as the mass 
\[
M[u](t):=\int_\R\big(1 - |u|^2\big)dx=M[u](0), 
\]
\noindent

\noindent
and the \emph{classical} energy 
\[E_1[u](t):=\int_\R\Big(|u_x|^2  - \tfrac{1}{3}(1- |u|^6)\Big)dx=E_1[u](0).\]
\noindent
In this work will be important the so called {\color{black}quintic}  {\it Ginzburg-Landau energy} given by

\[E_2[u]=E_1[u] + M[u],\]
\noindent
or explicitly
\begin{equation}\label{E2}
E_2[u](t):=\int_\R\Big(|u_x|^2  + \tfrac{1}{3}(1-|u|^2)^2(2+|u|^2)\Big)dx,
\end{equation}
which is also preserved along the flow. {\color{black}Another conserved quantity of \eqref{5gp} is the \emph{momentum}, which in the context of solutions verifying \eqref{5gp-boundary} can be suited in different forms. For example, for nonvanishing solutions (see \cite{KY}) is written as

\be\label{P1}
P_1[u](t):=\int_\R \langle iu,u_x\rangle_\C\left(1-\frac{1}{|u|^2}\right)dx.
\ee
\noindent
Moreover, considering vanishing solutions, in \cite{BGSS} it was introduced a renormalized version of the momentum \eqref{P1}, namely, (here $u(t,x)=A(t,x) e^{i\varphi(t,x)}$)

\be\label{P2}
P_2[u](t):=\lim_{R_1,R_2\rightarrow +\infty}\left(\frac12\int_{-R_1}^{R_2} \langle iu,u_x\rangle_\C dx-\frac12(\varphi(R_2)-\varphi(R_1))\right)\quad\text{mod}~\pi.
\ee
\noindent}

\medskip 
Here by \textit{regular solutions} we will understand those solutions that belong to the  energy space 
associated to \eqref{5gp}:
\begin{equation}\label{space-funtions}
\Sigma =\Big\{u\in H^1_{loc}(\R): \, u_x\in L^2(\R)\; \text{and}\; 1 -|u|^4\in L^2(\R)\Big\}.
\end{equation}

\medskip 
\noindent 
Notice that if $u\in \Sigma$,  then $1-|u|^2\in L^2(\R)$. Hence,
\be\label{sigma-well-defined}
(1-|u|^2)^2(2+|u|^2)=(1-|u|^2)^2 + (1-|u|^2)(1-|u|^4)\in L^1(\R), 
\ee
and $E_2[u]$ is well-defined.

\medskip
Some previous results on the Cauchy problem of \eqref{5gp} are well known in the literature.
For example, local well-posedness in the context of a Zhidkov space $\big\{ u\in L^{\infty}(\R),\; \p_xu\in L^2(\R)\big\}$ was shown in \cite{Zhidkov}
and  global well-posedness of \eqref{5gp}-\eqref{5gp-boundary} was established in \cite{Gallo}, 
where it was considered the general model  

\be\label{generalized-gp}
iu_t +u_{xx} + f(|u|^2)u=0,
\ee
with regular nonlinearity $f:\R^+ \rightarrow \R$ satisfying $f(1)=0$  and  $f'(1)<0$. \eqref{generalized-gp} includes, 
as particular cases, other important equations such as
\begin{itemize}
\item  Pure powers: $f(r) =1-r^p,\; p\in \Z^+$.

\item Cubic case ($p=1$): $f(r) =1-r,$  the cubic Gross-Pitaevskii (cubic GP) equation.

\item Cubic-Quintic case: $f(r)=(r-1)(2a +1 -3r)$ with $0< a<1$.

\item Quintic case ($p=2$): $f(r) =1-r^2,$  the quintic GP equation \eqref{5gp}.
\end{itemize}
More precisely, in {\cite[Theorem 1.1]{Gallo}} was proved that the Cauchy problem  for the quintic GP equation \eqref{5gp}-\eqref{5gp-boundary} 
is globally well-posed in the space 
\[\phi + H^1(\R),\]
for any $\phi$ verifying 
\[\phi \in C^2_b(\R),\quad \phi'\in H^2(\R), \quad |\phi|^2-1 \in L^2(\R).\]
\noindent
See \cite{BGS,Gerard} and \cite{Gerard2} for further reading on these generalized Schr\"odinger models.
 
\medskip
Concerning solutions, complex constants with modulus one are the simplest solutions contained in \eqref{5gp}. Moreover, and with 
respect to particular soliton solutions, and specifically to the stability of solitonic waves for \eqref{generalized-gp}, the situation is well 
understood in the case of the cubic GP equation, {\color{black} profiting} its integrable character  (\cite{Zak}). Indeed, it is well known that the black soliton of 
the cubic GP is
\[
\nu_0(x)=\tanh\left(\frac{x}{\sqrt{2}}\right),
\]
which is a stationary, i.e. time independent, wave solution.  Furthermore, the study of orbital and asymptotic stability for $\nu_0(x)$ 
was considered in several works 
\cite{BGSS, Gallay-Pelinovsky, GeZ, G-Smets}. Beside that, for some cases of the  cubic-quintic model  
($f(r)=(r-1)(2a +1 -3r)$), the stability of traveling solitonic bubbles was shown in \cite{Lin}. See \cite{Angulo, BGSS, PeliKivshar} for more details  
on these models. Finally, \cite{Chiron-2013} dealt with stability (and instability) problems for stationary and subsonic traveling waves 
giving an explicit condition on a general  $C^2$ nonlinearity $f$ in the NLS model. {\color{black}Once in this work we have obtained exact traveling wave solitons (also named as dark solitons), \cite[Theorem 24]{Chiron-2013} can be applied to study the orbital stability of stationary solutions (i.e. black solutions) but in another metric, well adapted to the $\Sigma$ space \eqref{space-funtions}, different from the metric used in the current work}.

\medskip 
In comparison with the cubic GP, the non-integrability of the quintic GP equation makes the search of solutions even harder as well as 
the rigorous study of the analytical properties related to them. Actually, and 
as far as we know, the \emph{black soliton} solution for the quintic GP was discovered in \cite[eqn.$(12)$]{Kolom}. Beside that, we present in this work the explicit expression of this solution  as well as its formal derivation 
(see Section \ref{Sec2}). Namely, the black soliton of  \eqref{5gp} is given by

\begin{equation}\label{black5gp}
\phi_0(x)=\sqrt{2}\frac{\tanh(x)}{\sqrt{3-\tanh^2(x)}},
\end{equation}
\noindent
which is a solution of 

\begin{equation}\label{edoblack1}
\phi'' + (1- \phi^4)\phi=0,
\end{equation}
\noindent
the corresponding differential equation describing stationary real waves of \eqref{5gp}  with $u(0,x)=\phi(x)$ (see Section \ref{Sec2} for further details). 
Therefore, it is natural to question whether, in the case of the quintic GP, the stability of $\phi_0$ is preserved in some sense. In fact, 
the main result of this work is the following (see Section \ref{Sec5} for a more detailed version and proof of this result):

\begin{theorem}\label{teorema1a}
The black soliton solution $\phi_0$ \eqref{black5gp} of the quintic GP equation \eqref{5gp} is orbitally stable in a subspace of the energy space $\Sigma$ 
\eqref{space-funtions}.
\end{theorem}

\medskip
The black soliton \eqref{black5gp}, stationary by nature, belongs to a greater family of traveling waves. As far as we know,  an explicit and correct
expression of a traveling wave family of solutions for the quintic GP \eqref{5gp} was missed. In fact, we show in this work that the quintic 
GP \eqref{5gp}-\eqref{5gp-boundary} also bears explicit traveling-wave solutions.  These waves, with the form
$$u(t,x)={\color{black}\Phi_c(x-ct)},$$
are currently known as \textit{dark solitons}, a reminiscent terminology coming from nonlinear optics (see \cite{Kivshar}). 
The function ${\color{black}\Phi_c}$ satisfies the complex nonlinear ordinary differential equation
\begin{equation}\label{edodark}
{\color{black}\Phi_c'' -ic\Phi_c' + (1-|\Phi_c|^4)\Phi_c= 0.}
\end{equation}
\noindent
Indeed, for $|c|< 2$ we are able to obtain the following explicit family of dark solitons:

\be\label{darksolitoninicial}
{\color{black}\Phi_c}(\xi)=\frac{i\mu_1(c) + \mu_2(c)\tanh(\ka(c)\xi)}
{\sqrt{2}\sqrt{1+\mu(c) \tanh^2(\ka(c)\xi)}},
\ee

\noindent
with $\xi=x-ct$ and where
\be\label{kappa}
\ka\equiv\ka(c)=\frac{\sqrt{4-c^2}}{2},
\ee

\be\label{mus}\begin{aligned}
&\mu_1\equiv\mu_1(c)=\frac{3c^2-4+2\sqrt{3c^2+4}}{\sqrt{18c^2-8+(3c^2+4)^{3/2}}},\\
&\mu_2\equiv\mu_2(c)=\frac{3c\sqrt{4-c^2}}{\sqrt{18c^2-8 +(3c^2+4)^{3/2}}},
\end{aligned}\ee
\noindent
and $\mu\equiv\mu(c)$ verifying the constraint relation

\be\label{murelation}
\frac{\mu_1^2 + \mu_2^2}{2 + 2\mu}=1,
\ee
for all $|c| <2$ and which comes from \eqref{5gp-boundary}. Therefore, $\mu$ is explicitly
\be\label{mus}\begin{aligned}
&\mu\equiv\mu(c)=\frac{3c^2+20-8\sqrt{4+3c^2}}{3(-4+c^2)}.
\end{aligned}\ee

\medskip 
\noindent
Note  that 

\[
\lim_{c\rightarrow0}\mu_1=0,\quad \lim_{c\rightarrow0^{\pm}}\mu_2=\pm\frac{2}{\sqrt{3}}\quad\text{and, from \eqref{murelation},}
\quad \lim_{c\rightarrow0}\mu=-\frac{1}{3}~~\text{with}~~-\frac{1}{3}\leq\mu\leq0.
\]

\noindent
Also notice that, as a consequence of the above limits, we get

\be\label{darkApprox}
\lim\limits_{c\to 0^{\pm}}{\color{black}\Phi_c(x)=\pm\Phi_0(x)}=\pm\sqrt{2}\frac{\tanh(x)}{\sqrt{3-\tanh^2(x)}}.
\ee
\noindent

{\color{black}Finally hereafter, since $\pm\Phi_c$ are both solutions of \eqref{edodark}, it is better to consider a $c-$smooth continuation of \eqref{black5gp} in the following way:

\be\label{darksoliton}
\phi_c=
\begin{cases}
\hspace{0.25cm} \Phi_c,\qquad c\geq 0,\\
-\Phi_c,\qquad c<0.
\end{cases}\ee
}

{\color{black}Hence, 
\be\label{limitDarktoBlack}\lim_{c\rightarrow0^\pm}\phi_c=\phi_0.\ee}

\noindent

\medskip

\begin{remark}
The main ingredient in the orbital stability proof is the use of the associated family of complex dark profiles $\phi_c$ whose 
real parts are odd functions with respect to the speed $c$ and are laterally approximated to the stationary black solution when ${\color{black}c\to 0}$, as 
shown in \eqref{darkApprox}. 
\end{remark}

\begin{figure}[h!] \centering
\includegraphics[scale=.44]{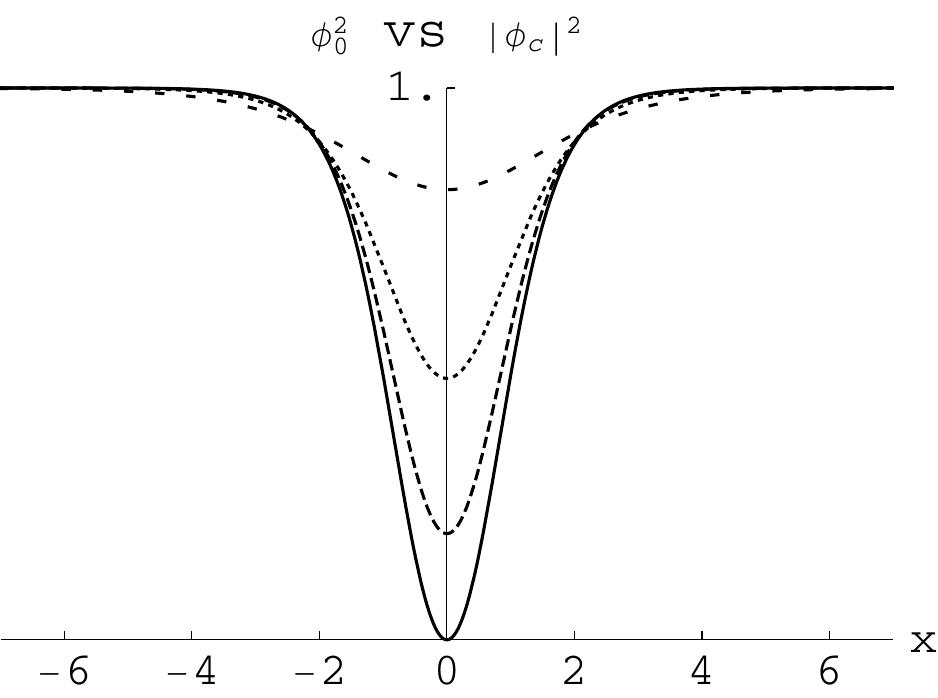} 
\includegraphics[scale=.46]{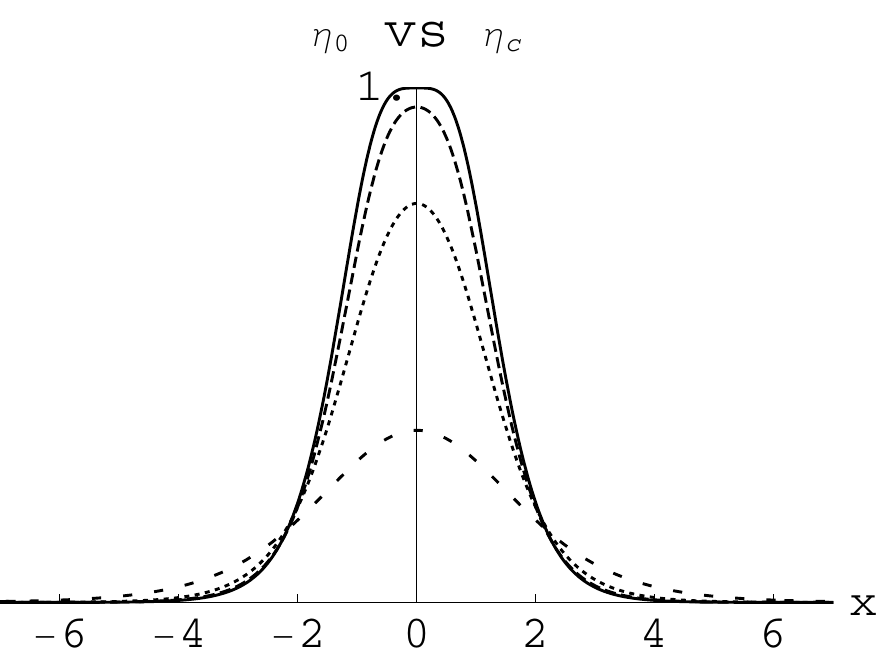} 
\caption{\label{fig1} {Left: graphics for the black soliton $\phi_0^2$ \eqref{black5gp} (full line) against several profiles of dark solitons $|\phi_c|^2$ 
\eqref{darksoliton} corresponding to dashed ($c=0.75$), dotted ($c=1.25$), and dashed but less segmented   ($c=1.75$). Right:
the nonlinear weight $\eta_0$  (full line) against several weights $\eta_c$ \eqref{peso-blackdark}, varying $c$ similarly.}}
\end{figure}

Our motivation to deal with the orbital stability 
of black solitons with this particular quintic nonlinearity {\color{black}(note that the cubic GP case was approached in \cite{G-Smets})} comes firstly from the  physical relevance that this model 
has in quantum gases as we said above. We were also motivated to prove this orbital stability result getting rid of a hydrodynamical formulation because, this approach  only describes non-vanishing solutions, and therefore excluding the black soliton solution. From a specific mathematical point of view, we avoid technical issues coming from the non integrable character of 
the model and therefore  not being allowed to use classical integrability methods. 

\medskip
\noindent

\medskip
\noindent
More specifically, our proof  establishes the coercivity of the functional $Q_c[\zf]$ \eqref{Qc}, when the function $\zf$ satisfies suitable orthogonality 
conditions well adapted to this specific quintic nonlinearity. These orthogonality conditions are guaranteed by the introduction of modulation
parameters (see Proposition \ref{prop2a}) and needed to perturb a stationary object as the black soliton of the quintic GP. This  approach has the advantage to show a better control on the perturbation with respect to the black soliton. 


\medskip
\noindent
Besides that, we are able to explicitly obtain (despite the nonintegrable nature of the model) dark solitons 
(traveling wave kinks) of the quintic GP, being one of the special cases where it is still possible to get 
these solutions (the other one is the cubic GP). The  knowledge about dark solitons of quintic 
GP allowed us to perform precise perturbations on the black soliton. 

\medskip
\noindent

Summarizing, we highlight here the main results involved in the proof of the orbital stability of the black soliton of 
the quintic GP \eqref{5gp}. Our proof introduces new theoretical and technical tools, 
with respect to the integrable cubic GP equation \cite{G-Smets} or more recently with respect to systems of cubic GP equations \cite{CPP}. 
These tools are specially suited to deal with the associated nonlinear solutions of \eqref{5gp}, namely 
the black soliton $\phi_0$ \eqref{black5gp} and the dark soliton $\phi_c$ profile \eqref{darkprofile}. Specifically we introduced

\medskip

\begin{itemize}
 \item a new family of traveling wave solutions $\phi_c$ \eqref{darksoliton},  close to the stationary black soliton $\phi_0$, 
 for  the quintic GP equation \eqref{5gp}. Obtaining non-constant solutions of this 
 non-integrable equation is not a simple task, and even more in the case when one has to solve a coupled nonlinear ODE system \eqref{edodark}. 
 Only by proposing a suitable ansatz and a careful tuning of the free parameters allowed us to  obtain them. 
 Just, compare these solutions of the quintic GP equation with the corresponding ones of the cubic GP equation, where a simple complex constant translation gives 
 the traveling family. See Section \ref{derivationdark}  for further reading.
 
\item a modified metric $d_c$.  This is a weighted metric with  nonlinear weight $\phi_c^3$ as it is dictated from the coercivity estimates that we need to prove on the 
{\color{black}quintic} Ginzburg-Landau energy on black and dark solitons. See  \eqref{dc} for a precise definition of $d_c$ and also  Propositions \ref{prop1} - \ref{anticor1}.

\item new functional spaces, in order to correctly measure the distance between black and dark solitons and their perturbations $\zf$. 
See Section \ref{Preliminaries} for details.

\item new orthogonality conditions associated to perturbations of the black and dark solitons 
\eqref{ortogonalityBlack} and \eqref{genortogonalityDark} and specially adapted to the spectral properties of the quintic GP equation. 


\end{itemize}

\medskip
In this work, we  were able to overcome several technical issues coming from the nonlinear functional structure of the quintic GP and its 
black and dark solitons, by working in a small speed region  {\color{black}$|c|<\cf$}. 
Moreover, the apparent structural difference between 
black solitons in the cubic GP and the quintic GP, is reflected in many identities and 
related functions around these black (and dark) solitons, e.g. the {\color{black}quintic} Ginzburg-Landau energy $E_2$ \eqref{E2} or the spatial derivative $\phi_0'$. 

\medskip
The strategy we used for the proof of the orbital stability result 
for the black soliton $\phi_0$ of \eqref{5gp} was focused to first show  that the {\color{black}quintic} Ginzburg-Landau energy 
$E_2$ is coercive around the black and dark solitons. This was done by using some orthogonality 
relations based on perturbations $\zf$ of the black and dark solitons and arising from the particular spectral problem related to \eqref{5gp}, suitable nonlinear identities 
and some proper Gagliardo-Nirenberg estimates on functions of the black and dark solitons of \eqref{5gp}. 

\medskip
We notice that the orthogonality conditions arising from the coercivity result (Proposition \ref{prop1}) on the black soliton $\phi_0$, 
do not include a linear term  appearing after expansion of $E_2$ \eqref{E2}  around the black and dark solitons, and therefore we must 
deal with this remaining linear term along the proof, estimating it in a suitable way to obtain the expected bounds, in contrast with previous approaches (\cite{G-Smets})  where their natural orthogonality conditions imposed its cancellation. 

\medskip
After that main step, we continued with	 Proposition \ref{prop2a}, proving, through a modulation of parameters,
the existence of suitable perturbations $\zf$ of the dark soliton which satisfy the orthogonality conditions 
defined in \eqref{genortogonalityDark}. 


\medskip 
Finally note that, related with the orbital stability, is the concept of asymptotic stability which essentially states the convergence of 
perturbations of the black soliton to a special element in the tubular neighborhood generated by its symmetries, e.g. phase and translation invariances. 
A detailed study on the asymptotic stability of the black soliton \eqref{black5gp} of the quintic GP \eqref{5gp} 
is currently being made and it will appear elsewhere.

\medskip
\subsection{ Final remarks}

\begin{itemize}
 
	\item {\color{black}  Our work does not get an orbital stability result for dark solitons \eqref{darksoliton} in $d_c$ metric \eqref{dc} for speeds close to $0$. However, this kind of stability for dark profiles can be obtained in an alternative metric as the used in \cite[Theorem 1.1]{Lin}. In fact, computing directly
	
	\[
	 P_1[\phi_c]=\frac{\mu_1\mu_2}{\sqrt{\mu}}\arctan(\sqrt{\mu})-2\arctan\left(\frac{\mu_2}{\mu_1}\right),
	\]
\noindent
we get 
	
	\[
	 \frac{dP_1}{dc}[\phi_c]<0,
	\]
	\noindent
	as it can also be seen in Figure \ref{figP1-a}
	
	\begin{figure}[h!] \centering
	\includegraphics[scale=.40]{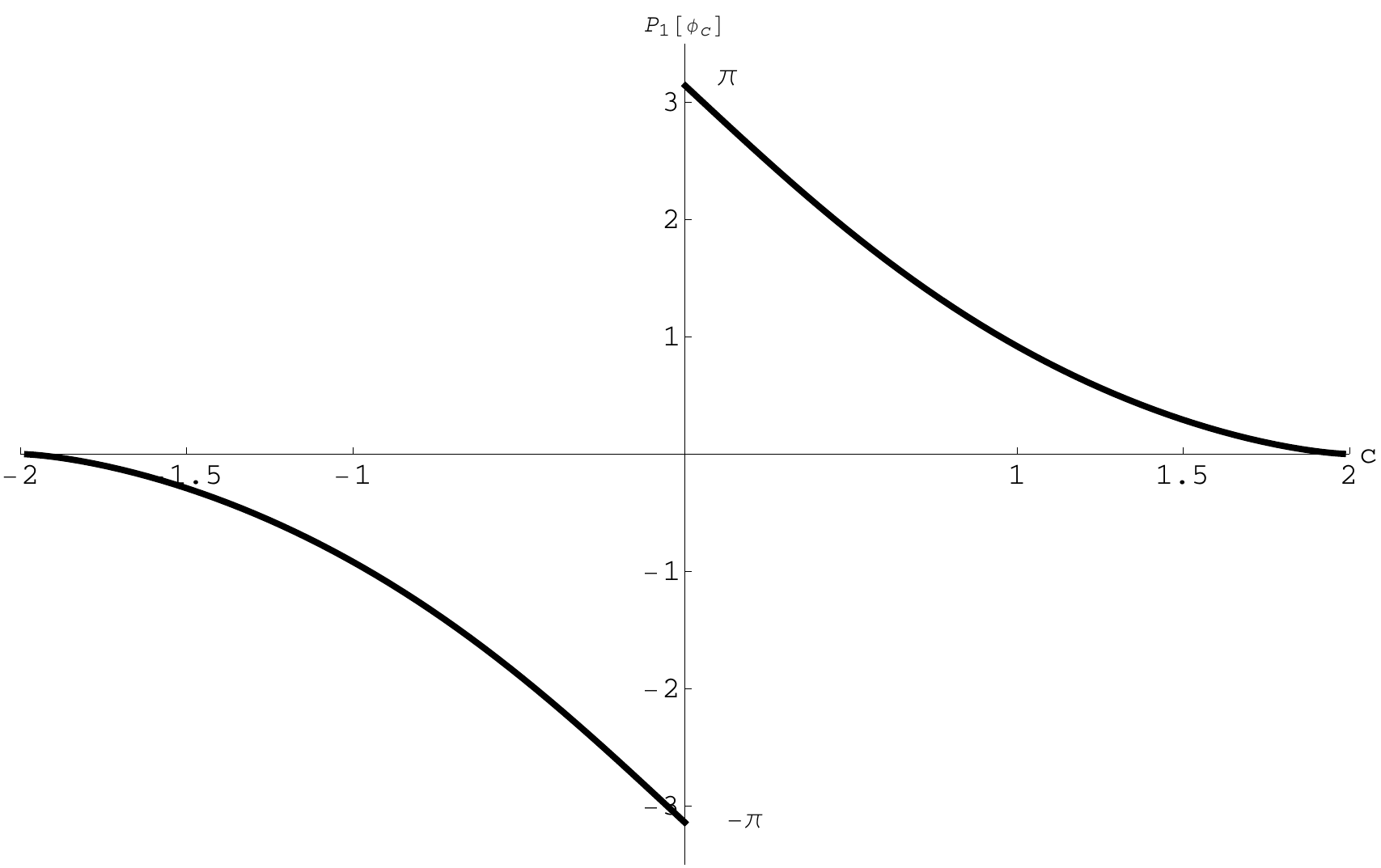} 
	\caption{\label{figP1-a} Momentum $P_1$ \eqref{P1} at $\phi_c$ \eqref{darksoliton}.}
\end{figure}

	\item Remember that \eqref{5gp} is phase invariant, and therefore since $e^{i\tau_0}\phi_c \xrightarrow[c\to 0]{}  e^{i\tau_0}\phi_0,~~\tau_0\in(0,2\pi)$,  we also have orbital stability for this phase transformed family of black solitons.
	}
	
	\item The quintic NLS: \[iv_t + v_{xx} - |v|^4v=0.\]
	\noindent
	The application to this model is rather direct, because it only involves the introduction of a rotation in time transformation 
	$u=e^{it}v$ to connect \eqref{5gp} with  the quintic NLS.
	\item {\color{black}Note that some recent works (see \cite{PeliPlum1,PeliPlum2}) have approached another NLS model with modified dispersion terms, and dealing with orbital stability of black solitons using dark solitons with small speed, close to $0$, but without an explicit expression of them and resorting to symetries to simplify the coercivity analysis.}

\end{itemize}

\noindent

\subsection{Structure of the paper}
In Section \ref{Sec1} we introduce the problem and the main result. In Section \ref{Sec2} we
obtain the black and dark solitons of \eqref{5gp} and describe some properties and nonlinear identities and norms based on them. 
In Section \ref{Sec3} we present the coercivity properties of the {\color{black}quintic} Ginzburg-Landau energy 
$E_2$ around black and dark solitons. In Section \ref{Sec4} we study the existence and time growth of some modulation parameters associated to black and dark
solitons. Finally in  Section \ref{Sec5} we prove the main Theorem on the orbital stability of the black soliton of \eqref{5gp}, gathering the  results 
obtained in the previous sections.

\section*{Acknowledgments}

We  thank here to Henrik Kalisch and Didier Pilod for some valid discussions and observations along this work. We also thank to the referees for their comments and suggestions that actually improved a previous version.

\section{Derivation of black and dark solitons for the quintic GP}\label{Sec2}
In this section  we explain the derivation of the black and dark solutions given in \eqref{black5gp} and \eqref{darksoliton}. 
The following basic result will be useful for obtaining the black family \eqref{black5gp}.

\medskip 
\begin{lemma}\label{lemma-elemental-calculus}
Let $b>0$. Then,  
\[
\int_0^y\frac{ds}{(b-s^2)\sqrt{s^2+2b}}=\frac{1}{2b\sqrt{3}}\ln \left( \frac{\sqrt{2b +y^2} + \sqrt{3}y}{\sqrt{2b +y^2} - \sqrt{3}y} \right),
\]
for all $|y| < \sqrt{b}$.
\end{lemma}
\noindent
See Appendix \ref{0a} for a proof of this identity. 

\subsection{Derivation of black solitons}\label{derivationblack}
Using \eqref{edoblack1}, we get after multiplication by $\phi'$
$$\phi(x)\phi'(x) + \phi''(x) \phi'(x) = \phi^5(x) \phi'(x),$$
and then 
$$\frac{d}{dx}\left[\phi(x)^2 + (\phi'(x))^2-\frac{\phi^6(x)}{3} \right] = 0,$$
which yields 
$$\phi(x)^2 + (\phi'(x))^2-\frac{\phi^6(x)}{3}  = K_0. $$

\medskip 
\noindent 
From the boundary conditions at infinity in \eqref{5gp-boundary} we conclude that $K_0=\frac{2}{3}$ and we obtain the following first order ODE
\begin{equation}\label{edoblack2-a}
(\phi')^2 = \frac{1}{3}\phi^6 -\phi^2 + \frac{2}{3} = \frac{1}{3}(1-\phi^2)^2 (2 + \phi^2).
\end{equation}

\medskip
\noindent
Assuming that $\phi'>0$ and integrating, we get

$$\int_{x_0}^{x}\frac{\phi'(\tilde{x})\, d\tilde{x}}{\sqrt{(1 -\phi(\tilde{x})^2)^2 (\phi(\tilde{x})^2 + 2)}}= \frac{x-x_0}{\sqrt{3}},$$

\medskip 
\noindent
where we consider $x_0=\phi^{-1}(0)$. Then, making the change $s=\phi(\tilde{x})$ we have
$$\int_{0}^{\phi(x)}\frac{ds}{\sqrt{(1 -s^2)^2 (s^2 + 2)}}= \frac{x-x_0}{\sqrt{3}}.$$

\medskip 
\noindent
Without loss of generality we can assume $x_0=0$. Since $|\phi(x)|<1$, using Lemma \ref{lemma-elemental-calculus} it follows that
$$\frac{1}{2 \sqrt{3}}\ln \left( \frac{\sqrt{2 + \phi^2(x)} + \sqrt{3}\phi(x)}{\sqrt{2 + \phi^2(x)} - \sqrt{3}\phi(x)}\right)= \frac{x}{\sqrt{3}},$$
which yields
$$\frac{\phi(x)}{\sqrt{2 + \phi^2(x)}}= \frac{1}{\sqrt{3}} \frac{e^{2x}-1}{e^{2x}+1}=\frac{1}{\sqrt{3}}\tanh (x),$$
and consequently
$$\phi(x)=\sqrt{2}\,\frac{\tanh(x)}{\sqrt{3-\tanh^2(x)}},$$
which, {\color{black}in fact is the unique (up to symmetries of the equation) non-trivial stationary solution of the quintic GP} \eqref{5gp}-\eqref{5gp-boundary} and named as \emph{black} soliton.  

\medskip
An important observation  is that  the black soliton $\phi_0$ \eqref{black5gp} has a definite variational structure. More precisely, 
considering the {\color{black}quintic} Ginzburg-Landau energy $E_2[u]$ defined in \eqref{E2}  as the corresponding Lyapunov functional, and 
considering a small perturbation $\zf$ of the black soliton $\phi_0$, namely a $\zf\in \mathcal{H}_0(\R)$ 
with $\mathcal{H}_0(\R)\subset H^1_{loc}(\R)$ to be defined in
\eqref{Hc}, we get, after a power expansion in  $\zf$ of $E_2$ \eqref{E2} 

\[
E_2[\phi_0 + \zf] = E_2[\phi_0] - 2\text{Re}\Big[\int_\R\bar{\zf}(\phi_0'' + (1-|\phi_0|^4)\phi_0)\Big] 
+ \mathcal{O}(\zf^2).
\]
\noindent
Because the first variation of $E_2$ vanishes for \eqref{edoblack1},  the black soliton is characterized as critical point  of the
functional $E_2$ associated to the quintic GP \eqref{5gp}. In fact, it is easy to see that
\be\label{blackenergy}
E_2[\phi_0]:=2\sqrt{3}\arctanh\left(\frac{1}{\sqrt{3}}\right).
\ee

Moreover, it is possible to state the following minimality's characterization on the black soliton solution 

\begin{proposition}[{\cite[Lemma 2.6]{BGSS}}]\label{minimalidadBlack}
Let $E_2$ \eqref{E2} and let $\phi_0$ be the black soliton solution \eqref{black5gp}. Then we have
\[
 E_2[\phi_0]=\inf\Big\{E_2[\phi]:\phi\in H^1_{loc}(\R),~\inf_{x\in\R}|\phi(x)|=0\Big\}.
\]
\noindent
Moreover, if $E_2[\phi]<E_2[\phi_0],$ \, then \,$ \inf_{x\in\R}|\phi(x)|>0.$
\end{proposition}
\begin{proof}
This result is essentially contained in \cite[Lemma 2.6]{BGSS}, where the black soliton case for the cubic GP was considered. 
The extension to the quintic GP case, once we work with the energy $E_2$ is direct and does not require additional steps. We therefore skip the details.
\end{proof}

\subsection{Derivation of dark solitons}\label{derivationdark}

  Once obtained the black soliton \eqref{black5gp}, the detailed construction of the  dark soliton solution \eqref{darksoliton} to \eqref{edodark} is presented in the Appendix \ref{AppenDarkSol}. A sketch of the derivation is the following: bearing in mind that  \eqref{edodark} reduces to \eqref{edoblack1} at $c=0$, we proposed a suitable \emph{ansatz} like 

\be\label{candidate}
{\color{black}\Phi_c}(x)=\frac{ia_1 + a_2\tanh(k x)}{\sqrt{1+ a_3\tanh^2(k x)}},
\ee

\noindent
with $a_1,~a_2,~a_3$ and $k$ as free parameters to be determined in order that \eqref{candidate} is actually a solution of \eqref{edodark}, and verifying 
the asymptotic behavior

\[
 \lim_{x\rightarrow\pm\infty}|{\color{black}\Phi_c}(x)|^2=1.
\]
\noindent
Hence, substituting \eqref{candidate} into \eqref{edodark} and after lengthy manipulations, we got \eqref{darksoliton}, 
with $\mu_1=\sqrt{2}a_1,~\mu_2=\sqrt{2}a_2$ and $a_3=\mu$, satisfying the relation \eqref{murelation} and $k=\ka$ as in \eqref{kappa}.

\medskip 
\noindent 
Finally, we introduce the notion of \emph{dark profile}.

\begin{definition}[Dark profile]
	Let $c\in(-2,2),$ and $x_0\in\R$ be fixed parameters. We define the complex-valued dark profile $\phi_c$ with speed ${\color{black}c\neq0}$ as follows
	
	\be\label{darkprofile}
	\phi_c(x):=\phi_c(x;c,x_0) = {\color{black}\sgn(c)}\frac{i\mu_1(c) + \mu_2(c)\tanh(\ka(c)(x+x_0))}{\sqrt{2}\sqrt{1+\mu(c) \tanh^2(\ka(c)(x+x_0))}}.
	\ee
\end{definition}

\begin{remark}
	Note that the profile $\phi_c$ is the standard profile associated to the dark soliton solution \eqref{darksoliton}. Note moreover, that 
	although $\phi_c$ is not an exact solution of \eqref{5gp}, it can be interpreted as follows: for each $(t,x)\in\R^2,$
	$$(t,x)\longmapsto\phi_c(x;c,x_0-ct),$$
	\noindent
	is an exact dark soliton solution of \eqref{5gp} moving with speed $c$. 
\end{remark}

\medskip 
\subsection{Preliminaries}\label{Preliminaries}
First of all,  we introduce the following notation for the nonlinear weights

\be\label{peso-blackdark}
\eta_0(x)=1-\phi_0^4(x)\quad\text{and}\quad\eta_c(x)=1-|\phi_c(x)|^4,
\ee

\medskip
\noindent 
and for the real and imaginary parts of the dark soliton
\begin{align}
&R_c(x)= \text{Re}\, \phi_c(x)= \frac{\mu_2\tanh(\ka x)}{\sqrt{2}\sqrt{1+\mu \tanh^2(\ka x)}},\label{real-dark}\\
&I_c(x)= \text{Im}\, \phi_c(x)= \frac{\mu_1}{\sqrt{2}\sqrt{1+\mu \tanh^2(\ka x)}}.\label{imaginary-dark}
\end{align}

\medskip 
To simplify the notation, we shall also denote 

\be\label{complexproduct}
\langle f,\, g\rangle_\C = \text{Re}(f\bar{g}).
\ee

\medskip
Moreover, we define the following functional spaces: given $c\in(-2,2)$ we consider the weighted Sobolev space
 
\be\label{Hc}
\mathcal{H}_c(\R):=\left\{f\in C^0(\R,\C): f'\in L^2(\R)~~\text{and}~~\eta_c^{1/2}f\in L^2(\R)\right\},
\ee
\noindent
with the norm 

\be\label{normHc}
\|f\|_{\mathcal{H}_c}:= \Big(\int_\R |f'|^2 + \eta_c|f|^2\Big)^{1/2}.
\ee

\medskip
We will also use $\mathcal{H}^{real}_c(\R)$ to denote the set of real-valued functions in $\mathcal{H}_c(\R)$, that is,
 
\be\label{Hc-a}
\mathcal{H}_c^{real}(\R)=\left\{f\in C^0(\R,\R): f'\in L^2(\R)~~\text{and}~~\eta_c^{1/2}f\in L^2(\R)\right\}.
\ee

\medskip
\noindent
Using the exponential decay of $\eta_c$ we can check that the space $\mathcal{H}_c$ does not depend on the 
velocity $c$ when $|c|\le \cf$,  for some $\cf$ small enough. Even more, the norms $\|\cdot \|_{\mathcal{H}_c}$ are equivalent with 
$\|\cdot \|_{\mathcal{H}_0}$. For further details see Lemma \ref{lemma-integral-black-dark-normH0}.
Therefore, hereafter we simplify the notation using the identification

\be\label{identification-spaces}
\mathcal{H}:=\mathcal{H}_c\quad \text{and} \quad \mathcal{H}^{real}:=\mathcal{H}_c^{real},
\ee

\medskip
\noindent
for all $|c| < \cf$. Beside that, we define a proper subset of $\mathcal{Z}(\R)\subseteq \mathcal{H}(\R)$, namely

\be\label{Ec}
\mathcal{Z}(\R):=\left\{u\in \mathcal{H}(\R): 1-|u|^4 \in L^2(\R)\right\},
\ee

\medskip
\noindent
which has metric structure with the distance. 
\be\label{dc}
d_c(u_1,u_2):= \Big(\|u_1-u_2\|^2_{\mathcal{H}_c} + \|\phi_c^3(|u_1|^2-|u_2|^2)\|^2_{L^2}\Big)^{1/2},
\ee

\medskip
\noindent
for all $|c| < \cf$. Also notice that if $u \in \mathcal{Z}(\R)$, from the computations in \eqref{sigma-well-defined} we see that, 
the energy $E_2$ \eqref{E2} is well defined for elements in $\mathcal{Z}(\R)$.

\begin{remark}
Similarly to the theory developed in \cite{G-Smets} in the context of the cubic GP model, 
here we also have that the unique global solution $u$ of \eqref{5gp}  with initial data $u_0\in\mathcal{Z}(\R)$ 
remains continuous from $\R$ to $\mathcal{Z}(\R)$ endowed with the metric structure induced by $d_c$ \eqref{dc}. 
\end{remark}

\medskip
\subsection{Nonlinear identities and estimates for black and dark solitons} 
Now we present some nonlinear identities related to the black soliton and dark soliton profile \eqref{black5gp} and \eqref{darkprofile}, which shall be useful along the work.   Firstly we note that from \eqref{edoblack2-a} we get

\begin{equation}\label{edoblack3}
\phi_0'(x) =  \tfrac{1}{\sqrt{3}}(1-\phi_0^2(x))\sqrt{2 + \phi_0^2(x)}.
\end{equation}

\medskip
\noindent
In comparison, the dark soliton satisfies the following identity:

\[
\phi_c'(x) = \frac{1}{\sqrt{2}}\frac{\ka \sech^2 (\ka x)}{(1+\mu\tanh^2(\ka x))^{3/2}}\Big( \mu_2-i\mu\mu_1\tanh(\ka x)\Big),
\]

\noindent
and which shows the localized character of $\phi_c'$.

\medskip

Notice that from \eqref{edoblack1}, \eqref{edoblack2-a} and \eqref{blackenergy}, the  black soliton solution \eqref{black5gp} satisfies the  identity 
\[
\begin{split}
\|\phi_0\|^2_{\mathcal{H}_0}=\int_\R\Big( \eta_0\phi_0^2  + (\phi_0')^2\Big)dx&=\int_\R \Big(-\phi_0\phi_0'' + \tfrac{1}{3}(1-|\phi_0|^2)^2(2+|\phi_0|^2)\Big)dx\\
&=E_2[\phi_0]=2\sqrt{3}\arctanh{\left(\frac{1}{\sqrt{3}}\right)}, 
\end{split}
\]

\medskip 
\noindent 
and by direct calculation we have  also $\|\eta_0\|_{L^2(\R)}^2=2\sqrt{3}\arctanh{\left(\frac{1}{\sqrt{3}}\right)}.$ Hence, 
\begin{equation}\label{H0L2}
\|\phi_0\|_{\mathcal{H}_0(\R)}^2 = \|\eta_0\|_{L^2(\R)}^2= E_2[\phi_0]= 2\sqrt{3}\arctanh{\left(\frac{1}{\sqrt{3}}\right)}.
\end{equation}

\medskip
Coming back to \eqref{darksoliton} and using \eqref{mus} and \eqref{murelation},
we get that the explicit {\color{black}quintic} Ginzburg-Landau energy \eqref{E2} of the \emph{dark} soliton \eqref{darksoliton} is 

\be\label{darkenergy}
\begin{aligned}
&E_2[\phi_c]:= \frac{s_1 + s_2\arctanh(\sqrt{|\mu|})}{32\ka\mu^2},
\end{aligned}
\ee

\medskip
\noindent
with
\be\label{m1m2}
\begin{aligned}
&s_1:=\left(2 \mu (\mu+3)-\mu_2^2 (\mu-1)\right) \left(\mu_2^4-4 \mu_2^2 \mu+4 \mu \left(\mu+\ka ^2\right)\right),\\
&s_2:=\frac{12 \mu \ka ^2 \left(\mu_1^2 (\mu-3) \mu+\mu_2^2 (3 \mu-1)\right)-\left(\mu_1^2-2\right)^2 
\left(\mu_2^2 \left(3 \mu^2-2 \mu+3\right)-2 \mu \left(3 \mu^2+10 \mu-9\right)\right)}{3\sqrt{|\mu|}}.
\end{aligned}
\ee

\medskip
\noindent
On the other hand, we get the right convergence of  \eqref{darkenergy} to \eqref{H0L2} when the speed $c$ goes to $0$, namely (see Figure \ref{fig1-a})

\be\label{darkenergyApprox}
\lim_{c\to0}E_2[\phi_c]= E_2[\phi_0]=2\sqrt{3}\arctanh\left(\frac{1}{\sqrt{3}}\right).
\ee

\begin{figure}[h!] \centering
	\includegraphics[scale=.52]{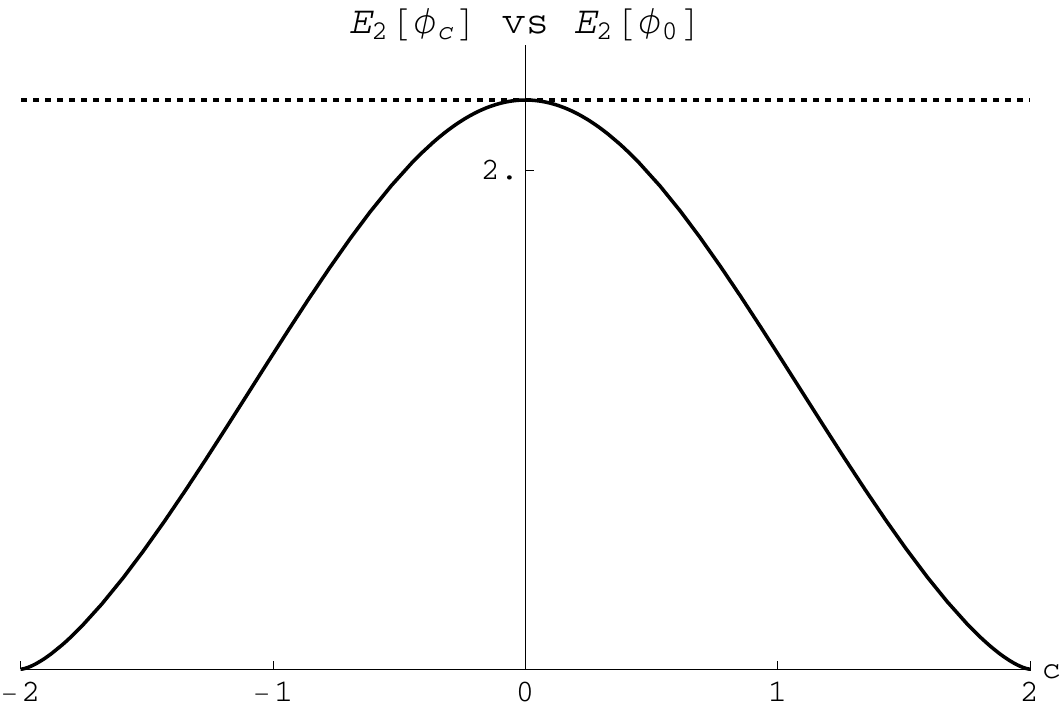} 
	\caption{\label{fig1-a} Comparison between the {\color{black}quintic} Ginzburg-Landau energy \eqref{E2} of the dark (full line) \eqref{darkenergy} and black (dotted) \eqref{blackenergy} soliton solutions of  \eqref{5gp}.}
\end{figure}

\noindent
In fact, the expansion of \eqref{darkenergy} around $\phi_0$ up to $c^2$ order is 

\be\label{darkenergyOrder2}
E_2[\phi_c] = E_2[\phi_0] - \frac{1}{4}\Big(3 + E_2[\phi_0]\Big)c^2  + \mathcal{O}(c^3),
\ee

\noindent
and therefore, for $c$ small enough one gets

\be\label{darkenergyApprox2}
E_2[\phi_c] - E_2[\phi_0]\geq -2\sqrt{3}c^2.
\ee

\medskip

For the sake of completeness, we also show here the following related amounts:

\be\label{d0}
d_0^2(\phi_0,\phi_c):= \|\phi_0 - \phi_c\|^2_{\mathcal{H}_0} + \|\phi_0^3(|\phi_c|^2-\phi_0^2)\|^2_{L^2}.
\ee

\medskip 
\noindent
We now have that if $|c|<\cf,$\;  with some\; $\cf\ll1,$ 

\be\label{difPhis}
\|\phi_0 - \phi_c\|^2_{\mathcal{H}_0}=\mathcal{O}(c^2)
\ee

\medskip
\noindent
and
\[
\|\phi_0^3(|\phi_c|^2-\phi_0^2)\|^2_{L^2}=\mathcal{O}(c^4).
\]

\medskip
\noindent
Therefore, we get that

\be\label{d01}
d_0^2(\phi_0,\phi_c)= \mathcal{O}(c^2). 
\ee

\medskip
\noindent
In the following lines, we show some interesting computations, which are justified in Appendix 
\ref{0}. Firstly we have that

\be\label{L2-ddark-raizetac}
\Bigl\|\frac{\phi_c'}{\sqrt{\eta_c}}\Bigr\|_{L^2}^2\leq \frac{\pi }{3 \sqrt{3}}+\frac{2}{\sqrt{3}}\arccotanh\left(\sqrt{3}\right),
\ee
\medskip
\noindent
for all $|c| < 2$.  On the other hand, because $-\frac{1}{3}\leq \mu<0$ for all $|c| < 2$, we have  

\be\label{maxCoercivity-a}
\big \|I_c\big\|_{L^\infty}=\frac{\mu_1}{\sqrt{2+2 \mu}}=\mathcal{O}(c),
\ee

\medskip
\noindent
with $I_c$ defined in \eqref{imaginary-dark}. We  also have the next useful estimates:

\medskip
\be\label{new-maxCoercivity-a}
 \Big\|\frac{|\phi_c|^2 - \phi_0^2}{\sqrt{\eta_0}}\Big\|_{L^2}=\mathcal{O}(c^2),  
\ee

\be\label{new-maxCoercivity-b}
 \Big\|\frac{\phi_0\eta_0-R_c\eta_c}{\sqrt{\eta_0}}\Big\|_{L^2}=\mathcal{O}(c^2),
\ee

\be\label{new-maxCoercivity-aL2-1}
  \Big\|\frac{\eta_c|\phi_c|^2 - \eta_0\phi_0^2}{\sqrt{\eta_0}}\Big\|_{L^2}=\mathcal{O}(c^2),
\ee

\medskip 
\noindent
and 

\be\label{new-maxCoercivity-aL2-2}
  \Big\|\frac{\eta_c|\phi_c|^2R_c^2 - \eta_0\phi_0^4}{\sqrt{\eta_0}}\Big\|_{L^2}=\mathcal{O}(c^2),  
\ee
\noindent
\medskip
\be\label{new-maxCoercivity-a00}
 \Big\|\frac{|\phi_c|^2 - \phi_0^2}{(1+x^2)\eta_c}\Big\|^2_{L^\infty}=\mathcal{O}(c^2)
\ee

\medskip 
\noindent 
for all $|c|\le \cf$ with some $\cf\ll1$. Also we have the uniform pointwise estimate

\be\label{black-dark-comparison}
|\phi_0(x)| \lesssim |\phi_c(x)|,  \quad \forall\, x\in \R,\, |c|\le \cf \,, ~ \cf\ll1.
\ee
\noindent

For more details on these $L^2$ and $L^\infty$-norms, see subsections $C.1$ and $C.2$ in Appendix \ref{0L2} and \ref{0Linfty} respectively.

\medskip
The next estimate will be useful in subsequent technical results on perturbations of the black soliton $\phi_0$, 
and therefore  we present a brief proof of it.

\begin{lemma}[Equivalent norms]\label{lemma-integral-black-dark-normH0}
Let $\phi_0$ and $\phi_c$ be the black soliton and dark soliton profile \eqref{black5gp} and \eqref{darkprofile} respectively. Then, there exists $\cf \in (0, 2)$ such that 
\begin{equation}\label{estimate-integral-black-dark-normH0}
\int_\R\big||\phi_c|^4-\phi_0^4\big||\zf|^2dx \lesssim c^2\|\zf\|^2_{\cH_{c^*}},
\end{equation}

\noindent 
for all $|c|, |c^*|< \cf$  and any $\zf\in \cH_{c^*}$. Furthermore, we conclude that $\cH_c\equiv \cH_0$ for all $|c| < \cf$ and there exist positive constants $\sigma_1$ and $\sigma_2$ such that 

\begin{equation}\label{equivalence-norm-a}
\sigma_1\|\zf\|_{\mathcal{H}_0}^2 \le \|\zf\|_{\mathcal{H}_c}^2 \le \sigma_2\|\zf\|_{\mathcal{H}_0}^2.
\end{equation}
\end{lemma}

\begin{proof}
From \eqref{new-maxCoercivity-a00} and using the identity $\displaystyle \zf(x)= \zf(0) + \int_0^x\zf'(\tilde{x})d\tilde{x}$, which implies  
\[|\zf(x)| \le |\zf(0)| + |x|^{1/2}\|\zf'\|_{L^2},\]
we have that 

\begin{equation}\label{lemma-integral-black-dark-normH0-a}
\begin{aligned}
\int_\R\big||\phi_c|^4-\phi_0^4\big||\zf|^2dx& \le 2\int_\R \big||\phi_c|^2-\phi_0^2\big||\zf|^2dx\\
&\lesssim c^2\int_\R (1+x^2)\eta_c(x)|\zf(x)|^2dx\\
& \lesssim c^2\Big(|\zf(0)|^2\underbrace{\int_\R (1+x^2)\eta_c(x)dx}_{(I)}
+ \|\zf'\|_{L^2}^2\underbrace{\int_\R (1+x^2)\eta_c(x)|x|dx}_{(II)}\Big). 
\end{aligned}
\end{equation}

\medskip 
\noindent 
Now we show that the last two integrals are uniformly bounded in $c$, with $|c|\le 1$. Firstly, to do this we observe that 
\[\max\big\{ 1+x^2,\; ( 1+x^2)|x|\big\}\lesssim \cosh\big(\kappa(c)x\big),\]
for all $x\in \R$ and $|c|\leq 1$. Furthermore, due to the exponential decay of $\eta_c(x)$ one gets
\begin{equation*}
\begin{aligned}
 (I) + (II) \lesssim &\int_\R \cosh\big(\kappa(c)x\big)\eta_c(x)dx =
\int_\R \cosh\big(\kappa(c)x\big)\left(1-\frac{(\mu_1^2+\mu_2^2\tanh^2(\kappa(c)x))^2}{4(1+\mu\tanh^2(\kappa(c)x))^2}\right)dx\\
& =\frac{\pi}{4\sqrt{2}}\frac{12-4\mu_1^2 -\mu_1^4}{\kappa(c)\sqrt{\mu_1^2 + \mu_2^2}}.
\end{aligned}
\end{equation*}

\noindent 
So, using this control, from \eqref{lemma-integral-black-dark-normH0-a} we conclude that

\begin{equation}\label{lemma-integral-black-dark-normH0-b}
\int_\R\big||\phi_c|^4-\phi_0^4\big||\zf|^2dx \lesssim c^2 \big( |\zf(0)|^2 + \|\zf\|_{\cH_{c^*}}^2\big). 
\end{equation}

\noindent
To estimate $|\zf(0)|^2$ we consider a cut-off function $\chi\in C^{\infty}(\R,[0,1])$ such that 
\[ \chi =  1 \; \text{on}  \; [-1,1] \quad  \text{and} \quad  \chi =  0\;  \text{on}\;  \R\setminus[-2,2].\] 
Then, using that the functions 
 
\[\chi(x)/\sqrt{\eta_{c^*}(x)}\quad \text{and} \quad  \chi'(x)/\sqrt{\eta_{c^*}(x)}\]

\medskip 
\noindent 
are bounded  on $\R$ (uniformly for  $|c^*|< \cf$), combined with the Sobolev embedding, we have 

\begin{equation}\label{lemma-integral-black-dark-normH0-c}
|\zf(0)|^2 \le \|\chi \zf \|^2_{L^{\infty}}\lesssim \|\chi \zf \|_{L^2}\big(\|\chi' \zf \|_{L^2} + \|\chi \zf' \|_{L^2}\big) \lesssim \|\zf\|^2_{\cH_{c^*}}. 
\end{equation}

\medskip
\noindent
Thus, \eqref{estimate-integral-black-dark-normH0} follows by substituting \eqref{lemma-integral-black-dark-normH0-c} into \eqref{lemma-integral-black-dark-normH0-b}.  
Finally, in view of \eqref{estimate-integral-black-dark-normH0},  and using the relation 

\[
\|\zf\|_{\mathcal{H}_c}^2 -\|\zf\|_{\mathcal{H}_0}^2  = \int_\R(\phi_0^4-|\phi_c|^4)|\zf|^2dx, 
\] 

\noindent 
we check that $\cH_c \equiv \cH_0$ for all $|c| < \cf$ and further we have \eqref{equivalence-norm-a}. 
\end{proof}

\section{Coercivity  of the {\color{black}quintic} Ginzburg-Landau energy}\label{Sec3}

In this section we establish that the {\color{black}quintic} Ginzburg-Landau energy $E_2$ \eqref{E2} is coercive  around {\color{black} $\phi_0$ and $\phi_c$} solitons respectively. First of all, we establish some preliminary notation and results. 

\medskip 

We first expand the energy $E_2$ in \eqref{E2} {\color{black} around $\phi_0$} given in \eqref{black5gp}.  Let, $\zf:= \zf_1+ i\zf_2$, with $\zf_1, \zf_2\in \R$, and define
\be\label{rho0}
\rho_0(\zf):=|\phi_0+\zf|^2-|\phi_0|^2= 2\re(\phi_0\bar{\zf}) + |\zf|^2 = 2\phi_0\zf_1 + |\zf|^2. 
\ee
\medskip
\noindent
Then, 
\be \nonumber 
\begin{aligned}
 E_2[\phi_0+\zf]&=  \int_\R \Big[|\phi'_0+\zf_x|^2 + \tfrac13(1-|\phi_0+\zf|^2)^2(2+|\phi_0+\zf|^2)\Big]dx\\
 &=\int_\R \Big[(\phi'_0)^2 + 2\re(\phi'_0\bar{\zf}_x) + |\zf_x|^2 + \tfrac13(1-\phi_0^2 - \rho_0)^2(2+\phi_0^2+ \rho_0)\Big]dx\\
 &= E_2[\phi_0] - 2\re\int_\R \bar{\zf}\big(\phi''_0 + \eta_0\phi_0\big)dx\\
 &\hspace{3cm}+ \int_\R\big(|\zf_x|^2 - \eta_0|\zf|^2\big)dx + \int_\R\big(\phi_0^2\rho_0^2  +  \tfrac{1}{3}\rho_0^3\big)dx,
\end{aligned}
\ee

\medskip 
\noindent 
thus, using \eqref{edoblack1}, we have

\be\label{E2expanBlack}
 E_2[\phi_0+\zf] -E_2[\phi_0] = 2 Q_0[\zf] + \cN_0[\zf],
\ee

\medskip 
\noindent
where $Q_0[\zf]$ is the quadratic form 
\be\label{Q0}
Q_0[\zf]:=\frac{1}{2}\int_\R\big(|\zf_x|^2 - \eta_0|\zf|^2\big)dx,
\ee
and $\mathcal{N}_0[\zf]$ is the nonlinear term 
\be\label{N0}
\mathcal{N}_0[\zf]:= \int_\R\big(|\phi_0|^2\rho_0^2  +  \tfrac{1}{3}\rho_0^3\big)dx.
\ee

\medskip 
\noindent
In the case of $\zf=f$ is a real-valued function belonging to the space $\mathcal{H}^{real}(\R)$ (see \eqref{identification-spaces}), 

\be\label{Q01}
Q_0[f]:=\frac{1}{2}\int_\R\big[(f')^2 - \eta_0f^2\big]dx.
\ee

\medskip 
\noindent
Then, considering now  $\mathcal{H}^{real}(\R)$ endowed with the inner product 
\be\label{inner-product}
\langle f,g\rangle_0:= \int_{\R}(f'g' + \eta_0fg)dx, 
\ee
we have that  $\big(\mathcal{H}^{real}(\R),\, \langle \cdot,\cdot\rangle_0\big)$ is a Hilbert space with the induced norm
\be\label{H0-norm-real}
\|f\|_{\cH_0}^2=\int_{\R}\big[(f')^2 + \eta_0f^2\big]dx.
\ee

\medskip 
For a fixed $f\in \mathcal{H}^{real}(\R)$ we have $g \longmapsto \ds \int_\R\eta_0fg\in \big[\mathcal{H}^{real}(\R)\big]'$. Indeed,

\be\label{step1a}
\begin{aligned}
 |\int_\R\eta_0fgdx| &\leq\|\eta_0^{1/2}f\|_{L^2}\|\eta_0^{1/2}g\|_{L^2}\leq \|\eta_0^{1/2}f\|_{L^2}\|g\|_{\cH_0}.
\end{aligned}
\ee

\medskip 
\noindent
Therefore, by the Riesz Theorem, there exists a bounded and self-adjoint operator $T_0$ such that

\be\label{step1a0}
 \langle T_0f,g\rangle_0 = \int_\R\eta_0fgdx,\quad \forall g\in \mathcal{H}^{real}(\R),
\ee

\noindent
and also

\[
 \|T_0f\|_{\cH_0} \leq \|\eta_0^{1/2}f\|_{L^2}.
\]

\medskip
\noindent 
Moreover, the quadratic form $Q_0$ satisfies

\be\label{step1b}
\begin{aligned}
Q_0[f]&=\tfrac12\int_\R\big[(f')^2 + \eta_0f^2\big]dx - \int_\R\eta_0f^2dx\\
&= \big \langle (\tfrac12\id - T_0)f, f \big\rangle_0,
\end{aligned}
\ee

\medskip 
\noindent 
for all $f\in \mathcal{H}^{real}(\R)$.

\begin{lemma}[Compactness of $T_0$] \label{lemma-compact-T0}
The operator $T_0: \mathcal{H}^{real}(\R) \rightarrow \mathcal{H}^{real}(\R)$ is compact.
\end{lemma}

\bp
 Throughout the  proof we will use $M_j,\; j=1,2\dots,6$, to denote some universal constants.

Consider now a sequence $f_n\in \mathcal{H}^{real}(\R)$ such that

\begin{equation}\label{step1b0}
 \|f_n\|_{\cH_0}^2=\|f'_n\|^2_{L^2} + \|\eta_0^{1/2}f_n\|^2_{L^2}\leq M_1,\quad \forall n\in\N.
\end{equation}
\noindent
Then, we can assume that
\[
f_n \rightharpoonup f^* \in  \mathcal{H}^{real}(\R),\quad \text{when} \quad  n\rightarrow\infty.
\]

\medskip 
\noindent
\textbf{\emph {Claim 1:}} It holds that 
\begin{equation}\label{step1c0}
\begin{aligned}
\|\eta_0^{1/2}f_n\|_{H^1}^2\leq M_2,\quad \text{for all}\;\, n\in\N.
\end{aligned}
\end{equation}

\medskip 
To obtain this estimate we note that

\begin{equation}\label{step1c}
\begin{aligned}
 \|\eta_0^{1/2}f_n\|_{H^1}^2 \backsimeq  \|\eta_0^{1/2}f_n\|_{L^2}^2 + \|\eta_0^{1/2}f'_n\|_{L^2}^2
 + \Big\|\frac{-2\phi_0^3\phi_0'}{\eta_0^{1/2}}f_n\Big\|_{L^2}^2.
\end{aligned}
\end{equation}

\medskip 
\noindent 
Now, using \eqref{edoblack3} and that $|\phi_0|\leq1$, we get

\begin{equation}\label{step1d}
\begin{aligned}
\frac{2|\phi_0|^3|\phi_0'|}{\eta_0^{1/2}}
& = \tfrac{2}{\sqrt{3}} \frac{|\phi_0|^3(2 + \phi_0^2)^{1/2}}{(1+\phi_0^2)^{1/2}}(1-\phi_0^2)^{1/2}
  \leq 2\eta_0^{1/2},
\end{aligned}
\end{equation}
\noindent
so we have 
\begin{equation}\label{step1e}
\Big\|\frac{-2\phi_0^3\phi_0'}{\eta_0^{1/2}}f_n\Big\|_{L^2} \leq 2 \|\eta_0^{1/2}f_n\|_{L^2}.
\end{equation}

\medskip
\noindent 
Then, using that  $\eta_0(x) \le 3 \sech^2(x)$, combined with   \eqref{step1b0}, \eqref{step1c} and \eqref{step1e} we obtain the statement 
in \eqref{step1c0} and Claim 1 \eqref{step1c0} is proved. 

\medskip 
In particular, from \eqref{step1c0} we conclude that 

\begin{equation}\label{step1k}
|f_n(0)|\le \|\eta_0^{1/2}f_n\|_{L^{\infty}} \le M_3,\quad \forall \; n\in\N,
\end{equation}

\noindent 
and also we can assume that

\begin{equation}\label{step1f}
\eta_0^{1/2}f_n\longrightarrow \eta_0^{1/2}f^* \in C_{loc}^{0}(\R),
\end{equation}

\medskip 
\noindent
i.e. we get uniform convergence on  compact subsets of $\R$.  

\medskip 
\noindent
\textbf{\emph {Claim 2:}} It holds that
\begin{equation}\label{step1l}
\|\eta_0^{1/4}f_n\|_{L^2}\leq M_4,\quad \text{for all}\;\, n\in\N.
\end{equation}

\medskip 
To prove this estimate we first observe that 

\[
 f_n(x) = f_n(0) + \int_{0}^xf_n'(s)ds,\quad\forall n\in\N,
\]

\noindent 
which implies that 

\[
\eta_0^{1/4}|f_n(x)| \leq \eta_0^{1/4}|f_n(0)| + |x|^{1/2}\eta_0^{1/4}||f'_n||_{L^2}. 
\]

\medskip 
\noindent
Then, from \eqref{step1b0}, \eqref{step1k} and using the exponential decay of $\eta_0$ we get the estimate \eqref{step1l} in Claim 2.

 \medskip 
 Now given $\epsilon>0$, due to the exponential decay of $\eta_0$, we can take $a_{\epsilon}>0$ such that

 \[
 \eta_0^{1/2} < \epsilon,\quad \forall\;  |x|>a_{\epsilon}.
 \]
 
 \noindent
 Then, using \eqref{step1l}, we have
 
\begin{equation}\label{step1f2}
\int_{|x|>a_{\epsilon}} \eta_0(f_n - f^*)^2dx \le \epsilon \int_{|x|>a_{\epsilon}} \eta_0^{1/2}(f_n - f^*)^2dx
\leq \epsilon \|\eta_0^{1/4}(f_n - f^*)\|^2_{L^2} \leq M_5\,\epsilon.
\end{equation}

\medskip
\noindent
Then,  from \eqref{step1f} one gets
 
\[
 ||\eta_0^{1/2}(f_n - f^*)||_{L^{\infty}(|x|\leq a_{\epsilon})} < \epsilon,\quad \forall n>n_\epsilon,~~\text{with some}~~n_\epsilon\gg1.
\]

\medskip 
\noindent
Therefore, by using \eqref{step1l}, we get 
 
\begin{equation}\label{step1n}
\int_{|x|\leq a_{\epsilon}} \eta_0(f_n - f^*)^2dx \leq \epsilon \int_{|x|\leq a_{\epsilon}}\eta_0^{1/2}|f_n - f^*|dx
\leq \epsilon \|\eta_0^{1/4}\|_{L^2}\|\eta_0^{1/4}(f_n - f^*)\|_{L^2}
\leq M_6\epsilon.
\end{equation}

\medskip 
\noindent
Now, from \eqref{step1f2} and \eqref{step1n} notice that

\[
 \int_{-\infty}^{+\infty} \eta_0(f_n - f^*)^2dx\lesssim \epsilon,
\]
\noindent
for all $n\gg n_{\epsilon}$, so $ \lim_{n\rightarrow\infty}\|\eta_0^{1/2}(f_n - f^*)\|_{L^2}=0$. Finally, from \eqref{step1a0} we have that 

$$\|T_0(f_n-f^*)\|_{\cH_0}\leq \|\eta_0^{1/2}(f_n - f^*)\|_{L^2},$$

\noindent
which implies that 
\[
 T_0f_n\xrightarrow{n\to \infty} T_0f^* \quad \text{in}\quad   \mathcal{H}^{real}(\R),
\]
and the proof is finished. 
\ep

\begin{proposition}[Coercivity of $E_2$ around the black soliton]\label{prop1}
Let $\zf\in  \mathcal{H}(\R)$  be such that the perturbation  $\phi_0 + \zf\in \mathcal{Z}(\R)$ and set $\rho_0=2\re(\phi_0 \bar{\zf}) + |\zf|^2$ as in \eqref{rho0}.
Then there exists a universal positive constant $\Lambda_0>0$  such that

\be\label{expanE2}
E_2[\phi_0+\zf] - E_2[\phi_0]\geq \Lambda_0\big(\|\zf\|_{\cH_0}^2 + \|\phi_0^3\rho_0\|_{L^2}^2+ \|\zf_2\rho_0\|_{L^2}^2\big) - \frac{1}{\Lambda_0}\|\zf\|_{\cH_0}^3
\ee
\noindent
as soon as 

\be\label{ortogonalityBlack}
\int_\R\langle\eta_0,\zf\rangle_\C=0,\quad  \int_\R\langle i\eta_0,\zf\rangle_\C =0,\quad \text{and} \quad \int_\R\langle i\phi_0\eta_0,\zf\rangle_\C=0.
\ee
\end{proposition}

\begin{proof}

The proof will be divided into 3 steps. 

\medskip

\noindent \textbf{Step 1}: There exists a  constant $\Lambda_1>0$ such that

\medskip 

\[
Q_0[f]\geq \Lambda_1 \langle f,f\rangle_0 = \Lambda_1\int_\R\big[(f')^2 + \eta_0f^2\big]dx,
\]
\noindent
for any function  $f\in \mathcal{H}^{real}(\R)$ such that 

\medskip 
\noindent (1a)\; $\ds  \int_\R f\eta_0dx =0$\quad  and \quad  $\ds \int_\R f\phi_0\eta_0dx=0.$

\medskip 
\noindent 
Furthermore,

\medskip 
\noindent
(1b)\; $Q_0[f]\geq0$\; if only the first orthogonality  condition in (1a)  is satisfied.

\medskip
\noindent{\it \textbf{Proof of Step 1.}} 
Recall that from \eqref{step1b}  we have $Q_0[f]=\big \langle (\tfrac12\id - T_0)f, f \big\rangle_0
= \langle \widetilde{Q}_0f, f \big\rangle_0$, where

\be\label{tildeQ0}
\widetilde{Q}_0:=\tfrac12\id - T_0.
\ee

Then, using the Spectral Theorem, there exists a sequence$\{\lambda_n\}$ of eigenvalues for $\widetilde{Q}_0$ with
$\displaystyle \lim_{n\to +\infty}\lambda_n=\tfrac12$ and a Hilbert basis $\{e_n\}$ of  $\mathcal{H}^{real}(\R)$ such that

\[
 \widetilde{Q}_0e_n=\lambda_n e_n,\quad n\in\N.
\]

\noindent
\medskip
Notice that 
\[Q_0[f]\leq \tfrac12\langle f,f\rangle_0\quad \forall \;f\in \mathcal{H}^{real}_0(\R),\]

\noindent
\medskip
consequently, 
\[
\widetilde{Q}_0\leq \tfrac12\id,\quad (\lambda_n)_{n\in \N}\subset(-\infty,\tfrac12] \quad \text{and}\quad \lambda_n\nearrow \tfrac12.
\] 

\noindent
\medskip 
Now, let $\lambda\in(-\infty,\tfrac12]$ be an eigenvalue with $f$ as the corresponding eigenfunction. 
Then for all $g\in  \mathcal{H}^{real}(\R)$ we have 

\[
 \langle  \widetilde{Q}_0f,g \rangle_0 = \lambda \langle f,g \rangle_0.
\]

\noindent
\medskip
So, from \eqref{step1a0}, it holds that
\begin{equation*}
\tfrac12\int_\R f'g'dx - \tfrac12\int_\R\eta_0fgdx = \lambda\Big[\int_\R f'g'dx +\int_\R\eta_0fgdx\Big]
\end{equation*}
which yields
\[
\int_\R \Big[ (1-2\lambda)f'' + (2\lambda+1)\eta_0f\Big]gdx = 0,
\]
for all $g\in  \mathcal{H}^{real}(\R)$. Thus,

\[
(1-2\lambda)f'' + (2\lambda+1)\eta_0f=0 \Longrightarrow -f'' -\eta_0f = \frac{4\lambda}{1-2\lambda}\eta_0f,
\]
therefore
\begin{itemize}
 \item  $f=1=:e_0$ \quad \text{is a solution for}\quad $\lambda=-\tfrac12$,\medskip 
 \item  $f=\phi_0=:e_1$ \quad \text{is a solution for} \quad $\lambda=0$.
\end{itemize} 

\medskip
\noindent 

Note that, since $\phi_0$ has exactly one zero, the Sturm-Liouville theory guarantees that $\lambda=-\tfrac12$ is the only negative eigenvalue of  $\widetilde{Q}_0$
with kernel given by  the $\text{span}(\phi_0)$, more precisely

\[
\lambda_0=-\tfrac12 < 0 =\lambda_1 < \lambda_2 <  \cdots \cdots  < \tfrac12,
\]

and 

\[\text{Ker}(\widetilde{Q}_0 + \tfrac12\id) = \R,\quad  \text{Ker}(\widetilde{Q}_0) = \R\cdot\phi_0.
\]

\medskip 
\noindent
Thus, expanding $f\in\mathcal{H}^{real}$ on the normalized basis of eigenfunctions  

\[
f=\sum_{n=0}^{+\infty}\langle f,\tilde{e}_n\rangle_0\tilde{e}_n,\quad \tilde{e}_n=\frac{e_n}{\|e_n\|_{\cH_0}},
\]

if $\langle f,1\rangle_0 = \langle f,\phi_0\rangle_0 = 0$ we get $\displaystyle f=\sum_{n=2}^{+\infty}\langle f,\tilde{e}_n\rangle_0\tilde{e}_n$, and then

\be \label{espectral4a}
 Q_0[f] =\langle \widetilde{Q}_0[f],f\rangle_0 =\sum_{n=2}^{+\infty}\lambda_n\langle f,\tilde{e}_n\rangle_0^2
 \geq\lambda_2\sum_{n=2}^{+\infty}\langle f,\tilde{e}_n\rangle_0^2=\lambda_2\langle f,f\rangle_0.
\ee

Hence, under hypothesis:

\[
\langle f,1\rangle_0 = \int_\R\eta_0fdx=0\quad \text{and}\quad \langle f,\phi_0\rangle_0  = 2\int_\R f\phi_0\eta_0=0,
\]

and taking $\Lambda_1:=\lambda_2$ we finish the proof of Step 1. 

\medskip 

\noindent \textbf{Step 2:} Let $\zf\in \mathcal{H}(\R)$ fulfilling the orthogonality conditions \eqref{ortogonalityBlack}. 
Then, it follows that

\[
  E_2[\phi_0 + \zf] - E_2[\phi_0] \geq 2Q_0[\zf_1]+ \tfrac{2}{3}\|\phi_0\rho_0\|_{L^2}^2 + \tfrac{1}{3}\|\zf_2\rho_0\|_{L^2}^2 
  + 2\Lambda_1\|\zf_2\|_{\mathcal{H}_0}^2,
\]

\medskip 
\noindent 
with $\Lambda_1$ as in Step 1. 

\medskip
\noindent{\it \textbf{Proof of Step 2.}}  Recall that from the expansion of $E_2$ in  \eqref{E2expanBlack} we have
\[ E_2[\phi_0 +\zf] - E_2[\phi_0] = 2Q_0[\zf] + \mathcal{N}_0[\zf],\]
where $\zf = \zf_1 + i\zf_2$, $\rho_0 = 2\phi_0\zf_1 + |\zf|^2$ and $Q_0$ satisfying $Q_0[\zf] = Q_0[\zf_1] + Q_0[\zf_2].$ Applying Young's inequality, we obtain the estimate
\begin{equation}\label{Step2-1}
\begin{split}
\mathcal{N}_0[\zf]&=\int_\R\big(\phi_0^2\rho_0^2 + \tfrac{1}{3}\rho_0^3\big)dx\\
&=\int_{\R}\phi_0^2\rho_0^2dx  + \tfrac{1}{3}\int_{\R}\big(2\phi_0\zf_1 + |\zf|^2\big)\rho_0^2dx\\
&\geq \int_{\R}\phi_0^2\rho_0^2dx + \tfrac{1}{3}\int_{\R}|\zf|^2\rho_0^2dx -\Big|\tfrac{2}{3}\int_{\R}\phi_0\zf_1\rho_0^2dx\Big|\\
&\geq \tfrac{2}{3}\int_{\R}\phi_0^2\rho_0^2dx + \tfrac{1}{3}\int_{\R}(|\zf|^2-\zf_1^2)\rho_0^2dx\\
&=\tfrac{2}{3}\int_{\R}\phi_0^2\rho_0^2dx + \tfrac{1}{3}\int_{\R}\zf_2^2\rho_0^2dx.
\end{split}
\end{equation}
On the other hand, from Step 1, the first two orthogonality conditions in \eqref{ortogonalityBlack} imply that 
\be \label{Step2-2}
Q_0[\zf_1]\ge 0 \quad  \text{and} \quad  Q_0[\zf_2]\ge 0,
\ee
while, in addition, the last orthogonality condition in \eqref{ortogonalityBlack} ensures that 
\begin{equation}\label{Step2-3}
Q_0[\zf_2]\ge \Lambda_1||\zf_2||_{\mathcal{H}_0}^2,
\end{equation}
\noindent
where $\Lambda_1:=\lambda_2$ is the first positive eigenvalue obtained in Step 1. 
Then, putting the bounds given in \eqref{Step2-1}, \eqref{Step2-2} and \eqref{Step2-3} into 
the expansion of $E_2,$ we obtain the claimed estimate in Step 2.

\medskip 
Since $Q_0[\zf_1]\ge 0$, in order to complete the proof of Proposition \ref{prop1}, we remark that, bearing in mind the estimate in the Step 2, 
we only need to show the coercivity  property for the operator $Q_0$ on the full variable $\zf$. We will explain this in the next step.

\bigskip
\noindent \textbf{Step 3:} Now we proceed with the proof of \eqref{expanE2}.

\medskip 
\noindent{\it \textbf{Proof of Step 3.}} We begin by estimating the term $\tfrac{2}{3}\|\phi_0\rho_0\|_{L^2}^2$ which appears in the 
lower estimate of the Step 2

\be\label{estim1}
\tfrac{2}{3}\|\phi_0\rho_0\|_{L^2}^2 =\tfrac{2}{3}\|\phi_0^3\rho_0\|_{L^2}^2 + I,\\
\ee
where
\be\label{estim1-a}
\begin{aligned}
I&:=\tfrac{2}{3}\int_\R\eta_0\phi_0^2\rho_0^2dx\\
&=\tfrac{2}{3}\int_\R\eta_0\phi_0^2|\zf|^4dx +  \tfrac{8}{3}\int_\R\eta_0\phi_0^4\zf_1^2dx + \tfrac{8}{3}\int_\R\eta_0\phi_0^3\zf_1|\zf|^2dx\\
&:= I_1 + I_2 + I_3 \ge I_2 -|I_3|.
\end{aligned}
\ee
Now, bearing in mind \eqref{edoblack1} and integrating by parts, we simplify $I_3$ as follows:

\be\label{estim2}
I_3 = -\tfrac{8}{3}\int_\R\phi_0''\phi_0^2\zf_1|\zf|^2dx =\tfrac{8}{3}\int_\R2(\phi_0')^2\phi_0\zf_1|\zf|^2dx + \tfrac{8}{3}\int_\R\phi_0'\phi_0^2(\zf_1|\zf|^2)'dx
:=I_{3,1} +I_{3,2}.
\ee
Using  \eqref{edoblack3}, the inequality  $0 < 1-\phi_0^2 \le 1-\phi_0^4$ and a Gagliardo-Nirenberg inequality we obtain
\be \label{estim3}
\begin{aligned}
|I_{3,1}|&\leq \tfrac{16}{9}\int_\R(1-\phi_0^2)^2(2+\phi_0^2)|\phi_0||\zf_1| |\zf|^2dx\\
&\leq \tfrac{16}{3}\|(1-\phi_0^2)^{2/3}\zf\|_{L^3}^3\\
&\leq \tfrac{16}{3}\|(1-\phi_0^2)^{1/2}\zf\|_{L^3}^3\\
&\lesssim \big\|(1-\phi_0^2)^{1/2}\zf\big\|^{5/2}_{L^2}\big\|\big((1-\phi_0^2)^{1/2}\zf\big)'\big\|^{1/2}_{L^2}\\
&\lesssim \big\|\eta_0^{1/2}\zf\big\|^{5/2}_{L^2}\big( \big\|\phi_0\phi'_0(1-\phi_0^2)^{-1/2}\zf\big\|_{L^2} +
\big\|(1-\phi_0^2)\zf'\big\|_{L^2}\big)^{1/2}\\
&\lesssim \big\|\zf\big\|_{\cH_0}^{5/2}\big( \big\|\phi_0\phi'_0(1-\phi_0^2)^{-1/2}\zf\big\|_{L^2} + \|\zf'\|_{L^2}\big)^{1/2}\\
&\lesssim \big\|\zf\big\|_{\cH_0}^{5/2}\big( \big\|(1-\phi_0^2)^{1/2}\zf\big\|_{L^2} + \|\zf'\|_{L^2}\big)^{1/2}\\
&\lesssim \big\|\zf\big\|_{\cH_0}^3, 
\end{aligned}
\ee
and in a similar way we deduce 
\be \label{estim4}
\begin{aligned}
|I_{3,2}|&\leq \tfrac{8}{3\sqrt{3}}\int_\R\phi_0^2(1-\phi_0^2)(2+\phi_0^2)^{1/2}\big|\zf_1'(3\zf_1^2 + \zf_2^2) + 2\zf_1\zf_2\zf_2'\big|dx\\
&\leq \tfrac{8}{3}\int_\R(1-\phi_0^2)\big|\zf_1'(3\zf_1^2 + \zf_2^2) + 2\zf_1\zf_2\zf_2'\big|dx\\
&\lesssim \int_\R(1-\phi_0^2) (|\zf_1'|+ |\zf_2'|)|\zf|^2dx\\
&\lesssim \| |\zf_1'|+ |\zf_2'|\|_{L^2} \|(1-\phi_0^2)^{1/2}\zf\|_{L^4}^2\\
&\lesssim \|\zf\|_{\cH_0}\|(1-\phi_0^2)^{1/2}\zf\|_{L^2}^{3/2}\|\big( (1-\phi_0^2)^{1/2}\zf\big)'\|_{L^2}^{1/2}\\
&\lesssim \|\zf\|_{\cH_0}^3.
\end{aligned}
\ee
Therefore, combining \eqref{estim1}, \eqref{estim1-a}, \eqref{estim2}, \eqref{estim3} and \eqref{estim4} we get, for some positive  number $\gamma$

\be\label{estim5}
\begin{aligned}
\tfrac{2}{3}\|\phi_0\rho_0\|_{L^2}^2&\ge \tfrac{2}{3}\|\phi_0^3\rho_0\|_{L^2}^2 + \tfrac{8}{3}\int_\R\eta_0\phi_0^4\zf_1^2dx - \gamma \|\zf\|_{\mathcal{H}_0}^3.
\end{aligned}
\ee
\noindent
Now, we consider the real function $\tilde{\zf}_1:= \zf_1-\langle \zf_1, e_1\rangle_0e_1$, where $e_1=\phi_0/\|\phi_0\|_{\cH_0}$. 
Then, using the first orthogonality condition in \eqref{ortogonalityBlack} we have 
$\langle \tilde{\zf}_1, e_0\rangle_0=\langle \tilde{\zf}_1, e_1\rangle_0 =0$. Thus, the expansion of $\tilde{\zf}_1$ is given by 
\[\tilde{\zf}_1=\sum_{n=2}^{+\infty}\langle \tilde{\zf}_1, e_n\rangle_0 e_n\quad \text{and}\quad Q_0[\zf_1]=Q_0[\tilde{\zf}_1].\]
Hence, from Step 1, it follows that 
\[
\begin{aligned}
Q_0[\zf_1] & \ge \lambda_2\|\zf_1-\langle \zf_1, e_1\rangle_0e_1\|_{\cH_0}^2=\lambda_2\|\zf_1\|_{\cH_0}^2-
\lambda_2\langle \zf_1, e_1\rangle_0^2.
\end{aligned}
\]
Now, for any number  $0< \nu <1$ which will be chosen later, using the identities \eqref{H0L2} and 
\[
\|\phi_0\|_{\cH_0}^2 =2 \|\phi_0'\|_{L^2}^2\quad \text{and} \quad \int_\R\eta_0dx=3,
\]
and combined with the  Cauchy-Schwarz inequality we have

\be\label{estim7-b}
\begin{aligned}
	Q_0[\zf_1] & \ge \lambda_2\|\zf_1\|_{\cH_0}^2-\lambda_2\langle \zf_1, e_1\rangle_0^2\\
	&= \lambda_2\|\zf_1'\|_{L^2}^2 + \lambda_2\|\eta_0^{1/2}\zf_1\|_{L^2}^2 \\
	&\hspace{2.5cm} -\frac{\lambda_2}{\|\phi_0\|_{\cH_0}^2}\Big((1-\nu)\int_{\R}\zf_1'\phi_0'dx + (1+\nu)\int_\R\eta_0\zf_1\phi_0 dx \Big)^2\\
	&\ge  \lambda_2\|\zf_1'\|_{L^2}^2 + \lambda_2\|\eta_0^{1/2}\zf_1\|_{L^2}^2\\
	&\hspace{2.5cm}-\frac{2 \lambda_2}{\|\phi_0\|_{\cH_0}^2}	
	\Big((1-\nu)^2\|\zf_1'\|_{L^2}^2\|\phi_0'\|_{L^2}^2+ 3(1+\nu)^2\int_\R\eta_0\zf_1^2\phi_0^2 dx \Big)\\
	&\ge \lambda_2\big(1- (1-\nu)^2\big)\|\zf_1'\|_{L^2}^2 +\lambda_2\|\eta_0^{1/2}\zf_1\|_{L^2}^2  
	-6\lambda_2\frac{(1+\nu)^2}{\|\phi_0\|_{\cH_0}^2}\int_\R\eta_0\zf_1^2\phi_0^2 dx.
\end{aligned}
\ee
Now,  using the estimates  $\phi_0^2 \le 1/4 + \phi_0^4$ and $\lambda_2< 1/2$ in   \eqref{estim7-b}, this allows us to 
select a positive constant $\Lambda_{\nu}$ such that

\be\label{estim7-c}
\begin{aligned}
Q_0[\zf_1]& \ge \lambda_2\big(1- (1-\nu)^2\big)\|\zf_1'\|_{L^2}^2 +\lambda_2\|\eta_0^{1/2}\zf_1\|_{L^2}^2\\ 
&\hspace{2.5cm}
-\frac{3\lambda_2(1+\nu)^2}{2\|\phi_0\|_{\cH_0}^2}\int_\R\eta_0\zf_1^2dx-\frac{6\lambda_2(1+\nu)^2}{\|\phi_0\|_{\cH_0}^2}\int_\R\eta_0\zf_1^2\phi_0^4 dx\\
&\ge\lambda_2\big(1- (1-\nu)^2\big)\|\zf_1'\|_{L^2}^2 +\lambda_2\left(1-\frac{3(1+\nu)^2}{2\|\phi_0\|_{\cH_0}^2}\right)\|\eta_0^{1/2}\zf_1\|_{L^2}^2\\ 
&\hspace{7.5cm}
-\frac{3(1+\nu)^2}{\|\phi_0\|_{\cH_0}^2}\int_\R\eta_0\zf_1^2\phi_0^4 dx\\
&\ge \Lambda_{\nu}\|\zf_1\|_{\cH_0}^2-\frac{3(1+\nu)^2}{\|\phi_0\|_{\cH_0}^2}\int_\R\eta_0\zf_1^2\phi_0^4 dx,\\
\end{aligned}
\ee
which holds under the restriction 
\be\label{estim7-d}
\frac{3(1+\nu)^2}{2\|\phi_0\|_{\cH_0}^2} < 1,
\ee
valid for small enough $\nu $  because $\|\phi_0\|_{\cH_0}^2 \approx 2.28> 3/2$ (see \eqref{H0L2}).

\medskip 
On the other hand, from Step 2  and \eqref{estim5}, we have 
\be\label{estim8}
\begin{aligned} 
 E_2[\phi_0 + \zf] - E_2[\phi_0] &\geq 2\lambda_2\|\zf_2\|_{\mathcal{H}_0}^2  + \tfrac{2}{3}\|\phi_0^3\rho_0\|_{L^2}^2
 + \tfrac{1}{3}\|\zf_2\rho_0\|_{L^2}^2\\
 &\hspace{4cm} +  2Q_0[\zf_1]+  \tfrac{8}{3}\int_\R\eta_0\phi_0^4\zf_1^2 - \gamma\|\zf\|_{\mathcal{H}_0}^3.
\end{aligned}
\ee
Then substituting \eqref{estim7-c} in \eqref{estim8} we obtain the estimate
\[E_2[\phi_0 + \zf] - E_2[\phi_0]\ge 2\Lambda_{\nu}\|\zf_1\|_{\mathcal{H}_0}^2 + 2\lambda_2\|\zf_2\|_{\mathcal{H}_0}^2  + \tfrac{2}{3}\|\phi_0^3\rho_0\|_{L^2}^2+ \tfrac{1}{3}\|\zf_2\rho_0\|_{L^2}^2-\gamma \|\zf\|_{\mathcal{H}_0}^3\]
which holds if	
\be\label{estim9}
\frac{6(1+\nu)^2}{\|\phi_0\|_{\cH_0}^2}<\frac83.
\ee
One can find such a $\nu$ because $\|\phi_0\|_{\cH_0}^2 \approx 2.28>\frac{9}{4}$.

\medskip 
Finally, since conditions  \eqref{estim7-d} and \eqref{estim9} are satisfied for a  small enough positive number $\nu$, 
there exists a universal constant $\Lambda_0$  verifying the inequality \eqref{expanE2}, and the proof is complete.
\end{proof}

{\color{black}
\begin{remark}
We remind that in the case of cubic GP treated in \cite{G-Smets}, the corresponding black soliton (denoted as $U_0$) satisfies the relation 
$U_0'=\frac{1}{\sqrt{2}}(1-U_0^2)$, and hence, the  orthogonality condition  
$$\langle f,1\rangle_0 = \int_\R(1-U_0^2)fdx= \sqrt{2}\int_{\R} U_0'fdx=0,$$
but in the case of the quintic GP \eqref{5gp}  this relation is not satisfied anymore. 
\end{remark}
}

Now we perturb the black soliton $\phi_0$ \eqref{black5gp} of \eqref{5gp}, with a function $u\in\mathcal{H}(\R)$ belonging to the orbit generated by the symmetries of
\eqref{5gp}, namely

\be\label{orbit}
 {\color{black}\mathcal{U}_{0}}(\alpha):=\Big\{w \in \mathcal{H}(\R): \inf_{(b,\iota)\in\R^2}||e^{-i\iota}w(\cdot + b) - \phi_0||_{\mathcal{H}_0(\R)} < \alpha  \Big\},
\ee
\noindent
for some $\alpha>0$  and then, given a function $u\in{\color{black}\mathcal{U}_{0}}(\al)$ we can choose $(c, \iota, b)\in (-2,2)\times \R^2$ in such a way that
\[
e^{-i\iota}u(\cdot + b) = \phi_c + \zf,
\]
with $\zf$ satisfying the orthogonality conditions \eqref{genortogonalityDark} around the dark soliton.

\medskip
Finally note that we can define the following tubular subset of ${\color{black}\mathcal{U}_{0}}(\al),$

\be\label{orbitV0}
\mathcal{V}_{0}(\alpha):=\Big\{v\in  \mathcal{Z}(\R): \inf_{(b,\iota)\in\R^2}d_0(e^{-i\iota}v(\cdot +b),\phi_0) < \alpha  \Big\}
\subset  {\color{black}\mathcal{U}_{0}}(\alpha).
\ee

\medskip 
\noindent
Coming back to the main question on the orbital stability of the black soliton, we use the coercivity of $E_2$ around the black soliton $\phi_0$ to small perturbations around the dark soliton $\phi_c$ (see Proposition \ref{anticor1}) . {\color{black} In fact, the idea to introduce a dark soliton family in this argument is to give an extra degree of freedom 
which allows us to satisfy the third constraint in \eqref{prop1} rewritten as \eqref{genortogonalityDark}.}
In that case, the situation is different with respect to the cubic GP equation, because we can not assume  the  cancellation of the linear term $\langle i\phi'_c, \zf\rangle_\C$ in our approach, given the orthogonality conditions arising naturally, from the particular structure of the associated spectral problem as we already saw in Proposition \ref{prop1}, e.g. \eqref{espectral4a}. In fact, this extra technical difficulty introduced by the linear term $\langle i\phi'_c, \zf\rangle_\C$ is overcome in Proposition \ref{anticor1}, by using a previously computed $L^2$ norm \eqref{L2-ddark-raizetac}.

\medskip
Before establishing the next result, and with  \eqref{complexproduct}, we fix the following notation:

\be\label{rho-c}
\rho_c(\zf):=|\phi_c+\zf|^2-|\phi_c|^2=2\langle \phi_c,\zf\rangle_\C + |\zf|^2 = 2\re(\phi_c\bar{\zf}) + |\zf|^2. 
\ee

\begin{proposition}[Coercivity of $E_2$ around the dark soliton]\label{anticor1}
There exists  $\cf \in(0,2)$ small enough \footnote{Note that $\cf$ is chosen as the minimum of the values that guarantee that some precise estimates in the proof {\color{black} hold for}, e.g. \eqref{estimate-integral-black-dark-normH0}, \eqref{estimateE2}.} such that the following holds. For all  $|c|\leq  \cf$ and for any  $\zf\in\mathcal{H}(\R)$ satisfying 
\medskip
\be\label{condittion-cor1}
{\color{black}\phi_c + \zf\in\mathcal{Z}(\R),\quad\text{with}\quad\|\rho_c(\zf)\|_{L^2}<C,\;\; }
\ee
for some constant $C$ and the generalized orthogonality conditions 
\begin{equation}\label{genortogonalityDark}
\int_\R\langle\eta_c,\zf \rangle_\C=0,\quad \int_\R\langle i\eta_c, \zf \rangle_\C=0\quad \text{and} \quad \int_\R\langle iR_c\eta_c,\zf \rangle_\C=0,
\end{equation}

\medskip 
\noindent 
there exists $\tilde{\Gamma}>0,$ not depending on $\cf,$ such that 
\be\label{genexpanE2}
E_2[\phi_c+\zf] - E_2[\phi_0]\geq  \tilde{\Gamma}\big(\|\zf\|_{\mathcal{H}_0}^2 + \|\phi_c^3\rho_c\|_{L^2}^2\big) - 
\frac{1}{ \tilde{\Gamma} }\big(c^2 + \|\zf\|_{\mathcal{H}_0}^3\big).
\ee
\end{proposition}

\begin{remark}  Note that  the quadratic term $\|\zf_2\rho_c\|^2_{L^2}$ is not appearing in \eqref{genexpanE2} because the lower bound is already 
 guaranteed only with the current terms. Hereafter and for the sake of simplicity we will not include such a quadratic term  but note that keeping it, we would  recover \eqref{expanE2} in the limit $c\longrightarrow0$. 
\end{remark}

\begin{proof}
First of all, we remind that $\zf=\zf_1 + i\zf_2\in\mathcal{H}(\R)$ and that defining the quadratic form in $\zf$

\be\label{Qc}
Q_c[\zf]:= \frac{1}{2}\int_\R\Big(|\zf_x|^2 - \eta_c|\zf|^2\Big)dx,
\ee
\noindent
and the nonlinear term

\be\label{Nc}
\mathcal{N}_c[\zf]:= \int_{\R}\Big(|\phi_c|^2\rho_c^2  +  \frac{1}{3}\rho_c^3\Big)dx,
\ee
\noindent
we have 
\be\label{genexpanE2original}
\begin{aligned}
 E_2[\phi_c+\zf] - E_2[\phi_c] = -c\int_{\R}2\re(i\phi_c'\bar{\zf})dx + 2Q_c[\zf]+ \mathcal{N}_c[\zf].
\end{aligned}
\ee
\noindent
Also (see \eqref{darkenergyApprox}, \eqref{darkenergyOrder2} and \eqref{darkenergyApprox2}) we already know that for small $c$
\be\label{estimateE2}
 E_2[\phi_c] - E_2[\phi_0] =   -\frac{1}{4}\Big(3 + E_2[\phi_0]\Big)c^2 + \mathcal{O}(c^3) \geq -2\sqrt{3}c^2.
\ee

\noindent
We recall (see \eqref{E2expanBlack}) that

\be\label{estimateE3}
E_2[\phi_0+\zf] - E_2[\phi_0] =  2Q_0[\zf_1] + 2Q_0[\zf_2] + \mathcal{N}_0[\zf], 
\ee

\medskip 
\noindent
which implies
\begin{multline}\label{estimateE4}
E_2[\phi_c+\zf] - E_2[\phi_0] = \big( E_2[\phi_c+\zf] - E_2[\phi_0+\zf] \big) + 2Q_0[\zf_1] + 2Q_0[\zf_2] + \mathcal{N}_0[\zf].
\end{multline}

\medskip 
We begin by computing the first term on the r.h.s. of \eqref{estimateE4}. Note that subtracting \eqref{genexpanE2original} and \eqref{E2expanBlack}, we have
\begin{multline}\label{estimateE5}
 E_2[\phi_c+\zf] - E_2[\phi_0 + \zf] = E_2[\phi_c]- E_2[\phi_0]\\ -c\int_{\R}2\re(i\phi_c'\bar{\zf})dx + \int_{\R}\big(|\phi_c|^4  - \phi_0^4 \big)|\zf|^2dx + \Delta\mathcal{N}[\zf],
\end{multline}

\noindent 
where 
\be\label{restaeta23}
\Delta\mathcal{N}[\zf]:= \mathcal{N}_c[\zf] - \mathcal{N}_0[\zf] = \int_{\R}\big(|\phi_c|^2\rho_c^2 - \phi_0^2\rho_0^2\big)dx  
+  \int_{\R}\frac{1}{3}\big(\rho_c^3-\rho_0^3\big)dx.
\ee

\noindent
Substituting  \eqref{estimateE5}  in the r.h.s. of \eqref{estimateE4} we get

\begin{multline}\label{estimateE6}
E_2[\phi_c+\zf] - E_2[\phi_0] = 2Q_0[\zf_1] + 2Q_0[\zf_2] + \big(E_2[\phi_c]- E_2[\phi_0]\big)\\
  + \mathcal{N}_c[\zf] -c\int_{\R}2\re(i\phi_c'\bar{\zf})dx + \int_{\R}\big(|\phi_c|^4  - \phi_0^4\big)|\zf|^2 dx.
\end{multline}

\medskip 
\noindent
Hereinafter, unless otherwise noted, we shall consider the constant  $\cf$ as defined in \eqref{equivalence-norm-a}. Now we proceed
to estimate the last three terms in \eqref{estimateE6} 
and we begin by the last integral. Using  \eqref{estimate-integral-black-dark-normH0} we have 

\be\label{estimateE8}
\begin{aligned}
\left|\int_{\R}\big(|\phi_c|^4  - \phi_0^4\big)|\zf|^2 dx\right|
&\le \int_{\R}\left| |\phi_c|^4 - \phi_0^4\right | |\zf|^2dx\\
&  \le \beta c^2\|\zf\|_{\cH_0}^2,
\end{aligned}
\ee
\noindent
 for some positive constant $\beta$ and all $|c| \le \cf$.
\medskip 
Now we continue estimating  the linear term   $-c\displaystyle \int_{\R}2\re(i\phi_c'\bar{\zf})dx$. In fact, we use 
\eqref{equivalence-norm-a} and \eqref{L2-ddark-raizetac} to  get 
  \be\label{estimateE9}
  \begin{aligned}
  \left|-c\displaystyle \int_{\R}2\re(i\phi_c'\bar{\zf})dx\right|
  &\le 2|c|\left \|\frac{\phi_c'}{\sqrt{\eta_c}} \right \|_{L^2} \Big\|\sqrt{\eta_c}\,\zf \Big\|_{L^2}\\
  & \le \beta |c|\,\|\zf \|_{\cH_0},
  \end{aligned}
  \ee 
  
  \medskip
  \noindent
 with a larger constant $\beta$  if necessary, and all $|c| \le \cf$. Now we estimate the nonlinear term  $\cN_c[\zf]$. Using that $|\re(\phi_c\bar{\zf})|  \le \frac{|\phi_c|^2 + |\zf|^2}{2}$
we get 
   \be\label{estimateE10}
     \begin{aligned}
    \cN_c[\zf]& = \int_{\R}|\phi_c|^2\rho_c^2dx  +  \frac{1}{3}\int_{\R}\big( 2\re(\phi_c\bar{\zf}) + |\zf|^2 \big)\rho_c^2dx\\
     &\ge \int_{\R}|\phi_c|^2\rho_c^2dx  + \frac{1}{3}\int_{\R}|\zf|^2\rho_c^2dx - \frac{2}{3}\left|\int_{\R}\re(\phi_c\bar{\zf})\rho_c^2dx \right|\\
     & \ge \frac{2}{3}\int_{\R}|\phi_c|^2\rho_c^2dx.
     \end{aligned}
     \ee 
 
 \medskip
 \noindent
Combining \eqref{estimateE2}, \eqref{estimateE6}, \eqref{estimateE8}, \eqref{estimateE9} and \eqref{estimateE10} and Young's inequality we obtain 

\begin{equation}\label{estimateE11}
\begin{split}
E_2[\phi_c+\zf]- E_2[\phi_0] &\ge 2Q_0[\zf_1] + 2Q_0[\zf_2]\\ 
&\quad \quad -2\sqrt{3}c^2 + \frac{2}{3}\|\phi_c\rho_c\|_{L^2}^2 -\beta|c|\,\|\zf \|_{\cH_0} 
- \beta c^2\|\zf\|_{\cH_0}^2\\
&\ge 2Q_0[\zf_1] + 2Q_0[\zf_2]\\ 
&\quad \quad -2\sqrt{3}c^2 + \frac{2}{3}\|\phi_c\rho_c\|_{L^2}^2 -\beta_1c^2 - \beta_2\|\zf \|^2_{\cH_0}
- \beta c^2\|\zf\|_{\cH_0}^2,
\end{split}
\end{equation}
\noindent
for all $|c| \le \cf$ and $\beta_2$ to be fixed later. Now we split the components of the perturbation  $\zf =\zf_1 + i \zf_2$ in the following way
\begin{equation}\label{new-anticor1-a}
\begin{aligned}
&\zf_1 = \zf_1^*  +  \omega_1(c) \eta_0\\
&\zf_2 = \zf_2^*  +  \omega_2(c) \eta_0 + \omega_3(c)\phi_0\eta_0,
\end{aligned}
\end{equation}
with $\omega_1$, $\omega_2$ and $\omega_3$  real-valued  functions chosen so that  $\zf^* := \zf_1^* + i \zf_2^*$ satisfies the orthogonality conditions
in \eqref{ortogonalityBlack}.  Thus, using \eqref{genortogonalityDark}, the functions $\omega_i$ satisfy the relations
\begin{equation}\label{new-anticor1-b}
\begin{aligned}
&\int_\R\langle\eta_0-\eta_c,\zf \rangle_\C dx=\int_\R\langle\eta_0,\zf \rangle_\C dx = \omega_1(c)\int_{\R}\eta_0^2dx,\\
&\int_\R\langle i\eta_0-i\eta_c,\zf \rangle_\C dx =\int_\R\langle i\eta_0,\zf \rangle_\C dx=\omega_2(c)\int_{\R}\eta_0^2dx + \omega_3(c)\int_{\R}\phi_0\eta_0^2dx,\\
&\int_\R\langle i\phi_0\eta_0-iR_c\eta_c,\zf \rangle_\C dx = \int_\R\langle i\phi_0\eta_0,\zf \rangle_\C dx=\omega_2(c)\int_{\R}\phi_0\eta_0^2dx 
+ \omega_3(c)\int_{\R}\phi_0^2\eta_0^2dx,
\end{aligned}
\end{equation}
\noindent
and the system has a solution, because it has a nonvanishing determinant\footnote{Note that for parity reasons $\int_{\R}\phi_0\eta_0^2dx=0.$}.

\medskip
\noindent
Now, using \eqref{new-maxCoercivity-a} and \eqref{new-maxCoercivity-b}, the integrals in the left hand of \eqref{new-anticor1-b} can be estimated as follows

\[\Big|\int_\R\langle\eta_0-\eta_c,\zf \rangle_\C dx\Big| + \Big|\int_\R\langle i\eta_0-i\eta_c,\zf \rangle_\C dx\Big| + 
\Big|\int_\R\langle i\phi_0\eta_0-iR_c\eta_c,\zf \rangle_\C dx\Big| \lesssim  c^2\|\sqrt{\eta_0}\zf \|_{L^2} \lesssim  c^2\|\zf \|_{\cH_0}, 
\]

\medskip 
\noindent 
and therefore there exists a positive constant $\beta_3$ such that 

\begin{equation}\label{new-anticor1-c}
|\omega_1(c)| + |\omega_2(c)| + |\omega_3(c)| \leq \beta_3c^2\|\zf\|_{\cH_0}. 
\end{equation}

\medskip 
\noindent
Notice that we can obtain the following estimates, by using Step 1 of Proposition \ref{prop1} applied to $\zf_1^*$ and $\zf_2^*$, and combined with \eqref{new-anticor1-c}, 
\begin{equation}\label{new-anticor1-d}
\begin{split}
Q_0[\zf_1]& = Q_0[\zf_1^*] - \omega_1^2Q_0[\eta_0] + \omega_1\Big(\int_{\R}\zf_1'\eta_0'dx  - \int_{\R}\zf_1\eta_0^2dx\Big)\\
& \geq Q_0[\zf_1^*]-\omega_1^2|Q_0[\eta_0]| -|\omega_1|\big(\|\zf_1'\|_{L^2}\|\eta_0'\|_{L^2} + \|\zf_1\sqrt{\eta_0}\|_{L^2}\big\|\eta_0^{\frac32}\big\|_{L^2}\big)\\
& \geq Q_0[\zf_1^*]  -\beta_4 c^2\|\zf\|^2_{\cH_0},
\end{split}
\end{equation}

\noindent
where $Q_0[\zf_1^*]\ge 0$, $\beta_4$ is a  positive constant and  $|c| < \cf<1$. Analogously, since $\zf_2^*$ verifies the inequality  $Q_0[\zf_2^*] \ge \Lambda_1\|\zf^*_2\|^2_{\cH_0}$, we deduce that 

\begin{equation}\label{new-anticor1-e}
Q_0[\zf_2] \geq  \Lambda_1\|\zf_2\|^2_{\cH_0} - \beta_4c^2\|\zf\|^2_{\cH_0}.
\end{equation}

\medskip 
\noindent 
for a larger $\beta_4$ if necessary and for all $|c| < \cf<1$. 

\medskip 
On the other hand, since $\zf^*_1$ is orthogonal to $\eta_0$ in $L^2$, the same arguments used to obtain \eqref{estim7-c} in Step 3
of Proposition \ref{prop1} allow us to conclude the existence of positive $\Lambda_{\nu}$ such that 

\begin{equation}\label{new-anticor1-f}
\begin{split}
Q_0[\zf_1^*] &\geq \Lambda_\nu\|\zf^*_1\|_{\cH_0}^2-\frac{3(1+\nu)^2}{\|\phi_0\|_{\cH_0}^2}\int_\R\eta_0{\zf_1^*}^2\phi_0^4 dx,\\
&\geq \Lambda_{\nu}\|\zf_1\|_{\cH_0}^2 -\beta_5c^2\|\zf\|^2_{\cH_0}-\underbrace{\frac{3(1+\nu)^2}{\|\phi_0\|_{\cH_0}^2}\int_\R\eta_0{\zf_1}^2\phi_0^4 dx}
_{J[\zf_1]}
\end{split}
\end{equation}
with $0< \nu <1$ such that $\frac{3(1+\nu)^2}{2\|\phi_0\|_{\cH_0}^2} < 1.$ 

\medskip
Now, combining \eqref{new-anticor1-d}, \eqref{new-anticor1-e} and \eqref{new-anticor1-f} we have

\begin{equation}\label{new-anticor1-g}
Q_0[\zf_1] + Q_0[\zf_2] \geq \Lambda_{\nu}\|\zf_1\|_{\cH_0}^2 + \Lambda_1\|\zf_2\|^2_{\cH_0} - (2\beta_4+\beta_5)c^2\|\zf\|^2_{\cH_0}-J[\zf_1]
\end{equation}

\medskip 
\noindent
and substituting \eqref{new-anticor1-g} in \eqref{estimateE11} it follows that 

\begin{equation}\label{new-anticor1-h}
\begin{split}
E_2[\phi_c+\zf]- E_2[\phi_0] & \ge 2\Lambda_{\nu}\|\zf_1\|_{\cH_0}^2 + 2\Lambda_1\|\zf_2\|^2_{\cH_0}\\ 
&\quad \quad + \frac{2}{3}\|\phi_c\rho_c\|^2_{L^2}-2J[\zf_1]- \beta_2\|\zf \|^2_{\cH_0} -(\beta_1+2\sqrt{3})c^2 
-\beta_6 c^2\|\zf\|_{\cH_0}^2 ,
\end{split}
\end{equation}

\medskip 
\noindent 
where $\beta_6=\beta+4\beta_4 + 2\beta_5$. At this point, to control the effect of $J[\zf_1]$ on the lower bound of \eqref{new-anticor1-h},  
we shall proceed in a similar way as in \eqref{estim1} - \eqref{estim1-a}, estimating the following $L^2$ norm:

\be\label{new-anticor1-i}
\begin{split}
\frac{2}{3}\|\phi_c\rho_c\|_{L^2}^2& \ge \frac{2}{3}\|\phi_c^3\rho_c\|_{L^2}^2 + \frac{8}{3}\int_{\R}\eta_0\phi_0^4\zf_1^2dx + \frac{8}{3}\underbrace{\int_{\R}\big(\eta_c|\phi_c|^2R_c^2-\eta_0\phi_0^4\big)\zf_1^2dx}_{J_a[\zf]}\\
&\qquad - \frac{8}{3}\underbrace{\int_{\R}\eta_c|\phi_c|^2|R_c\zf_1 + I_c\zf_2||\zf|^2dx}_{J_b[\zf]}- \frac{16}{3}
\underbrace{\int_{\R}\eta_c|\phi_c|^2|R_cI_c\zf_1\zf_2|dx}_{J_c[\zf]},
\end{split}
\ee

\noindent
where $R_c,~I_c$ are defined in \eqref{real-dark} and \eqref{imaginary-dark}. Now, using that $|\phi_c|\le 1, $ \eqref{maxCoercivity-a} and \eqref{equivalence-norm-a} one gets 
\be\label{est-J2}
|J_c[\zf]|\lesssim c \|\zf\|^2_{\cH_0}.
\ee

\medskip 
\noindent
On the other hand, since $|\phi_c|\le 1$ we have

\begin{equation}\label{new-anticor1-i-01}
\begin{split}
J_b[\zf]& \le \int_{\R}\eta_c|\phi_c|^3|\zf|^3dx\\
&\le \int_{\R}\big(\eta_c|\phi_c|^2 - \eta_0\phi_0^2\big)|\zf|^3dx +  \int_{\R}\eta_0\phi_0^2|\zf|^3dx\\
&:= J_{b1}[\zf] + J_{b2}[\zf].
\end{split}
\end{equation}
Note that 
\[J_{b2}[\zf] = -\int_{\R}\phi_0''\phi_0|\zf|^3dx;\]
so in a similar way as for the integral $I_3$ in \eqref{estim2} we have
\begin{equation}\label{new-anticor1-i-01-a}
|J_{b2}[\zf]|\lesssim \|\zf\|_{\cH_0}^3.
\end{equation}
%
{\color{black}Remembering that (see Appendix \ref{0b} for details)

\be\label{LinfinitoPertubationEstimateOK}
 \|\zf\|_{L^\infty}\lesssim(1+\|\rho_c\|_{L^2})(1+\|\zf\|_{\mathcal{H}_0}),
\ee
\noindent
then using \eqref{condittion-cor1}, we obtain

\be\label{LinfinitoPertubationEstimate}
\|\zf\|_{L^\infty}\lesssim 1 + \|\zf\|_{\mathcal{H}_0},
\ee
\noindent
}
and finally by \eqref{new-maxCoercivity-aL2-1},  we obtain 
\begin{equation}\label{new-anticor1-i-01-b}
\begin{split}
|J_{b1}[\zf]|& \lesssim {\color{black}(1 + \|\zf\|_{\mathcal{H}_0})^2}\int_{\R}\frac{\eta_c|\phi_c|^2 - \eta_0\phi_0^2}{\sqrt{\eta_0}}\sqrt{\eta_0}|\zf|dx\\
& \lesssim c^2{\color{black}(\|\zf\|_{\mathcal{H}_0}+\|\zf\|_{\mathcal{H}_0}^3).}
\end{split}
\end{equation}
Following the same procedure as in estimate $J_{b1}[\zf]$  and using \eqref{new-maxCoercivity-aL2-2} we get 
\begin{equation}\label{new-anticor1-i-01-c}
|J_a[\zf]|\lesssim c^2{\color{black}(\|\zf\|_{\mathcal{H}_0}+\|\zf\|_{\mathcal{H}_0}^3).}
\end{equation}
Therefore, collecting estimates $J_{a},J_{b}, J_{c},$ we conclude that
\begin{equation}\label{new-anticor1-j}
\frac{2}{3}\|\phi_c\rho_c\|_{L^2}^2 \ge \frac{2}{3}\|\phi_c^3\rho_c\|_{L^2}^2 + \frac{8}{3}\int_{\R}\eta_c|\phi_c|^2R_c^2\zf_1^2dx
- \gamma_1\|\zf\|_{\cH_0}^3 - \gamma_2(c+c^2)\|\zf\|^2_{\cH_0} -\beta_7c^2, 
\end{equation}
for some positive numbers $\gamma_1, \gamma_2, \beta_7$. Finally, we fix $\beta_2$  such that $\beta_2 < \min(\Lambda_1, \Lambda_{\nu})$ and  $0< \nu <1$, smaller if necessary,  such that

\[\frac{6(1+\nu)^2}{\|\phi_0\|_{\cH_0}^2} < \frac{8}{3}. \] 

\noindent 
Then, substituting \eqref{new-anticor1-j} into \eqref{new-anticor1-h}, the second term on the right hand side of  \eqref{new-anticor1-j} allows to control the integral $J[\zf_1]$ and  consequently we can take a positive constant $\Gamma_{\cf}$ such that
\begin{equation*}
E_2[\phi_c+\zf]- E_2[\phi_0]\geq \Gamma_{\cf}\big(\|\zf\|_{\mathcal{H}_0}^2 + \|\phi_c^3\rho_c\|_{L^2}^2\big) - 
\frac{1}{\Gamma_{\cf}}\big(c^2 + \|\zf\|_{\mathcal{H}_0}^3\big),
\end{equation*}
\noindent
for all $|c| < \cf$. 

\medskip
Notice that in the process of obtaining the constant $\Gamma_{\cf}$ we see that this coercivity constant is lower bounded by a constant  $\tilde{\Gamma}$ when  $\cf \to 0$. In other words, $\Gamma_{\cf} \ge \tilde{\Gamma}$ and   $-\frac{1}{\Gamma_{\cf}} \ge -\frac{1}{\tilde{\Gamma}}$ for all $\cf$ in a small interval $(0, \tilde{\gamma}$). Hence, we get 
	\bee
	E_2[\phi_c+\zf] - E_2[\phi_0]\geq \tilde{\Gamma}\big(\|\zf\|_{\mathcal{H}_0}^2 + \|\phi_c^3\rho_c\|_{L^2}^2\big) - 
	\frac{1}{ \tilde{\Gamma}}\big(c^2 + \|\zf\|_{\mathcal{H}_0}^3\big)
	\eee
as claimed in \eqref{genexpanE2}.
\end{proof}

\section{Modulation of parameters}\label{Sec4}
In order to apply the coercivity property of the {\color{black}quintic} Ginzburg-Landau energy $E_2$ shown in Section \ref{Sec3}, we have to ensure that the orthogonality relations hold. 

\medskip 
In this section, we  prove that there exist small perturbations $\zf\in\mathcal{H}(\R)$ such that the orthogonality conditions \eqref{ortogonalityBlack}
for the black soliton are satisfied. In fact, we will prove a more general result, valid for $c\neq0,$ and dealing with 
\emph{generalized orthogonality conditions} \eqref{genortogonalityDark} for perturbations $\zf\in\mathcal{H}(\R)$  around the dark soliton and therefore obtaining the desired orthogonality conditions related with the black soliton \eqref{ortogonalityBlack} in the limit {\color{black}$c=0$}. 

\medskip
Firstly, and for the sake of completeness, we will present a preliminary result on the continuous dependence for the shift and phase parameters $b,~\theta$ on 
the corresponding  dark soliton profile.

\begin{lemma}\label{lemma2}
	Let $(c,a,\theta)\in (-\cf,\cf)\times\R^2$ and set $\phi_{c,a,\theta}:=e^{i\theta}\phi_c(\cdot-a).$ Given a positive number $\delta$, 
	there exists a positive number $\tilde{\delta}$ such that if 
	\[\|\phi_{c,b_1,\theta_1}-\phi_{c,b_2,\theta_2}\|_{\mathcal{H}_c}<\tilde{\delta},\]
	then we have \,\, $|b_2-b_1| + |e^{i\theta_2}-e^{i\theta_1}|<\delta.$
\end{lemma}

\bp
The proof runs exactly as \cite[Lemma 2.1]{G-Smets}. 
\ep

%
%

\medskip

\begin{proposition}[Modulation]\label{prop2a}
	Let  $(c, a, \theta)\in (-\cf, \cf)\times\R^2$. There exist two positive numbers $\tilde{r}_c$ and $\tilde{s}_c$, depending continuously on $c$, for which there exist a map  $(\widetilde{c},~\widetilde{a},~\widetilde{\theta}): B_{{\mathcal{H}_0}} (\phi_{c,a,\theta};\,\tilde{r}_c)\to (-\cf, \cf)\times\R^2$ with $(\widetilde{c},~\widetilde{a},~\widetilde{\theta})(\phi_{c, a, \theta})=(c, a, \theta)$ 
	and such that for any $w\in  B_{{\mathcal{H}_0}} (\phi_{c,a,\theta};\,\tilde{r}_c)$ the perturbation of the dark soliton profile
	
	\be\label{pertProp2}
	\zf:=e^{-i\widetilde{\theta}(w)}w(\cdot + \widetilde{a}(w))-\phi_{\widetilde{c}(w)}, 
	\ee
	
	\medskip
	\noindent
	satisfies the generalized orthogonality conditions \eqref{genortogonalityDark}.
	Moreover, $\widetilde{c},~\widetilde{a},~\widetilde{\theta}$ are $C^1$-functions in $B_{{\mathcal{H}_0}} (\phi_{c,a,\theta},\,\tilde{r}_c)$ and  given any 
	$w\in  B_{{\mathcal{H}_0}} (\phi_{c,a,\theta};\,\tilde{r}_c)$,  the vector $(\widetilde{c},~\widetilde{a},~\widetilde{\theta})(w)$ is the unique element in the ball
	$B((c, a, \theta);\tilde{s}_c)\subset \R^3$ verifying \eqref{genortogonalityDark}.
\end{proposition}

\medskip 
\begin{proof}[Proof of Proposition \ref{prop2a}.]
	The proof of this result is a classical application of the Implicit Function Theorem. We begin by considering the functional 
	\[F:\mathcal{H}\times(-\cf, \cf)\times\R^2\longrightarrow\R^3,\]
	given by
	
	\begin{equation}\label{implicitFunction}
		F(w,\si,b,\iota):=
		\Big(\int_\R\langle\eta_\si, \zf \rangle_\C,\quad\int_\R\langle i\eta_\si,  \zf\rangle_\C, \quad\int_\R\langle iR_\si\eta_\si,  \zf \rangle_\C\Big),
	\end{equation}
	
	\noindent
	where  $\zf:=e^{-i\iota}w(\cdot + b)-\phi_{\si}$.
	
	\medskip 
	Notice that, similarly to the context of the cubic  GP \cite{G-Smets}, the functional $F$  has $C^1$-regularity.  Recall now the notation introduced in
	Lemma \ref{lemma2}
	\begin{equation}\label{phi-tilde}
	\phi_{c,a,\theta}:=e^{i\theta}\phi_c(\cdot-a).
	\end{equation}

	\noindent 
	Then 
	$$F(\phi_{c,a,\theta},c,a,\theta)={\bf0},\quad \text{for all}\quad (c,a,\theta)\in (-\cf, \cf)\times\R^2,$$
	
	\medskip 
	\noindent 
	where ${\bf0}:=(0,0,0)$. On the other hand, we have that\footnote{By parity arguments, some of the terms vanish in \eqref{DerivadasIF}.}
	
	\begin{equation}\label{DerivadasIF}
		\begin{aligned}
			&\partial_\si F(\phi_{c,a,\theta},c,a,\theta) = 
			\Big(0,\,\int_\R \langle i \eta_{c},-\partial_c\phi_c\rangle_\C,\, 0\Big),
			&\\
			&\partial_b F(\phi_{c,a,\theta},c,a,\theta) = 
			\Big(\int_\R\langle  \eta_{c},\, \partial_x\phi_c\rangle_{\C},\, 0\, ,\int_\R \langle iR_{c} \eta_{c},\, \partial_x\phi_c\rangle_{\C}\Big),
			&\\
			&\partial_\iota F(\phi_{c,a,\theta},c,a,\theta) = 
			\Big(\int_\R\langle  \eta_{c},-i\phi_c\rangle_{\C},\, 0\, ,\int_\R \langle iR_{c} \eta_{c},-i\phi_c\rangle_{\C}\Big).                     
		\end{aligned}
	\end{equation}
	
	\medskip
	Let $\mathcal{F}(c)$ be the $3\times3$ matrix  
	$\mathcal{F}(c):=(\partial_{\si} F,\partial_b F,\partial_{\iota} F)(\phi_{c,a,\theta}
	,c,a,\theta)$,  
	which is a continuously diffe\-rentiable function on the interval  $c\in(-\cf, \cf)$. 
	
	\medskip 
	From \eqref{DerivadasIF}, we have that 
	for all  $ c\in(-\cf, \cf)$ (see Appendix \ref{Ap1} for a detailed computation of $det\mathcal{F}(c)$)
	
	\begin{equation}\label{DetFI}
		det\mathcal{F}(c) = -\int_\R \langle i \eta_{c},-\partial_c\phi_c\rangle_\C \times  \mathcal{D}(c) \neq0,
	\end{equation}
	where
	\[
	\mathcal{D}(c)=
	\Big(\int_\R\langle  \eta_{c},\partial_x\phi_c\rangle_{\C}\int_\R \langle iR_{c} \eta_{c},-i\phi_c\rangle_{\C}
	-\int_\R\langle  \eta_{c},-i\phi_c\rangle_{\C}\int_\R \langle iR_{c} \eta_{c},\partial_x\phi_c\rangle_{\C}\Big).
	\]
	
	\medskip 
	\noindent
	
	Therefore, by the Implicit Function Theorem, there exists a neighborhood 
	$$B_{{\mathcal{H}_0}} (\phi_{c,a,\theta};\,\tilde{r}_c)\times B((c, a, \theta);\tilde{s}_c)\subset \mathcal{H}\times ((-\cf, \cf)\times\R^2)$$
    and a unique $C^1$ map $(\widetilde{c},~\widetilde{a},~\widetilde{\theta}): B_{{\mathcal{H}_0}} (\phi_{c,a,\theta};\,\tilde{r}_c)\to B((c, a, \theta);\tilde{s}_c)$ 
	such that 
		\[
		F(w,\widetilde{c}(w),\widetilde{a}(w),\widetilde{\theta}(w))={\bf0},
		\]
	for any $w\in B_{{\mathcal{H}_0}} (\phi_{c,a,\theta};\,\tilde{r}_c)$, and consequently we get \eqref{genortogonalityDark}.
\end{proof}	

{\color{black}
Before establishing the next result, we remember  the neighborhood \eqref{orbit} of the orbit of $\phi_0$,

\be\label{orbit2}
\mathcal{U}_{0}(\alpha):=\Big\{w \in\mathcal{H}(\R): \inf_{(b,\iota)\in \R^2}||e^{-i\iota}w(\cdot +b) - \phi_0||_{\mathcal{H}_0(\R)} < \alpha\Big\},
\ee

\medskip
\noindent
where we split $e^{-i\iota}w(\cdot +b) = \phi_c + \zf$. By taking $\alpha$ smaller, if necessary, we can apply a well known 
standard theory of modulation for the solution  $u(\cdot)\in{\color{black}\mathcal{U}_{0}}(\al)$ of the Cauchy problem \eqref{5gp}.
}

\begin{corollary}\label{1stParProp2}
Let $\tilde{r}_0$ and $\tilde{s}_0$ be the constants established in Proposition \ref{prop2a} for the case $c=0$, chosen in such a way that $\tilde{r}_0\tilde{s}_0<1$. There exists $\alpha >0$ such that for a given $w\in {\color{black}\mathcal{U}_{0}}(\alpha)$ there exist numbers $a_w$ and $\theta_w$ such that 
\[w\in  B_{{\mathcal{H}_0}} (\phi_{0,a_w,\theta_w};\,\tilde{r}_0/2)\]
and the map $(\tilde{c}, \tilde{a}, \tilde{\theta})$ established in Proposition \ref{prop2a} in each ball  $B_{{\mathcal{H}_0}} (\phi_{0,a_w,\theta_w};\,\tilde{r}_0/2)$  is well defined from the neighborhood ${\color{black}\mathcal{U}_{0}}(\alpha)$ with values in $\R^2\times \R/2\pi$. More precisely, the functions $\widetilde{c}(w),\widetilde{a}(w)$ and $\widetilde{\theta}(w)$ (modulo $2\pi$) 
do not depend on which  $(a,\theta)$ parameters are chosen.
\end{corollary}
\begin{proof}
Taking $\alpha\le \alpha_0: = \min\big\{\tilde{r}_0/2,\,\tilde{\delta}/4 \big\}$ (with $\tilde{\delta}$ provided in Lemma \ref{lemma2} in the case $c=0$), the proof follows in a similar way 
as it was done in the first part of the Step 2 in the proof of \cite[Proposition 2]{G-Smets}.
\end{proof}

\medskip 

\begin{corollary}\label{0stParProp2}
Consider $\alpha$ as in Corollary \ref{1stParProp2} and let $u(t,\cdot)$ be the solution of \eqref{5gp} - \eqref{5gp-boundary} 
	 with initial data $u_0$ satisfying $d_0(u_0, \phi_0)< \al$.
	Then, there exist $T>0$ and mappings 
	
	\be\label{mappings}[-T,T]\ni t\mapsto (c(t),a(t),\theta(t)),\ee
	
	\noindent
	such that $F(u(t,\cdot),c(t),a(t),\theta(t))={\bf 0}.$
\end{corollary}
\begin{proof}
	As a direct consequence of the continuity of the quintic GP flow in $\cZ(\R)$, we can find $T>0$ such that 
	
	$$\|u(t,\cdot)-\phi_0\|_{{\mathcal{H}_0}} < d_0(u_0, \phi_0) < \alpha,\quad \forall\; t\in [-T,T],$$
	and consequently
	$u(t,\cdot)\in B_{{\mathcal{H}_0}}(\phi_{0}; \alpha) \subset {\color{black}\mathcal{U}_{0}}(\alpha)$ 
	for all $t\in[-T,T]$. 	So, from Corollary \ref{1stParProp2} we can define the mappings 

	\be\label{mappings-a}t\mapsto c(t),\quad t\mapsto a(t), \quad t\mapsto \theta(t),\ee
	
	\noindent
	on $[-T,T]$ by setting 
	$c(t):=\widetilde c(u(t,\cdot )),~a(t):=\widetilde a(u(t, \cdot)),~\theta(t):=\widetilde \theta(u(t,\cdot)).$ Moreover, the perturbation
	$\zf(t)=e^{-i\theta(t)}u(\cdot + a(t))-\phi_{c(t)}$ satisfies
	
	\begin{multline*}
		F(u(t, \cdot),c(t),a(t),\theta(t))=	\Big(\int_\R\langle\eta_{c(t)},\zf(t)\rangle_{\C},~~\int_\R \langle i\eta_{c(t)},\zf(t)\rangle_{\C},
		~~
		\int_\R\langle iR_{c(t)}\eta_{c(t)},\zf(t)\rangle_{\C}\Big)={\bf 0},
	\end{multline*}
 for all $t\in[-T,T]$. 
\end{proof}

\medskip
Furthermore, using the definition in \eqref{phi-tilde}, we also have an estimate on the size of the modulation parameters involved in the perturbation 

\be\label{zFinal}
\zf(t,\cdot)=e^{-i\theta(t)}u(t,\cdot)-\phi_{c(t)}(\cdot-a(t)),
\ee

\medskip
\noindent
namely the following result:
\begin{corollary}\label{2ndParProp2}
	 Let $\alpha$ be given in Corollary \ref{1stParProp2} and $u(t,\cdot),$ $c(t),\theta(t),a(t)$ as in Corollary \ref{0stParProp2}. There exist positive constants $K_0$ and $A_0$ such that if for some $(a, \theta) \in\R^2$  and $0< \epsilon\leq \min\{ 1,\, \alpha\}$ is satisfied 
	\begin{equation}\label{Cor-ImpCond-Stability}
		\|u(t, \cdot) -\phi_{0, a, \theta}\|_{\mathcal{H}_0}=\|u(t, \cdot) - e^{i\theta}\phi_0(\cdot -a)\|_{\mathcal{H}_0}\leq \epsilon,\quad t\in[-T, T],
	\end{equation}
	 then it follows that
	
	\be\label{desigZ}
	|c(t)| + |a(t) - a| + |e^{i\theta(t)}  - e^{i\theta}| \leq K_0 \epsilon \quad \text{and} \quad \|\zf(t, \cdot)\|_{\mathcal{H}_0}\leq A_0 \sqrt{\epsilon}.
	\ee
\end{corollary}

\begin{proof}
	First of all, note that all components in the mapping 
	$$w\in B_{\mathcal{H}_0}(\phi_{0, a, \theta};\al) \mapsto (\widetilde{c}(w),\widetilde{a}(w),\widetilde{\theta}(w))\in B((0,a,\theta);\tilde{s}_0),$$ 
	\noindent
	are $C^1$-functions, and therefore, Lipschitz continuous with Lipschitz constant $K_0$. So,  from \eqref{Cor-ImpCond-Stability}, we have that
	
	\be\label{estimCorProp2}
	\begin{aligned}
		|c(t)| + |a(t) - a| + |\theta(t) - \theta|&= \big|\widetilde{c}(u(t,\cdot))\big| + 
		\big|\widetilde{a}(u(t, \cdot )) - a\big| + \big|\widetilde{\theta}(u(t, \cdot )) - \theta \big|\\
		&\leq K_0 \|u(t, \cdot ) - \phi_{0, a, \theta}\|_{\mathcal{H}_0}\leq K_0  \epsilon,
	\end{aligned}
	\ee 

	\noindent 
	for all  $t\in [-T,T]$.  This implies the first estimate in \eqref{desigZ}.

	\medskip
	On the other hand, using \eqref{difPhis} we have that
	\be\label{estimCorProp2-01}
	\begin{split}
	\|\phi_{c(t), a(t), \theta(t)} -  &\phi_{0, a, \theta}\|^2_{\mathcal{H}_0} = 	\|e^{i\theta(t)}\phi_{c(t)}(\cdot - a(t)) -  e^{i\theta}\phi_0(\cdot-a)\|^2_{\mathcal{H}_0}\\
	&\lesssim \|e^{i\theta(t)}\phi_{c(t)}(\cdot - a(t)) -  e^{i\theta(t)}\phi_0(\cdot-a(t))\|^2_{\mathcal{H}_0} \\
	& \hspace{5cm}+ \|e^{i\theta(t)}\phi_0(\cdot - a(t)) -   e^{i\theta}\phi_0(\cdot-a)\|^2_{\mathcal{H}_0}\\
	& \lesssim c^2(t) + \|\phi_0(\cdot - a(t)) -  \phi_0(\cdot-a)\|^2_{\mathcal{H}_0} + |e^{i\theta(t)}-e^{i\theta}|^2\|\phi_0(\cdot-a)\|^2_{\mathcal{H}_0}\\
	&\lesssim  c^2(t) + \|\phi_0(\cdot - a(t)) -  \phi_0(\cdot-a)\|^2_{\mathcal{H}_0} + |\theta(t)-\theta|^2\|\phi_0\|^2_{\mathcal{H}_0}.
	\end{split}
	\ee

	\noindent
	 Now, using the Mean Value Theorem there exist $\nu_i=\nu_i(t,x)\in (0,1)$, $i=1,2$, such that
	\be\label{estimCorProp2-02}
	\begin{split}
	\|\phi_0&(\cdot - a(t)) -  \phi_0(\cdot-a)\|^2_{\mathcal{H}_0}= \int_{\R}\eta_0|\phi_0(\cdot - a(t)) -  \phi_0(\cdot-a)|^2 \\
	& \hspace{5.5cm} +  \int_{\R}|\phi'_0(\cdot - a(t)) -  \phi'_0(\cdot-a)|^2\\	
	&= |a(t)-a|^2 \int_{\R}\eta_0|\phi'_0(\cdot -\nu_1a + (1-\nu_1)a(t))|^2\\
	& \hspace{2.5cm} +   |a(t)-a|\int_{\R}|\phi'_0(\cdot - a(t)) -  \phi'_0(\cdot-a)|\,|\phi''_0(\cdot -\nu_2a + (1-\nu_2)a(t)|\\
    & \lesssim |a(t)-a|\big(\|\eta_0\|_{L^1}|a(t)-a| + 2\|\phi'_0\|_{L^1}\big)\\
    &\lesssim  |a(t)-a|^2 + |a(t)-a|.
	\end{split}
	\ee
	
	\medskip
	\noindent 
    Combining \eqref{estimCorProp2-01} and \eqref{estimCorProp2-02} we have
    \be\label{estimCorProp2-03}
    \|\phi_{c(t), a(t), \theta(t)} -  \phi_{0, a, \theta}\|^2_{\mathcal{H}_0}\le K\big( c(t)^2 + |a(t)-a|^2 + |a(t)-a| + |\theta(t) - \theta|^2 \big),
	\ee
	for some universal constant $K$  for all $|c|< \cf$.

	\medskip
	Now, from \eqref{estimCorProp2} and \eqref{estimCorProp2-03}, and using that $\epsilon <1$, one gets
	\begin{multline}\label{3estimCorProp2}
		\|\zf(t, \cdot)\|_{\mathcal{H}_0} =  \|u(t, \cdot)- \phi_{c(t), a(t), \theta(t)}\|_{\mathcal{H}_0}\leq 
		\|u(t, \cdot) - \phi_{0, a, \theta}\|_{\mathcal{H}_0}\\
		+ \|\phi_{c(t), a(t), \theta(t)} -  \phi_{0, a, \theta}\|_{\mathcal{H}_0}\leq \big(1 + \sqrt{K}(K_0 +\sqrt{K_0})\big) \sqrt{\epsilon}, 
	\end{multline}
	which yields the second estimate in \eqref{desigZ} with $A_0=1 + \sqrt{K}(K_0 +\sqrt{K_0})$.
\end{proof}

\medskip 
Now, we will determine the growth in time of the modulation parameters $c(t),a(t)$ and $\theta(t)$ for any $t\in[-T,T]$. We will first show the evolution equation 
satisfied by the perturbation $$\zf(t)\equiv \zf(t,\cdot)=e^{-i\theta(t)}u(t,\cdot + a(t)) - \phi_{c(t)}(\cdot).$$

\begin{lemma}[Evolution equation for $\zf$]\label{prop3a}
	Let $\zf(t)=e^{-i\theta(t)}u(t,\cdot + a(t)) - \phi_{c(t)}(\cdot)$ the perturbation of the dark soliton profile  $\phi_{c(t)}(\cdot)$ 
	\eqref{darkprofile}. Then we have that
	\be\label{fluxperturbationZ}
	\begin{aligned}
		\partial_t\zf(t):=& - c'(t)\partial_c\phi_{c(t)} - i\theta'(t)(\phi_{c(t)} + \zf(t)) + a'(t)(\partial_x\phi_{c(t)} + \partial_x\zf(t)) + iZ(t),
	\end{aligned}
	\ee
	\noindent
	with
	\be\label{bigZ}
	Z(t):=\partial_{xx}\zf(t) + ic(t)\partial_x\phi_{c(t)} + \eta_{c(t)\zf}(t) - (\rho_{c(t)}^2  + 2|\phi_{c(t)}|^2\rho_{c(t)})(\phi_{c(t)} + \zf(t)).
	\ee
	\noindent
	and $\rho_{c(t)}=\rho_{c(t)}(\zf(t))$.
\end{lemma}
\begin{proof}
	First consider the explicit time derivative of $\zf(t,\cdot)$:
	\[
	\partial_t \zf(t) = - c'(t)\partial_c\phi_{c(t)} - i\theta'(t)(\phi_{c(t)} + \zf(t)) + a'(t)(\partial_x\phi_{c(t)} + \partial_x\zf(t)) 
	+ e^{-i\theta(t)}\partial_tu(t,\cdot + a(t)).
	\]
	\noindent
	Now computing the last term $\partial_tu(t,\cdot + a(t)),$ bearing in mind that $u$ fulfills \eqref{5gp}, and also \eqref{peso-blackdark}, \eqref{rho-c} and that 
	\[|u|^4 = |\phi_{c(t)}|^4 + \rho_{c(t)}^2 + 2\rho_{c(t)}|\phi_{c(t)}|^2,\]
	a direct calculation gives us \eqref{fluxperturbationZ}.
\end{proof}

We now look for an expression for the growth in time of the modulation parameters $c(t),~a(t)$ and $\theta(t)$. In order to do that, we resort to the 
continuity of the quintic Gross-Pitaevskii flow in $\mathcal{Z}(\R)\subset\mathcal{H}(\R)$. Specifically, if an initial data $u_0$ is chosen such that 
$d_0(u_0,\phi_0)<\alpha,$ then we get a time $T$ such that the corresponding solution $u(t, \cdot)$ along the  quintic Gross-Pitaevskii  flow belongs to 
$\mathcal{V}_{0}(\alpha)$ \eqref{orbitV0}, for any $t\in[-T,T]$.

\medskip 
We will see in Section \ref{Sec5}, that we can fix the smallness parameter $\alpha$ in such a way that the solution $u(t, \cdot)$
of the Cauchy problem \eqref{5gp} still belongs to $\mathcal{V}_{0}(\alpha)$ for all $t\in\R$.

\begin{proposition}[Estimates on the growth of the modulation parameters]\label{prop3}
	There exist numbers $\alpha_1>0$ and  $A_1(\alpha_1)>0$ such that if the solution $u(t,\cdot)$ 
	lies in $ \mathcal{V}_{0}(\alpha_1)$ for any $t\in[-T,T]$, then the functions $c,~a$ and $\theta$ are $C^1\big([-T,T]; \R\big)$ and satisfy 
	
	\be\label{speedrateCAZ}
	|c'(t)| + |a'(t)| + |\theta'(t)| \leq A_1^2(\alpha_1)\|\zf(t, \cdot)\|_{\mathcal{H}_0},
	\ee
	
	\noindent
	for all $t\in[-T,T]$.
\end{proposition}

\begin{proof}
	We differentiate with respect to time the three  generalized orthogonality conditions \eqref{genortogonalityDark} for perturbations around the dark soliton 
	profile $\phi_{c(t)}$. 

	\medskip
	Since we need to compute the derivatives in time of the orthogonality conditions, we initially consider  regular enough initial data,
	for example $\partial_x{u_0} \in H^2(\R)$. In fact with this regularity we can justify \eqref{derOrto1}, \eqref{derOrtot2} and \eqref{derOrtot3} below. 
	We consider initially $\alpha$ and $u_0$ as in Corollary \ref{0stParProp2} so that the solution  $u(t,\cdot)\in {\color{black}\mathcal{U}_{0}}(\alpha)$ for all $t\in[-T,T]$ and then we can set the modulation parameters  $(c(t),a(t),\theta(t))\in(-\cf, \cf)\times\R^2$ for any $t\in[-T,T]$.

    \medskip 
	Note that $c,a$  and $\theta$ belong to $C^1\big([-T,T], \R\big)$ by the Chain Rule Theorem and moreover note that
	$\zf(t)\in C^1([-T,T],\mathcal{H}(\R))$, and therefore we can get \eqref{fluxperturbationZ}. 
	Therefore, derivating the first orthogonality condition in 
	\eqref{genortogonalityDark}, and with the notation $m_{ij},~i,j=1,2,3,$ for integrals independent of $\zf(t)$ and $n_k,~k=1,\dots,9$ 
	for integrals with terms depending on $\zf(t),$ we get
	\be\label{derOrto1}
	\begin{aligned}
		&\partial_t\int_\R\langle \eta_{c(t)}, \zf(t)\rangle_\C\\
		&\quad =  \int_\R\Big(\langle c'(t)\partial_c\eta_{c(t)}, \zf(t)\rangle_\C 
		+ \langle \eta_{c(t)}, c'(t)\partial_c\zf(t)\rangle_\C + \langle \eta_{c(t)}, \partial_t\zf(t)\rangle_\C \Big) \\
		&\quad = \int_\R\Big(\langle c'(t)\partial_c\eta_{c(t)}, \zf(t)\rangle_\C 
		+ \langle \eta_{c(t)}, c'(t)\partial_c\zf(t)\rangle_\C \\
		&\hspace{1cm}  + \langle \eta_{c(t)}, \left(- c'(t)\partial_c\phi_{c(t)} - i\theta'(t)(\phi_{c(t)} + \zf(t)) 
		+ a'(t)(\partial_x\phi_{c(t)} + \partial_x\zf(t)) + iZ(t)\right)\rangle_\C\Big).
	\end{aligned}
	\ee
	Thus, 
	\be\label{derOrto1a}
	\begin{aligned}
		\partial_t\int_\R&\langle \eta_{c(t)}, \zf(t)\rangle_\C\\
		&=a'(t)\Big(\int_\R\langle \eta_{c(t)}, \partial_x\phi_{c(t)}\rangle_\C + \int_\R\langle \eta_{c(t)}, \partial_x\zf(t)\rangle_\C\Big)\\ 
		&\hspace{1cm}+ c'(t)\Big(-\int_\R\langle \eta_{c(t)}, \partial_c\phi_{c(t)}\rangle_\C +\int_\R\langle \partial_c\eta_{c(t)}, \zf(t)\rangle_\C 
		+\int_\R\langle \eta_{c(t)}, \partial_c\zf(t)\rangle_\C  \Big)\\ 
		&\hspace{1cm}+ \theta'(t)\Big(\int_\R\langle \eta_{c(t)}, -i\phi_{c(t)}\rangle_\C  + \int_\R\langle \eta_{c(t)}, -i\zf(t)\rangle_\C\Big) 
		+ \int_\R\langle \eta_{c(t)}, iZ(t)\rangle_\C\\
		&=a'(t)(m_{11} + n_1)+ c'(t)(n_2 - m_{12}) + \theta'(t)(m_{13} + n_3) + \int_\R\langle \eta_{c(t)}, iZ(t)\rangle_\C =0.
	\end{aligned}
	\ee
	
	\medskip
	\noindent
	Now, we {\color{black} differentiate} the second orthogonality condition in \eqref{genortogonalityDark}, and we obtain
	
	\be\label{derOrtot2}
	\begin{aligned}
		\partial_t\int_\R&\langle i\eta_{c(t)}, \zf(t)\rangle_\C \\
		&=a'(t)\Big(\int_\R\langle i\eta_{c(t)}, \partial_x\phi_{c(t)}\rangle_\C + \int_\R\langle i\eta_{c(t)}, \partial_x\zf(t)\rangle_\C\Big)\\ 
		&\hspace{1cm}+ c'(t)\Big(-\int_\R\langle i\eta_{c(t)}, \partial_c\phi_{c(t)}\rangle_\C +\int_\R\langle i\partial_c\eta_{c(t)}, \zf(t)\rangle_\C 
		+\int_\R\langle i\eta_{c(t)}, \partial_c\zf(t)\rangle_\C  \Big)\\ 
		&\hspace{1cm}+ \theta'(t)\Big(\int_\R\langle i\eta_{c(t)}, -i\phi_{c(t)}\rangle_\C  + \int_\R\langle i\eta_{c(t)}, -i\zf(t)\rangle_\C\Big) 
		+ \int_\R\langle i\eta_{c(t)}, iZ(t)\rangle_\C\\
		&=a'(t)(m_{21} + n_4)+ c'(t)(n_5 - m_{22}) + \theta'(t)(m_{23} + n_6) + \int_\R\langle i\eta_{c(t)}, iZ(t)\rangle_\C=0.
	\end{aligned}
	\ee
	
	\medskip
	\noindent
	Finally, we  {\color{black} differentiate} the third orthogonality condition in \eqref{genortogonalityDark}, and we get
	
	\be\label{derOrtot3}
	\begin{aligned}
		\partial_t&\int_\R\langle iR_{c(t)}\eta_{c(t)}, \zf(t) \rangle_\C \\
		& = \int_\R\langle iR_{c(t)}\eta_{c(t)},\partial_t\zf(t)\rangle_\C +\int_\R\langle iR_{c(t)}\eta_{c(t)},c'(t)\partial_c\zf(t)\rangle_\C \\
		& \hspace{2cm}+  \int_\R\langle ic'(t)\partial_cR_{c(t)}\eta_{c(t)} +    ic'(t)R_{c(t)}\partial_c\eta_{c(t)},  \zf(t)\rangle_\C \\
		&=a'(t)\Big(\int_\R\langle  iR_{c(t)}\eta_{c(t)}, \partial_x\phi_{c(t)}\rangle_\C + \int_\R\langle iR_{c(t)}\eta_{c(t)}, \partial_x\zf(t)\rangle_\C\Big)\\ 
		& \quad + c'(t)\Big(\int_\R\langle i\partial_c(R_{c(t)}\eta_{c(t)}), \zf(t)\rangle_\C 
		+ \int_\R\langle iR_{c(t)}\eta_{c(t)},\partial_c\zf(t)\rangle_\C  -  \langle iR_{c(t)}\eta_{c(t)}, \partial_c\phi_{c(t)}\rangle_\C\Big)\\
		& \quad + \theta'(t)\Big(\int_\R\langle iR_{c(t)}\eta_{c(t)}, -i\phi_{c(t)}\rangle_\C
		+ \int_\R\langle iR_{c(t)}\eta_{c(t)}, -i\zf(t)\rangle_\C\Big) + \int_\R\langle iR_{c(t)}\eta_{c(t)}, iZ(t)\rangle_\C\\
		&=a'(t)(m_{31} + n_7)+ c'(t)(n_8 - m_{32}) + \theta'(t)(m_{33} + n_9) + \int_\R\langle iR_{c(t)}\eta_{c(t)}, iZ(t)\rangle_\C = 0.
	\end{aligned}
	\ee
	
	\medskip
	
	Gathering all three previous derivatives, we obtain the {\color{black} following  linear system}
	
	\be\label{matricialSystem}
	\mathcal{M}(c,\zf)\begin{pmatrix}a'(t)\\c'(t)\\ \theta'(t)\end{pmatrix} = \mathcal{B}(c,\zf),
	\ee
	
	\medskip
	\noindent
	with the matrix $\mathcal{M}(c,\zf)$ defined by
	\be\label{matrixM2}
	\mathcal{M}(c,\zf):=\begin{pmatrix}m_{11} + n_1& n_2 - m_{12}&m_{13} + n_3\\
		m_{21} + n_4& n_5  - m_{22}&m_{23} + n_6\\
		m_{31} + n_7&n_8 - m_{32}&m_{33} +  n_9\end{pmatrix},
	\ee
	\medskip
	\noindent
	and the matrix $\mathcal{B}(c,\zf)$ is written as follows:
	
	\be\label{matrixB0}
	\mathcal{B}(c,\zf):=\begin{pmatrix}-\int_\R\langle \eta_{c(t)}, iZ(t)\rangle_\C \\
		-\int_\R\langle i\eta_{c(t)}, iZ(t)\rangle_\C \\
		-\int_\R\langle iR_{c(t)}\eta_{c(t)}, iZ(t)\rangle_\C
	\end{pmatrix},
	\ee
	
	\medskip
	\noindent
	with $Z(t)$ as in \eqref{bigZ}. See Appendix \ref{Ap1} for a full expression of the computed matrix elements $m_{i,j},~i,j=1,2,3$.
	
	\bigskip
	\noindent
	Note that, in the case of null perturbation in \eqref{matrixM2}, and considering the limit case of  {\color{black}$c=0$}, it turns out that 
	$\mathcal{M}({\color{black}0},0)$ has a nonvanishing determinant, namely
	
	\[
	\det\mathcal{M}({\color{black}0},0)=\frac{8}{5}E_2[\phi_0],
	\]
	
	\noindent
	and therefore  $\mathcal{M}({\color{black}0},0)$ is invertible. By using a continuity argument,
	we can select a small enough parameter $\alpha_1<\alpha$ such that for small speeds and  perturbations $(c,\zf),$ the matrix $\mathcal{M}(c,\zf)$ 
	is still invertible.  In fact, having in mind the Neumann Series Theorem, it is enough  to consider $(c,\zf)$ verifying
	\[ \|\mathcal{M}(c,\zf)-\mathcal{M}({\color{black}0},0)\|_{M_{2\times2}(\mathbb{C})}\le \alpha_1 < \|\mathcal{M}^{-1}({\color{black}0},0)\|^{-1}_{M_{2\times2}(\mathbb{C})}. \]
	Namely, choosing $\alpha_1<\alpha$ small enough such that $u(t, \cdot)\in{\color{black}\mathcal{U}_{0}}(\alpha_1),$ for all $t\in[-T,T]$, and therefore
	that from \eqref{desigZ} in Proposition \ref{prop2a}, it holds 
	
	\be\label{estimateZmatrix}
	\|\zf(t, \cdot)\|_{\mathcal{H}_0} + |c(t)| \leq A_0\alpha_1,
	\ee
	
	\medskip 
	\noindent
	with $\det\mathcal{M}(c,\zf)\neq0$  and, consequently, the operator norm of its inverse is bounded by some positive number $A_1(\alpha_1)$. 
	In the same way, the r.h.s. of \eqref{matricialSystem} is bounded as follows:
	
	\[
	\|\mathcal{B}(c,\zf)\|_{\R^3}\leq  A_1(\alpha_1)\|\zf(t, \cdot)\|_{\mathcal{H}_0},
	\]
	\noindent
	for a suitable choice of the constant $A_1(\alpha_1)$. Therefore, from \eqref{matricialSystem}, we finally get that
	
	\be\label{estimateParameters}
	|a'(t)| + |c'(t)| + |\theta'(t)| \leq \big |\mathcal{M}(c,\zf)^{-1}\cdot \mathcal{B}(c,\zf) \big| \leq A_1^2(\alpha_1) \|\zf(t, \cdot)\|_{\mathcal{H}_0},
	\ee
	
	\medskip 
	\noindent 
	for all $t\in[-T,T]$.
	
	\medskip 
	Finally, we extend the above estimate \eqref{estimateParameters} for general initial data $u_0\in\mathcal{Z}(\R)$. In fact,
	the flow of \eqref{5gp} is continuous with respect to initial data in $\mathcal{Z}(\R)$ (see\cite{Gallo}).
	Moreover, from the continuity of the modulation  parameters $c(t),a(t)$ and $\theta(t)$, we have that the matrices $\mathcal{M}(c,\zf)$ and $\mathcal{B}(c,\zf)$ depend continuously on  
	$u\in \mathcal{H}(\R)$. Therefore, since the matrix $\mathcal{M}(c,\zf)$ is invertible with an operator norm of its inverse depending on $\alpha_1$, we can use
	a standard density argument to extend \eqref{matricialSystem} 
	to a general solution. Therefore we get the continuous differentiability property of the modulation parameters $c(t),a(t)$ and $\theta(t)$, and we obtain the 
	corresponding estimates  \eqref{speedrateCAZ} from \eqref{matricialSystem}.
\end{proof}

\section{Proof of the Main Theorem}\label{Sec5}
In this section we  prove a detailed version of Theorem \ref{teorema1a}.

\begin{theorem}[Orbital stability of the black soliton]\label{teoremaMAIN}
	Let $\phi_0$ be the black soliton \eqref{black5gp} of the quintic GP equation \eqref{5gp}.  
	Given $\epsilon>0$  there exists $\delta(\epsilon)>0$ and a positive constant $A_*$ such that if  the initial data  $u_0$ verifies
	\[
	u_0\in\mathcal{Z}(\R)\quad \quad \text{and}\quad \quad d_0(u_0,\phi_0)<\delta(\epsilon), 
	\]
	\noindent 
	then there exist functions  $a,\theta  \in C^1(\R,\R)$ such that the solution $u$ of the Cauchy problem  for the quintic GP 
	equation \eqref{5gp}, with initial data $u_0$, satisfies
	
	\begin{equation}\label{defstability}
		d_0(e^{-i\theta(t)}u(t,\cdot + a(t)),\phi_0)< \epsilon
	\end{equation}
\noindent
	and
	\noindent
	\be\label{speedScalings}
	|a'(t)|+|\theta'(t)|< A_*\epsilon
	\ee
for any $ t\in \R$.

\end{theorem}
\begin{remark}\label{finalRemarkTeo}
{\color{black}With respect to the cubic case \cite{G-Smets}, a difference appears in the proof of the orbital stability Theorem for black solitons of \eqref{5gp},  that is, we could not achieve a lipschitzian control of the metric, i.e. 
\be\label{finalRemarkTeo-01}
d_0(e^{-i\theta(t)}u(t,\cdot + a(t)),\phi_{c(t)}) \lesssim d_0(u_0, \phi_0),\quad \text{for all}\quad c\in(0,\cf).
\ee

 The main reason is that if the momentum \cite[p.313,~$(1.27)$]{G-Smets} is used in our problem, a linear term 
 \be\label{linealterm}
 \int_\R\langle i\phi'_{c(t)}, \zf\rangle_\C,
 \ee
 \noindent
 appears when expanding it around dark solitons $\phi_{c(t)}$. Unfortunately this term does not match with any orthogonality relation \eqref{condittion-cor1} and it can not be bounded from above nor controlled in the right and proper way. This is a big difference with respect to the cubic GP case (see p.314,l.7 in \cite{G-Smets}) which allows them to get an upper bound  on the speed $c(t)$ as shown in p.314,l.-3  in \cite{G-Smets}.}

%
In our case, instead \eqref{finalRemarkTeo-01}  we get
	\be\label{finalRemarkTeo-02}
	d_0(e^{-i\theta(t)}u(t,\cdot + a(t)),\phi_{c(t)}) \lesssim d_0(u_0, \phi_0)+\cf.
	\ee
	{\color{black}Moreover, and again  in view of this technical issue with the linear term \eqref{linealterm}, we decided to change this uniform control on $c(t)$ on a fixed speed interval $(-\cf,\cf)$ by using the following strategy: for each fixed $\varepsilon>0$ we choose a suitable interval $(0,\cf(\varepsilon))$ which allows us to select initial data in an appropiate ball with center $\phi_0$ in the $d_0$ metric and such that the solution remains in the $\varepsilon$-neighborhood for all time by a bootstrap argument (note that \eqref{defstability} is not as strong as the corresponding statement in \cite[p.308,~$(1.9)$]{G-Smets}).  Obviously with this approach, once we reduce $\varepsilon$, the speed interval $(-\cf,\cf)$ can also be reduced. As a consequence of our approach, we lose any possibility to say something about the orbital stability of the dark soliton solution. 
	}
%
\end{remark}

\begin{proof}[Proof of theorem \ref{teoremaMAIN}] In order to simplify the explanation we show the  proof for $t\ge 0$.

	\medskip 
	Let $\alpha$,\,  $A_0$ and $\alpha_1$ as in Corollary \ref{1stParProp2}, Corollary \ref{2ndParProp2} and  Proposition \ref{prop3}, respectively. Consider now $\epsilon >0$ such that 
		\be\label{MainTh-proof-01}
		0< \epsilon\leq \min\{ 1,\, \alpha\},\quad 0< \epsilon < \alpha_1\quad  \text{and}\quad  A_0\epsilon\ll 1.
		\ee
		Firstly, we take $u_0\in \mathcal{Z}(\R)$ \eqref{Ec} such that $d_0(u_0, \phi_0)< \epsilon/2$ and $u \in C(\R, \mathcal{Z})$ the corresponding solution to \eqref{5gp}. 
		
		\medskip 
		 
		Now we define 
		\be\label{MainTh-proof-01-timestar}
		T^*:= \sup\limits\Big\{T>0: \forall\,   t\in [0, T],\, \inf\limits_{(\iota, b)\in \R^2}d_0(e^{-i\iota}u(t,\cdot + b),\,\phi_0) < \epsilon\,\Big\}\ee
		and the idea is to use a contradiction argument under the assumption $T^*<\infty$ when $d_0(u_0, \phi_0)$ is small enough.
	
		\medskip 
		Note that since $d_0(u_0, \phi_0)< \epsilon$, then as a direct consequence of the continuity of the quintic GP flow in $\cZ(\R)$ with respect to the metric $d_0$, we can find $T_0>0$ such that  
		\[d_0(u(t,\cdot ), \phi_0)< \epsilon,\; \text{for all}\; t\in [0, T_0],\]
		which, in particular, implies that $T^*$ is well-defined in \eqref{MainTh-proof-01-timestar}. Furthermore, 
		
		\be\label{MainTh-proof-01a}
	    u(t, \cdot) \in \mathcal{V}_{0}(\epsilon) \subset {\color{black}\mathcal{U}_{0}}(\epsilon)\subset {\color{black}\mathcal{U}_{0}}(\alpha),
	    \ee
		
		\noindent 
		for all $t\in [0, T^*)$.
		
		\medskip
	    Consider now the functions  $c(t), a(t), \theta(t)$  given  in Corollary \ref{0stParProp2} and notice that, in view of \eqref{MainTh-proof-01a}, we can consider these functions defined on the whole interval $[0, T^*)$.
		
		\medskip 
		Now suppose that $T^*<+\infty$ and consider
		\[\zf(t, \cdot)=e^{-i\theta(t)}u(t,\cdot + a(t)) - \phi_{c(t)}(\cdot),\quad t\in [0, T^*),\] 
		\noindent
		where $\big(c(t), a(t), \theta(t)\big) \in  (-\cf, \cf)\times \R^2$. 
		
        \medskip 
		{\color{black}Then, having in mind the global theory in \cite{Gallo} which guarantees that $\|\rho_c(\zf)\|_{L^2}$ verifies \eqref{condittion-cor1},} using the coercivity of $E_2$ \eqref{E2} around the dark soliton \eqref{genexpanE2} in Proposition \ref{anticor1}  and   Corollary \ref{2ndParProp2} with $(a,\theta)= (0,0)$, we obtain
		\begin{equation}\label{MainTh-proof-02}
			\begin{split}
				\|\zf(t, \cdot)\|_{\mathcal{H}_0}^2 + &\|\phi_{c(t)}^3\rho_{c(t)}(\zf(t, \cdot))\|^2_{L^2}
				\leq  \frac{1}{\tilde{\Gamma}^2}\Big(\tilde{\Gamma}\big(E_2[\phi_{c(t)} + \zf] - E_2[\phi_0]\big) + c^2(t) + \|\zf(t, \cdot)\|_{\mathcal{H}_0}^3\Big)\\
				&=  \frac{1}{\tilde{\Gamma}^2}\Big(\tilde{\Gamma}
				\big(E_2[e^{-i\theta(t)}u(t,\cdot + a(t))] - E_2[\phi_0]\big) + c^2(t) + A_0\sqrt{\epsilon}\|\zf(t, \cdot)\|^2_{\cH_0}\Big).
			\end{split}
		\end{equation}	
		\noindent
		Selecting $\epsilon$ such that $A_0\sqrt{\epsilon} < \frac{1}{2}\tilde{\Gamma}^2$ and using the conservation of the $E_2$ energy \eqref{E2} one gets
		\begin{equation}\label{MainTh-proof-03}
			\|\zf(t, \cdot)\|_{\mathcal{H}_0}^2 + \|\phi_{c(t)}^3\rho_{c(t)}(\zf(t, \cdot))\|^2_{L^2} \leq 
			\frac{2}{\tilde{\Gamma}^2}\Big(\tilde{\Gamma}\big(E_2[u_0] - E_2[\phi_0]\big) + c(t)^{2}\Big),
		\end{equation}
		for all $t\in [0,T^*)$. 
		
		\medskip 
		Now, from the expansion \eqref{E2expanBlack} with $\zf = u_0 - \phi_0$  there exists a positive constant $\tilde{k}_1$ such that
		\be\label{MainTh-proof-04}
		E_2[u_0] - E_2[\phi_0]\le \tilde{k}_1d_0^2(u_0, \phi_0),
		\ee
		with $\tilde{k}_1$ independent of $u_0$. Then,  putting this estimate into \eqref{MainTh-proof-03}  and using \eqref{black-dark-comparison} we have that there exists {\color{black}a universal positive constant $\widetilde{K},$ non depending on $\cf,$} such that 
		
		\begin{equation}\label{newCoercivityDARK-a}
			\begin{split}
			\|\zf(t, \cdot)\|_{\mathcal{H}_0}^2 + \|\phi_0^3\rho_{c(t)}(\zf(t, \cdot))\|^2_{L^2}
			& \le 	\widetilde{K}\big(\|\zf(t, \cdot)\|_{\mathcal{H}_0}^2 + \|\phi_{c(t)}^3\rho_{c(t)}(\zf(t, \cdot))\|^2_{L^2}\big)\\ 
			&\leq\frac{2\widetilde{K}}{\tilde{\Gamma}^2}\Big(\tilde{\Gamma}\tilde{k}_1 d_0^2(u_0, \phi_0) + c^2(t) \Big)\\
			&\le \frac{2\widetilde{K}\tilde{k}_1}{\tilde{\Gamma}}d_0^2(u_0, \phi_0) +  \frac{2\widetilde{K}}{\tilde{\Gamma}^2}\cf^2.
			\end{split}
		\end{equation}
	
	\medskip
	\noindent
     So,  we have 
    \begin{equation*}\label{newCoercivityDARK-b}
	d_0(e^{-i\theta(t)}u(t,\cdot + a(t)),\phi_{c(t)}) =	\Big(\|\zf(t, \cdot)\|_{\mathcal{H}_0}^2 + \|\phi_0^3\rho_{c(t)}(\zf(t, \cdot))\|^2_{L^2} \Big)^{\frac12}\le \Big( \frac{2\widetilde{K}\tilde{k}_1}{\tilde{\Gamma}}d_0^2(u_0, \phi_0) +  \frac{2\widetilde{K}}{\tilde{\Gamma}^2}\cf^2\Big)^{\frac12}
    \end{equation*}
	and hence from \eqref{d01} we have 
	\begin{equation}\label{newCoercivityDARK-c}
	\begin{split}
		d_0(e^{-i\theta(t)}u(t,\cdot + a(t)),\phi_0) & \le \Big( \frac{2\widetilde{K}\tilde{k}_1}{\tilde{\Gamma}}d_0^2(u_0, \phi_0) +  \frac{2\widetilde{K}}{\tilde{\Gamma}^2}\cf^2\Big)^{\frac12} + d_0(\phi_{c(t)}, \phi_0)\\
		&\le  \Big(\frac{2\widetilde{K}\tilde{k}_1}{\tilde{\Gamma}}\Big)^{\frac12}d_0(u_0, \phi_0) +  \frac{\sqrt{2\widetilde{K}}}{\tilde{\Gamma}}\cf + \tilde{k}_2\cf,
	\end{split}
	\end{equation}
	for some positive constant $\tilde{k}_2$.  Now we reduce $\cf$, if necessary, to hold
    \begin{equation}\label{newCoercivityDARK-d}
    \Big(\frac{\sqrt{2\widetilde{K}}}{\tilde{\Gamma}}+ \tilde{k}_2\Big)\cf< \frac{\epsilon}{4}
    \end{equation}
    \noindent
    and we also consider $u_0$ satisfying
    \begin{equation}\label{newCoercivityDARK-e}
    d(u_0, \phi_0)< \min \bigg\{ \frac{\epsilon}{2},\, \frac{\epsilon}{4}\Big(\frac{\tilde{\Gamma}}{2\widetilde{K}\tilde{k}_1}\Big)^{\frac12}\bigg\}.
    \end{equation}
    
    \medskip
    \noindent 
	Then, combining \eqref{newCoercivityDARK-c}, \eqref{newCoercivityDARK-d} and \eqref{newCoercivityDARK-e}, we get 
	\be\label{finalEstim}
	d_0(e^{-i\theta(t)}u(t,\cdot + a(t)),\phi_0) < \frac{\epsilon}{2} ,\quad \text{for all}\quad t\in [0, T^*), 
	\ee
	\noindent
	which contradicts the definition of $T^*< \infty$ in \eqref{MainTh-proof-01-timestar}, due to the continuity of the flow of  the solution $u(t, \cdot)$ with respect to the metric $d_0$. Then, $T^*=\infty$ and the proof of \eqref{defstability} is finished.

	\medskip 
	Finally, from Proposition \ref{prop3} one gets 
	\[
	\sup_{t\in\R}|a'(t)|+|\theta'(t)|<  A_*\epsilon,
	\]
	\noindent
    for some positive constant $A_*$.  This completes the proof of Theorem \ref{teoremaMAIN}.
\end{proof}


\appendix

	\section{Proof of \eqref{darksoliton}}\label{AppenDarkSol}
	In order to prove that \eqref{darksoliton} is a solution of \eqref{edodark}, we propose a suitable ansatz \eqref{candidate}:
	
	\be\label{AppDS0}
	{\color{black}\Phi_c}(\xi)=\frac{ia_1 + a_2\tanh(k \xi)}{\sqrt{1+ a_3\tanh^2(k \xi)}},\quad \xi=x-ct.
	\ee
	
	\noindent
	This ansatz {\color{black} must reduce} to the black solution \eqref{black5gp} when $c=0$, therefore this implies that $a_1$ has to be dependent on $c$ in some way. We make the following selection for 
	\[
	a_1=c\tilde{a}_1a_2,
	\]
	\noindent
	with $\tilde{a}_1$ to be determined. Hence, we recast \eqref{AppDS0} as follows:
	
	\be\label{AppDS1}
	{\color{black}\Phi_c}(\xi)=\frac{ic\tilde{a}_1a_2+a_2\tanh(k\xi)}{\sqrt{1+a_3\tanh(k\xi)^2}},
	\ee
	\noindent
	where $\tilde{a}_1,k,a_2,a_3$ are parameters to be determined imposing that \eqref{AppDS1} is a solution of \eqref{edodark}. Therefore, substituting \eqref{AppDS1} into \eqref{edodark} and simplifying (here $X=\tanh(k\xi),~~D=1 + a_3 \tanh(k\xi)^2$), we get

	\be\label{AppDS2}
	\begin{aligned}
		&{\color{black}\Phi_c}'' -ic{\color{black}\Phi_c}' + (1-|{\color{black}\Phi_c}|^4){\color{black}\Phi_c}
		=\frac{a_2}{D^{5/2}}\sum\limits_{i=0}^5r_iX^i,
	\end{aligned}
	\ee
	\noindent
	where $r_i,~i=0,\dots,5,$ are the following complex coefficients
	
	\be\label{AppDS3}\begin{aligned}
		&r_0=-ic(k + \tilde{a}_1 k^2 a_3 + \tilde{a}_1^5 c^4 a_2^4 - \tilde{a}_1),\\
		&r_1=(-\tilde{a}_1 c^2 k a_3 + k^2 (-3a_3-2) - \tilde{a}_1^4 c^4 a_2^4 + 1),\\
		&r_2=-i c(2 \tilde{a}_1 k^2(-a_3-2)a_3 - k (1 -a_3) + 2 \tilde{a}_1^3 c^2 a_2^4 -  2 \tilde{a}_1 a_3),\\
		&r_3=-(-\tilde{a}_1 c^2 k a_3 (1 - a_3) + k^2 (-4a_3-2) + 2 \tilde{a}_1^2 c^2 a_2^4 - 2 a_3),\\
		&r_4=-ic(\tilde{a}_1a_2^4-\tilde{a}_1a_3^2-a_3k+\tilde{a}_1a_3(3+2a_3)k^2),\\
		&r_5=(-a_2^4 - a_3(k^2 - \tilde{a}_1 c^2 k a_3 - a_3)).\\                                                                            
	\end{aligned}\ee
	\noindent
	
	Now, we impose that 
	\be\label{AppDS4}
	r_i=0,~~\forall i=0,\dots5,
	\ee
	and look for non trivial solutions (i.e. $\phi\neq0$). Starting with the last equation $r_5=0$, we get 
	
	\be\label{r2}
	a_2^4 =  a_3 (-k^2 + \tilde{a}_1 c^2 k a_3 + a_3).
	\ee
	\noindent
	Substituting the above value for $a_2^4$ into system \eqref{AppDS4}, we get that the equation  $r_4=0$ is solved for 
	
	\be\label{r}
	a_3 = -\frac{-1+2k\tilde{a}_1}{\tilde{a}_1(2k+c^2\tilde{a}_1)}.
	\ee
	
	\noindent
	Therefore, with these values for $a_2^4,a_3,$ the group of  \eqref{AppDS3} is recasted as follows ($H=\frac{1+c^2\tilde{a}_1^2}{(2k+c^2\tilde{a}_1)^2}$)
	
	\be\label{AppDS5}\begin{aligned}
		&r_0=-ickHM_0=0,\\
		&r_1=-\frac{kH}{\tilde{a}_1}M_0=0,\\
		&r_2=-ickHM_1=0,\\
		&r_3=-\frac{kH}{\tilde{a}_1}M_1=0,\\
		&r_4=0,\\
		&r_5=0,\\                                                                            
	\end{aligned}\ee
	\noindent
	with 
	\be\label{AppDS6}
	\begin{aligned}
		&M_0=6 k^2 + 6 \tilde{a}_1^4 c^4 k^2 + \tilde{a}_1 k (5 c^2 - 4 (k^2 + 1)) + \tilde{a}_1^3 c^2 k (-5 c^2 + 4 (k^2 + 1))\\
		&\qquad + \tilde{a}_1^2(c^4 - 4 c^2 (2 k^2 + 1)),\\
		&M_1=-8 k^2 + 8 \tilde{a}_1 k^3 + c^2 (1 - 10 \tilde{a}_1 k + 12 \tilde{a}_1^2 k^2) - 4  +   8 \tilde{a}_1 k.
    \end{aligned}\ee
	
	\medskip
	\noindent
	Solving  
	\[
	M_1=0,
	\]
	\noindent
	for $\tilde{a}_1,$ we get (selecting e.g. the $+$ root)
	
	\be\label{AppA}
	\tilde{a}_1= \frac{\left(5 c^2-4\right)-4 k^2 +\sqrt{13 c^4+8 c^2 \left(7 k^2+1\right)+16 \left(k^2+1\right)^2}}{12 c^2 k}.
	\ee
	Now, rewriting $M_0$ \eqref{AppDS6} with $\tilde{a}_1$ as in \eqref{AppA}, we get 
	
	\be\label{newM0}\begin{aligned}
	&M_0=\frac{(4k^2+c^2-4)}{144c^2k^2}m_0(c,k),
	\end{aligned}	\ee
	\noindent
	with
\be\label{newM00}\begin{aligned}
	m_0(c,k)=	\Big[& 16-16 c^2+19 c^4+32 k^2+80 c^2 k^2+16 k^4\\
	&\qquad+\left(5 c^2-4 k^2-4\right) \sqrt{13 c^4+8 c^2 \left(7 k^2+1\right)+16 \left(k^2+1\right)^2}\Big].
\end{aligned}	\ee
	Finally, selecting 
		\be\label{Appd}
		k=\frac12\sqrt{4-c^2},
		\ee

\noindent
we get 

	\[
	M_0=0,
	\]
	
\noindent
and we have solved system \eqref{AppDS4}, and therefore \eqref{AppDS1} is a solution. Note that for these values of $\tilde{a}_1$ and $k$, the factor $H$ is well defined; in fact, $H=\frac{6 c^2+\left(3 c^2-4\right) \sqrt{3 c^2+4}+8}{c^2 \left(\sqrt{3 c^2+4}+4\right)^2}$. In order to compare this solution with \eqref{darksoliton}, we rewrite it as follows: firstly note that with this value of $k$, \eqref{AppA} and \eqref{r} reduce to
		
		\be\label{AppAnew}
		\tilde{a}_1=\frac{3 c^2-4+2 \sqrt{4+3 c^2}}{3 c^2 \sqrt{4-c^2}},
		\ee
		\noindent
		and 
		
		\be\label{rnew}
		a_3= -\frac{3 \left(4-c^2\right) \left(\sqrt{3 c^2+4}-2\right)}{\left(4+\sqrt{3 c^2+4}\right) \left(3 c^2-4+2 \sqrt{3 c^2+4}\right)},
		\ee
		\noindent
		and hence, from \eqref{r2}, with the above values of $\tilde{a}_1,k,a_3$, and simplifying, we get (taking for instance a real $+$ root)
		
		\be\label{r2new}\begin{aligned}
			a_2&=\frac{3 c \left(c^2-4\right)}{\sqrt{2} \sqrt{-18 c^4+\left(3 \sqrt{4-c^2} \sqrt{-3 c^4+8 c^2+16}+80\right) c^2+4 \left(\sqrt{4-c^2} \sqrt{-3 c^4+8 c^2+16}-8\right)}}\\
			&=\frac{3 c \sqrt{4-c^2}}{\sqrt{2} \sqrt{3 \left(\sqrt{3 c^2+4}+6\right) c^2+4 \left(\sqrt{3 c^2+4}-2\right)}}
			=\frac{3 c \sqrt{4-c^2}}{\sqrt{2}\sqrt{18 c^2-8+\left(3 c^2+4\right)^{3/2}}}.
		\end{aligned}\ee
		\noindent
		Therefore
		\[
		\sqrt{2} a_2=\mu_2.
		\]
		\noindent
		Now, from ansatz \eqref{AppDS1}, and values \eqref{AppAnew} and \eqref{r2new}, we get that
		
		\be\label{mu1new}\begin{aligned}
			c\tilde{a}_1a_2&=c\times \left(\frac{3 c^2+2 \sqrt{3 c^2+4}-4}{3 c^2 \sqrt{4-c^2}}\right)
			\times \left(\frac{3 c \sqrt{4-c^2}}{\sqrt{2}\sqrt{18 c^2-8+\left(3 c^2+4\right)^{3/2}}}\right)\\
			&=\frac{3 c^2-4+2 \sqrt{3 c^2+4}}{\sqrt{2}\sqrt{18 c^2-8+\left(3 c^2+4\right)^{3/2}}},
		\end{aligned}\ee
		\noindent
		and hence
		\[
		\sqrt{2} c\tilde{a}_1a_2=\mu_1.
		\]
		
		Finally note that
		
		\[
		k=\kappa,
		\]
		\noindent
		
		and with \eqref{mu1new} and \eqref{r2new}, we get 
		
		\[
		\frac{\mu_1^2+\mu_2^2}{2+2a_3}-1=0,
		\]
		\noindent
		and then
		\[
		a_3=\mu.
		\]
	
	\section{Proof of Lemma \ref{lemma-elemental-calculus}}\label{0a}
	The proof of this identity is made by quadratures. Making the change $s=\sqrt{2b}\tan \theta$ we get the following equalities for the indefinite integrals 
	\begin{equation}
		\label{Lemma-Calculus-Integral-a}
		\int\frac{ds}{(b-s^2)\sqrt{s^2+2b}}=\int\frac{\sec \theta\;
			d\theta}{b(1-2\tan ^2 \theta)}=\int\frac{\cos \theta\;
			d\theta}{b(1-3\sin ^2 \theta)}.
	\end{equation}
	Now, by using the change $\rho = \sqrt{3}\sin \theta$ we obtain
	\begin{equation}\label{Lemma-Calculus-Integral-b}
		\int\frac{\cos \theta\; d\theta}{b(1-3\sin ^2
			\theta)}=\int\frac{d\rho}{\sqrt{3}b(1-\rho^2)}=\frac{1}{2b\sqrt{3}}
		\ln\left(\frac{1+\rho}{1-\rho}\right).
	\end{equation}
	Combining (\ref{Lemma-Calculus-Integral-a}) and (\ref{Lemma-Calculus-Integral-b}) the result follows from the Fundamental Theorem of Calculus. 
	
	{\color{black}\section{Proof of \eqref{LinfinitoPertubationEstimateOK}}\label{0b}
	Having in mind that $\rho_c=|\phi_c+\zf|^2-|\phi_c|^2=2\re(\phi_c\bar{\zf})+|\zf|^2$, it turns out that
	\[
	 |\phi_c+\zf|^2= |\phi_c|^2 + \rho_c,
	\]
\noindent
and therefore, we have 
\be\label{num1}
 \|\zf\|_{L^\infty}\lesssim(1+\|\rho_c\|_{L^\infty}^{1/2})\lesssim(1+\|\rho_c\|_{L^\infty}).
\ee
\noindent
On the other hand, 
\[\begin{aligned}
   \|\rho_c\|_{L^\infty}&\lesssim\|\rho_c\|_{L^2}^{1/2}\|\rho_c'\|_{L^2}^{1/2}\\
   &\lesssim\|\rho_c\|_{L^2}^{1/2}(\|\zf\|_{L^\infty} + \|\zf'\|_{L^2} + \|\zf\|_{L^\infty}\|\zf'\|_{L^2})^{1/2}.
  \end{aligned}
\]
\noindent
 Hence, using Lemma \ref{lemma-integral-black-dark-normH0}, we get
 
 \be\label{num2}
  \|\rho_c\|_{L^\infty}\lesssim\|\rho_c\|_{L^2}^{1/2}(\|\zf\|_{L^\infty}^{1/2} + \|\zf\|_{\mathcal{H}_0}^{1/2} + \|\zf\|_{L^\infty}^{1/2}\|\zf\|_{\mathcal{H}_0}^{1/2}).
 \ee
\noindent
Now, substituting \eqref{num2} into \eqref{num1} and using Young's inequality, we obtain

\be\label{LinfinitoPertubationEstimate1}\begin{aligned}
   \|\zf\|_{L^\infty}&\lesssim(1+\|\rho_c\|_{L^2} + \|\rho_c\|_{L^2}^{1/2}\|\zf\|_{\mathcal{H}_0}^{1/2}+\|\rho_c\|_{L^2}\|\zf\|_{\mathcal{H}_0})\\
  & \lesssim(1+\|\rho_c\|_{L^2})(1+\|\zf\|_{\mathcal{H}_0}).
  \end{aligned}
\ee
}
	
	\section{Computation of some \texorpdfstring{$L^2$}{Lg1} and \texorpdfstring{$L^\infty$}{Lg2} norms}\label{0}
	
	We collect some $L^2$ and $L^\infty$ norms needed along this work, in the following sections. 
	Hereafter, we will consider by $K$ the smallest of the constants that allow us to get the corresponding upper bound.

	\subsection{\texorpdfstring{$L^2$}{Lg1a} norms}\label{0L2}
	
	We first compute the associated $\mathcal{H}_0$ norm in the distance $d_0$  \eqref{d0}. By  definition,
	\[\|\phi_0 - \phi_c\|^2_{\mathcal{H}_0}= \|\phi_0' - \phi_c'\|^2_{L^2} + \|\sqrt{\eta_0}(\phi_0 - \phi_c)\|^2_{L^2},\]
	\noindent
	therefore we split the computation in two steps:
	\noindent
	first we consider (with $R_c$ in \eqref{real-dark}) \newline
	\noindent
	\[\begin{aligned}&\|\phi_0' - \phi_c'\|^2_{L^2}=\int_\R(\phi_0' - \phi_c')(\phi_0' - \bar{\phi}_c') = 
		\int_\R((\phi_0')^2 + |\phi_c'|^2 -2R_c'\phi_0').\\
	\end{aligned}\]
	
	\medskip
	\noindent
	Then, expanding in $c$ the last integrand, we note that this $L^2$ norm is bounded above, at small speeds $|c|\le \cf,$ with $\cf\ll1,$ by
	
	\be\label{part1H0}
	\|\phi_0' - \phi_c'\|^2_{L^2}\leq K\int_\R \Big(-\frac{9\left(\tanh ^2(x) \sech^4(x)\right)}{8 \left(\tanh ^2(x)-3\right)^3}c^2 \Big)dx
	=  \frac{K}{32}\Big(12 - 5\sqrt{3}\log(2+\sqrt{3})\Big)c^2.
	\ee
	
	\medskip
	\noindent
	On the other hand, we consider
	\noindent
	\[\begin{aligned}&\|\sqrt{\eta_0}(\phi_0 - \phi_c)\|^2_{L^2}=\int_\R\eta_0(\phi_0 - \phi_c)(\phi_0 - \bar{\phi}_c) = 
		\int_\R\eta_0((\phi_0)^2 + |\phi_c|^2 -2R_c\phi_0),\\
	\end{aligned}\]
	\noindent
	which again behaves (proceeding as above), at small speeds $|c|\le \cf,$ with $\cf\ll1,$ as
	
	\be\label{part2H0}\leq K\int_\R \Big(\frac{27  \left(\tanh ^4(x)+2 \tanh ^2(x)-3\right)}{8 \left(\tanh ^2(x)-3\right)^3}c^2 \Big)dx
	=  K\frac{3}{32}\Big(12 - \sqrt{3}\log(2-\sqrt{3})\Big)c^2.
	\ee
	
	\noindent
	Finally summing  \eqref{part1H0} and \eqref{part2H0} and simplifying, we get the $\mathcal{H}_0$ norm in \eqref{d0}
	
	\[
	\|\phi_0 - \phi_c\|^2_{\mathcal{H}_0}\leq K\frac{c^2}{16}\Big(24-\sqrt{3}\log(2+\sqrt{3})\Big).
	\]

\medskip 
\noindent
	With respect to \eqref{L2-ddark-raizetac},  we first compute $|\frac{\phi_c'}{\sqrt{\eta_c}}|^2$ as
	
	\[\begin{aligned}
		|\frac{\phi_c'}{\sqrt{\eta_c}}|^2&=\frac{\partial_x\phi_c\partial_x\overline{\phi}_c}{(\sqrt{\eta_c})^2}\\
		&=\frac{-2 \kappa ^2 \sech^4(\kappa  x) \left(\mu_1^2 \mu^2 \tanh ^2(\kappa  x)+\mu_2^2\right)}{\left(1+\mu \tanh ^2(\kappa  x)\right) \left(\mu_1^4+2 \left(\mu_1^2 \mu_2^2-4 \mu\right) \tanh ^2(\kappa  x)+\left(\mu_2^4-4 \mu^2\right) \tanh ^4(\kappa x)-4\right)}.
	\end{aligned}\]
	
	\medskip
	\noindent
	Therefore, integrating and having in mind the constraint relation \eqref{murelation} we have that
	\[\begin{aligned}
		\Big\|\frac{\phi_c'}{\sqrt{\eta_c}} \Big\|_{L^2}^2:&= \int_\R \frac{-2 \kappa ^2 \sech^4(\kappa  x) \left(\mu_1^2 \mu^2 \tanh ^2(\kappa  x)+\mu_2^2\right)dx}{\left(1+\mu \tanh ^2(\kappa  x)\right) \left(\mu_1^4+2 \left(\mu_1^2 \mu_2^2-4 \mu\right) \tanh ^2(\kappa  x)+\left(\mu_2^4-4 \mu^2\right) \tanh ^4(\kappa x)-4\right)}\\
		&=\frac{-4\ka}{(\mu_1^2-2)}\left(\sqrt{|\mu|}\arctanh(\sqrt{|\mu|}) +  \frac{(\mu_2^2 + 2\mu + \mu\mu_1^2)
			\arctan\left(\sqrt{\frac{2\mu +\mu_2^2}{2 + \mu_1^2}}\right)}
		{\sqrt{2+\mu_1^2}\sqrt{2\mu+\mu_2^2}}\right)\\
		&\leq \frac{\pi }{3 \sqrt{3}}+\frac{2 }{\sqrt{3}}\arccotanh\left(\sqrt{3}\right).
	\end{aligned}\]

\medskip
\noindent
With respect to the $L^2$-norms in \eqref{new-maxCoercivity-a} and \eqref{new-maxCoercivity-b}, we get after an expansion in $c, ~~|c|< \cf,$\; with $\cf\ll1,$ in the integrand of \eqref{new-maxCoercivity-a},  that this $L^2$ norm is bounded above by
\be\label{new-maxCoercivity-aAppen}\begin{aligned}
	\Big\|\frac{|\phi_c|^2 - \phi_0^2}{\sqrt{\eta_0}}\Big\|^2_{L^2}
	&\leq K\int_{\R}\frac{3\sech^2(x)\left(\tanh ^2(x)+9\right)^2}{64  \left(\tanh ^2(x)-3\right)^2 \left(\tanh ^2(x)+3\right)}c^4dx\\
	&\leq K\frac{c^4}{192}\left(36 + \sqrt{3}\pi +12\sqrt{3}\log(2+\sqrt{3})\right)\\
	&\leq \frac{K}{4}c^4, \end{aligned}
\ee

\medskip 
\noindent
and therefore we obtain \eqref{new-maxCoercivity-a}. Now in \eqref{new-maxCoercivity-b},  expanding again in $|c|< \cf,$\; with $\cf\ll1,$ we get that 

\be\label{new-maxCoercivity-bAppen}\begin{aligned}
	\Big\|\frac{\phi_0\eta_0-R_c\eta_c}{\sqrt{\eta_0}}\Big\|^2_{L^2}
	&\leq K \int_\R\frac{3}{512} c^4 \frac{(-1+\tanh^4(x))\tanh^2(x)}{(-3+\tanh^2(x))^5(3+\tanh^2(x))}dx\\
	&\leq K\frac{3}{512} c^4 \left(630-8\sqrt{3}\pi -39\sqrt{3} \log \left(\sqrt{3}+2\right)\right)\\
	&\leq\frac{K}{4}c^4. \end{aligned}
\ee

\medskip
\noindent
We now compute the $L^2$ norm in \eqref{new-maxCoercivity-aL2-1}. Firstly we write explicitly the integrand

\[
 \frac{(\eta_c|\phi_c|^2-\eta_0\phi_0^2)^2}{\eta_0}=\frac{\eta_c\frac{\mu_1^2+\mu_2^2 \tanh ^2(\kappa  x)}{2 \left(\mu \tanh ^2(\kappa  x)+1\right)}-\eta_0\frac{2 \tanh ^2(x)}{3-\tanh ^2(x)}}{\left(1-\frac{4 \tanh ^4(x)}{\left(3-\tanh ^2(x)\right)^2}\right)},
\]
\noindent
\medskip
In fact, in the same small speed region $|c|< \cf,$\; with $\cf\ll1,$ we get, after an expansion of the above expression,  that 

\[\begin{aligned}
 \Big\|\frac{\eta_c|\phi_c|^2-\eta_0\phi_0^2}{\sqrt{\eta_0}}\Big\|_{L^2}^2
	&\leq K\int_{\R}\frac{81\sech^4(x)}{8\left(3-\tanh ^2(x)\right)^3}c^4dx=K\frac{9c^4}{32}\sqrt{3}\log\left(\frac{1}{2-\sqrt{3}}\right)
	\\
	&\leq \frac{K}{\sqrt2}c^4. \end{aligned}
\]

\medskip
We now compute the $L^2$ norm in \eqref{new-maxCoercivity-aL2-2} proceeding in the same way. In fact after an expansion of the integrand  in the small speed region $|c|< \cf,$\; with $\cf\ll1,$ we get

\[\begin{aligned}
 \Big\|\frac{\eta_c|\phi_c|^2R_c^2 - \eta_0\phi_0^4}{\sqrt{\eta_0}}\Big\|_{L^2}^2
	&\leq K\int_{\R}\frac{27\sech(x)\tanh^4(x)}{8\left(3-\tanh ^2(x)\right)^3}c^4dx
	\leq \frac{K}{2}c^4. \end{aligned}
\]

\subsection{\texorpdfstring{$L^\infty$}{Lg2a}  norms}\label{0Linfty}

We now compute the  $L^{\infty}$ norm in \eqref{new-maxCoercivity-a00}. Expanding it in the small speed region  $|c|< \cf,$\; with $\cf\ll1,$ we get 
\[
\frac{|\phi_c|^2 - \phi_0^2}{(1+x^2)\eta_c} \leq \frac{(9 -4x\tanh(x)+\tanh^2(x))}{8(1+x^2)(3+\tanh^2(x))^3}c^2,
\]
\medskip
\noindent
uniformly in $x\in\R,$ and whose maximum value ($\frac38$) is attained at $x=0$. Therefore we get that

\be\label{new-maxCoercivity-a000}
\begin{aligned}
	\Big\| \frac{|\phi_c|^2 - \phi_0^2}{(1+x^2)\eta_c} \Big\|_{L^\infty}
	\leq \frac38c^2.
\end{aligned}
\ee

\medskip 

Now, we justify the uniform pointwise estimate in \eqref{black-dark-comparison}. First we note that for any given $x\in \R$ we have 
\be\label{black-dark-comparison-proof1}
\frac{\sqrt{3}}{2}|x|=\kappa(1)|x| \leq \kappa(c)|x|,\; \forall \; |c|\leq1,
\ee
with $\kappa$ defined in \eqref{kappa}. Now observe that 
\be\label{black-dark-comparison-proof2}
\lim\limits_{x\to 0^{\pm}}\frac{|\phi_0(x)|^2}{\big|\phi_0(\sqrt{3}x/2)\big|^2}=\frac{4}{3} \quad\text{and} \quad \lim\limits_{x\to \pm\infty}\frac{|\phi_0(x)|^2}{\big|\phi_0(\sqrt{3}x/2)\big|^2}=1, 
\ee
from which we can conclude that
\be\label{black-dark-comparison-proof3}
|\phi_0(x)|^2 \lesssim \frac{\tanh^2(\sqrt{3}|x|/2)}{3-\tanh^2(\sqrt{3}|x|/2)},\quad \forall \, x\in \R.
\ee
Now, selecting $s:=\tanh(\sqrt{3}|x|/2)$ and using that the function $s \mapsto \frac{s^2}{3-s^2}$ is increasing on the interval $[0,1]$ we have, combining \eqref{black-dark-comparison-proof1} and \eqref{black-dark-comparison-proof3}, that 
\be\label{black-dark-comparison-proof4}
|\phi_0(x)|^2 \lesssim \frac{\tanh^2(\kappa(c)|x|)}{3-\tanh^2(\kappa(c)|x|)},\quad \forall \, (x,c)\in \R\times [-1,1].
\ee
Finally, using that 
\[
\lim_{c\rightarrow0^{\pm}}\mu_2(c)=\pm\frac{2}{\sqrt{3}}\quad\text{and}\quad -\frac{1}{3}\leq\mu(c)\leq0
\]
we conclude, from \eqref{black-dark-comparison-proof4}, that there exists $\cf \ll 1$ such that 
	\be\label{black-dark-comparison-proof5}
	\begin{split}
	|\phi_0(x)|^2 & \lesssim \frac{\mu_1^2(c)+ \mu_2^2(c)\tanh^2(\kappa(c)|x|)}{2+ 2\mu(c)\tanh^2(\kappa(c)|x|)}\\
	& \sim |\phi_c(x)|^2
	\end{split}
	\ee
	for all $(x,c)\in \R\times (-\cf,\cf)$.

\section{Computation of \texorpdfstring{$det\mathcal{F}(c)$}{Lg0} and matrix elements of \texorpdfstring{$\mathcal{M}(c,z)$}{Lg00}}\label{Ap1}

First of all, we remember the expression of $det\mathcal{F}(c)$ \eqref{DetFI}:

\be\label{DetFIAppen}\begin{aligned}
	det\mathcal{F}(c)  = &\int_\R \langle i \eta_{c},-\partial_c\phi_c\rangle_\C\\
	&\times \Big\{\int_\R\langle  \eta_{c},-i\phi_c\rangle_{\C}\int_\R \langle iR_{c} \eta_{c},\partial_x\phi_c\rangle_{\C}
	-
	\int_\R\langle  \eta_{c},\partial_x\phi_c\rangle_{\C}\int_\R \langle iR_{c} \eta_{c},-i\phi_c\rangle_{\C}
	\Big\},
\end{aligned}
\ee

\noindent
for all  $c\in(-2,2)$. Now, we compute the  five different elements in \eqref{DetFIAppen} at $c=0$. 
We start with the first factor in \eqref{DetFIAppen} \newline

\medskip
\be\label{AAdetF}\begin{aligned}
	\int_\R\langle  i\eta_{c}&,-\partial_c\phi_c\rangle_{\C} = -\int_\R\re\Big(i\eta_{c}\partial_c\bar{\phi}_c\Big),
\end{aligned}
\ee
\noindent
and at $c=0$ we get 

\[
-\re\Big(i\eta_{c}\partial_c\bar{\phi}_c\Big)\Big|_{c=0}=-\frac{9 \left(\tanh ^4(x)+2 \tanh ^2(x)-3\right)}{2 \sqrt{2} \sqrt{3-\tanh ^2(x)} \left(\tanh ^2(x)-3\right)^2}.
\]
\noindent
Now, integrating the above expression, we obtain

\be\label{AAdetF2}
\begin{aligned}
	\int_\R\langle  i\eta_{c}&,-\partial_c\phi_c\rangle_{\C}\Big|_{c={\color{black}0}} = -2.
\end{aligned}
\ee

\bigskip
The second factor is
\be\label{CCdetF}
\begin{aligned}
	&\int_\R\langle  iR_c\eta_{c},\partial_x\phi_c\rangle_{\C} = \int_\R\re\Big(iR_c\eta_{c}\partial_x\bar{\phi}_c\Big),
\end{aligned}
\ee
\medskip
\noindent
and at $c=0$ we have that
\[
\re\Big(iR_c\eta_{c}\partial_x\bar{\phi}_c\Big)\Big|_{c={\color{black}0}}=0,
\]
\noindent
and we get that
\be\label{CCdetF2}
\begin{aligned}
	&\int_\R\langle  iR_c\eta_{c},\partial_x\phi_c\rangle_{\C}\Big|_{c={\color{black}0}} = 0.
\end{aligned}
\ee

\bigskip
The corresponding third factor is

\be\label{DDdetF}
\int_\R\langle  \eta_{c},-i\phi_c\rangle_{\C}= \int_\R\re\Big(i\eta_{c}\bar{\phi}_c\Big),
\ee
\medskip
\noindent
with
\[
\re\Big(i\eta_{c}\bar{\phi}_c\Big)\Big|_{c={\color{black}0}}=0.
\]
Therefore, as above we have,
\be\label{DDdetF2}
\int_\R\langle  \eta_{c},-i\phi_c\rangle_{\C}\Big|_{c={\color{black}0}}=0.
\ee
\medskip
\noindent

\bigskip
The fourth factor is

\be\label{BBdetF}
\begin{aligned}
	\int_\R\langle  \eta_{c},\partial_x\phi_c\rangle_{\C} = \int_\R\re\Big(\eta_{c}\partial_x\bar{\phi}_c\Big),
\end{aligned}
\ee
\medskip
\noindent
with
\[
\re\Big(\eta_{c}\partial_x\bar{\phi}_c\Big)\Big|_{c={\color{black}0}}=\frac{9 \sqrt{2} \left(\tanh ^4(x)+2 \tanh ^2(x)-3\right) \text{sech}^2(x)}{\sqrt{3-\tanh ^2(x)} \left(\tanh ^2(x)-3\right)^3}.
\]
\medskip
\noindent 
Therefore integrating, we get

\be\label{BBdetF2}
\begin{aligned}
	\int_\R\langle  \eta_{c},\partial_x\phi_c\rangle_{\C}\Big|_{c={\color{black}0}} = {\color{black}\frac85.}
\end{aligned}
\ee
\noindent

\bigskip
Finally the last factor is

\be\label{EEdetF}
\begin{aligned}
	&\int_\R\langle  iR_c\eta_{c},-i\phi_c\rangle_{\C}= \int_\R\re\Big(-R_c\eta_{c}\bar{\phi}_c\Big),
\end{aligned}
\ee 
\medskip
\noindent
with
\[
\re\Big(-R_c\eta_{c}\bar{\phi}_c\Big)\Big|_{c={\color{black}0}} = -\frac{6 \left(\tanh ^2(x) \left(\tanh ^4(x)+2 \tanh ^2(x)-3\right)\right)}{\left(\tanh ^2(x)-3\right)^3}.
\]
\noindent 
Integrating, we get 
\be\label{EEdetF2}
\begin{aligned}
	&\int_\R\langle  iR_c\eta_{c},-i\phi_c\rangle_{\C}\Big|_{c={\color{black}0}}= -\frac{1}{2} \sqrt{3} \log \left(\sqrt{3}+2\right).
\end{aligned}
\ee 
\medskip
\noindent
Finally, gathering the five terms above, we have that

\be\label{expandDetF0}
det\mathcal{F}({\color{black}0}) = -\frac{8\sqrt{3}}{5}  \log \left(\sqrt{3}+2\right)={\color{black}-\frac85 E_2[\phi_0].}
\ee

\medskip
\noindent
Now, using a classical continuity argument, we get that, for $c\in[0,\cf),$ with smaller $\cf\lll1,$ if necessary,

\be\label{expandDetFc}
det\mathcal{F}(c) \neq 0. 
\ee
\noindent

\bigskip
\noindent
We now list here the computed matrix elements of $\mathcal{M}(c,z)$ in \eqref{matrixM2}. By parity reasons some terms vanish. Namely

\be\label{prodinternos}
m_{11}=\int_\R\langle \eta_{c(t)}, \partial_x\phi_{c(t)}\rangle_\C=\eqref{BBdetF},
\ee

\be\label{prodinternos2}
m_{12}=\int_\R\langle \eta_{c(t)}, \partial_c\phi_{c(t)}\rangle_\C=0,
\quad\text{and}\quad m_{13}=\int_\R\langle  \eta_{c(t)}, -i\phi_{c(t)}\rangle_\C=\eqref{DDdetF}.
\ee
\medskip
\noindent

Now, we list the products coming from the second orthogonality condition in \eqref{genortogonalityDark}:

\be\label{prodinternos2COM}
m_{21}=\int_\R\langle i\eta_{c(t)}, \partial_x\phi_{c(t)}\rangle_\C=0,
\ee

\be\label{prodinternos2COMbis}
m_{22}=\int_\R\langle i\eta_{c(t)}, \partial_c\phi_{c(t)}\rangle_\C=- \eqref{AAdetF},
\quad\text{and}\quad m_{23}=\int_\R\langle  i\eta_{c(t)}, -i\phi_{c(t)}\rangle_\C= 0,
\ee

\noindent
and finally, the products coming from the third orthogonality relation of \eqref{genortogonalityDark}:

\be\label{prodinternos4}
\begin{aligned}
	m_{31}=\int_\R\langle iR_{c(t)}\eta_{c(t)}, \partial_x\phi_{c(t)}\rangle_\C= \eqref{CCdetF},
\end{aligned}
\ee

\be\label{prodinternos5}
m_{32}=\int_\R\langle iR_{c(t)}\eta_{c(t)}, \partial_c\phi_{c(t)}\rangle_\C=0,\quad
\text{and}\quad m_{33}=\int_\R\langle iR_{c(t)}\eta_{c(t)}, -i\phi_{c(t)}\rangle_\C= \eqref{EEdetF}.
\ee

\noindent
In the limit case when $(c={\color{black}0},z=0)$, the matrix \eqref{matrixM2} has the following simple expression

\be\label{M00}
\mathcal{M}(0,0):=
\begin{pmatrix}
\frac85 & 0 & 0\\
0 & 2 & 0\\
0 & 0 & -\frac12E_2[\phi_0]
\end{pmatrix},
\ee
\noindent
with 
\[\det\mathcal{M}({\color{black}0},0)=-\frac{8}{5}E_2[\phi_0].\]

\end{document}